\documentclass[12pt,epsfig,amsfonts]{amsart} 
\usepackage{amsmath,amsthm,amssymb,amscd,epsfig}
\usepackage{graphicx}

\newtheorem*{theorema}{Theorem}
\newtheorem{prop}{Proposition}[section]
\newtheorem{definition}{Definition}[section]
\newtheorem{lemma}{Lemma}[section]
\newtheorem{sublemma}{Sublemma}[section]
\newtheorem{remark}{Remark}[section]
\newtheorem{cor}{Corollary}[section]
\newtheorem{claim}{Claim}[section]

\begin{document}

\author{Hiroki Takahasi *}

\
\title{Prevalent dynamics at the first bifurcation of
H\'enon-like families}
\thanks{* Department of Mathematics, Faculty of Science,
Kyoto University, Kyoto 606-8502, JAPAN\\
(e-mail address: takahasi@math.kyoto-u.ac.jp)}

\begin{abstract}
We study the dynamics of strongly dissipative H\'enon-like maps,
around the first bifurcation parameter $a^*$ at which the uniform
hyperbolicity is destroyed by the formation of tangencies inside
the limit set. We prove that $a^*$ is a full Lebesgue density
point of the set of parameters for which Lebesgue almost every
initial point diverges to infinity under positive iteration.
A key ingredient is that $a^*$
corresponds to ``non-recurrence of every critical point'',
reminiscent of Misiurewicz parameters in one-dimensional dynamics.
Adapting on the one hand Benedicks $\&$ Carleson's parameter
exclusion argument, we construct a set of ``good parameters''
having $a^*$ as a full density point. Adapting Benedicks $\&$
Viana's volume control argument on the other, we analyze Lebesgue
typical dynamics corresponding to these good parameters.
\end{abstract}
\maketitle

\section{Introduction}


One important problem in dynamics is to describe transitions from
structurally stable to unstable regimes. Equally important is to
describe how strange attractors are created. A prototypical
example intimately connected to these problems is given by the
H\'enon family
$$H_{a}\colon(x,y)\mapsto (1-ax^2+\sqrt{b}y, \pm\sqrt{b}x),\ \ 0<b\ll1.$$
For all large $a$, one gets a uniformly hyperbolic horseshoe
\cite{DN}, a paradigmatic example of structurally stable chaotic
systems. As one decreases $a$, the horseshoe loses its stability
at a bifurcation parameter, and then a nonuniformly hyperbolic
strange attractor is created, with positive probability in
parameter space \cite{BC91}. The aim of this paper is to shed some
light on the process of this sort of transition from horseshoes to
strange attractors.

We work within a framework set up by Palis for studying
bifurcations of diffeomorphisms:
 consider arcs of diffeomorphisms losing their
 stability through generic bifurcations, and analyze which dynamical
phenomena are more frequently displayed (in the sense of Lebesgue
measure in parameter space) in the sequel of the bifurcation. More
precisely, let $(\varphi_a)$ be a parametrized family of
diffeomorphisms which undergoes a first bifurcation at $a=a^*$,
i.e., $\varphi_a$ is structurally stable for $a>a^*$ and
$\varphi_{a^*}$ has a cycle. We assume $(\varphi_a)$ unfolds the
bifurcation generically. A dynamical phenomenon $\mathcal P$ is
{\it prevalent} at $a^*$ if
$$\liminf_{\varepsilon\to+0} \varepsilon^{-1}{\rm Leb}
(\{a\in[a^*-\varepsilon,a^*]\colon \varphi_{a} \text{ displays
$\mathcal P$}\})>0.$$

This framework originates in the work of Newhouse and Palis
\cite{NP}, on the frequency of bifurcation sets in the unfoldings
of homoclinic tangencies. In that paper, diffeomorphisms before
the first bifurcation are Morse-Smale. Palis and Takens
\cite{PT0,PT1,PT2}, inspired by works of Newhouse, studied the
prevalence of uniform hyperbolicity in arcs of diffeomorphisms for
which the non-wandering set of the diffeomorphism  at the
bifurcation is a union of a non-trivial basic set of saddle type
and an orbit of tangency. In opposite direction, the frequency of
non-hyperbolicity was studied by Palis and Yoccoz \cite{PY1,PY2,PY3}.

For the H\'enon family, the first bifurcation where the horseshoe
ceases to be stable corresponds to the formation of homoclinic or
heteroclinic tangencies \cite{BS1}. This tangency is quadratic,
and $(H_a)_a$
unfolds the tangency generically \cite{BS2}. The orbit of the
tangency is accumulated by transverse homoclinic points, and hence
contained in the limit set. In \cite{CLR}, all these statements
are extended
 to H\'enon-like families, a perturbation of the H\'enon
family (see Section 2 for a precise definition).

This sort of bifurcation is completely different from the one
treated in \cite{PT0,PT1,PT2,PY1,PY2,PY3}. A key aspect of models
treated in these papers is that the orbit of tangency at the first
bifurcation is not contained in the limit set. This implies a
global control on new orbits added to the underlying basic set,
and moreover allows one to use its invariant foliations to
translate dynamical problems to the problem on how two Cantor sets
intersect each other. This argument is not viable, if the orbit of
tangency, responsible for the loss of the stability of the system,
is contained in the limit set, as in the case of H\'enon-like
families. Let us call such a bifurcation an {\it internal tangency
bifurcation}.

For an H\'enon-like family $(f_{a})$, we aim to describe changes
in the set $$K_{a}=\left\{ z\in\mathbb R^2\colon
\{f_{a}^nz\}_{n\in\mathbb Z} \text{ is bounded}\right\}.$$ By a result of
\cite{CLR}, there is a parameter $a^*$ such that $K_a$
is a hyperbolic set for $a>a^*$, and $(f_a)_a$ unfolds a quadratic
tangency at $a=a^*$ generically. This suggests that the structure
of $K_a$ depends in a very discontinuous way upon $a$. For
instance, $a^*$ is accumulated from left by:
$a$-intervals for which $f_{a}$ has sinks \cite{AY83,GS};
sets with positive Lebesgue measure for which $f_{a}$
has nonuniformly hyperbolic attractors \cite{MV93}, etc. A
consequence of our theorem is that the frequency of such
parameters tends to zero as $a\to a^*$. Let
$$K_{a}^+=\left\{ z\in\mathbb R^2\colon \{f_{a}^nz\}_{n\geq0}
\text{ is bounded}\right\}.$$

\begin{figure}
\begin{center}
\input{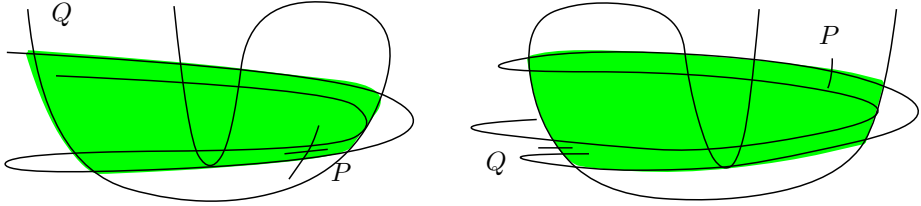}
\caption{Manifold organization for $a=a^*$. There exist two hyperbolic fixed saddles $P$, $Q$
near $(1/2,0)$, $(-1,0)$ correspondingly.
In the orientation preserving case (left),
$W^u(Q)$ meets $W^s(Q)$ tangencially.
In the orientation reversing case (right), $W^u(P)$ meets $W^s(Q)$
tangencially. The shaded regions represent $R_0$ (see Sect.2.1).}
\label{tangency}
\end{center}
\end{figure}

\begin{theorema}
For an H\'enon-like family $(f_{a})$
there exists
a set $\Delta$ of $a$-values such that:
\smallskip

{\rm (a)} $\lim_{\varepsilon\to+0}\varepsilon^{-1}{{\rm
Leb}(\Delta\cap[a^*-\varepsilon,a^*])}=1$;
\smallskip

{\rm (b)} if $a\in \Delta$, then $K_{a}^+$ has zero Lebesgue
measure.

{\rm (c)} if $a\in\Delta$, then $f_a$ is transitive on $K_a$.
\end{theorema}

To grasp the meanings of the theorem, it is worthwhile to recall
Jakobson's theorem \cite{Jak81} for the quadratic family $x\to
1-ax^2$, which states that $a=2$ is a (one-sided) full Lebesgue
density point of the set of parameters corresponding to absolutely
continuous invariant probability measures. These measures allow
one to statistically predict the asymptotic ``fate" of Lebesgue
almost every initial conditions. For $a>2$, the orbit of the
critical point $x=0$ is escaping, and thus the invariant set is
uniformly hyperbolic. In other words, $a=2$ is a first bifurcation
parameter of the quadratic family. Immediately right after the
bifurcation one mainly gets ``observable chaos". Our theorem
asserts a sharp contrast to this sort of transition. For
$a\in\Delta$, $K_a$ behaves like a basic set of saddle type, in
that Lebesgue typical points escape from any neighborhood of it.
This means that, physically observable complicated behaviors are
chaotic transient around $K_a$, not
sustained in time. 

This striking difference at the first bifurcation stems from a
simple fact intrinsic to two-dimension: at the parameter $a^*$,
the unstable manifold of the saddle fixed point(s) is not confined
in any bounded region. Indeed, one key step in the proof of the
theorem is to show that, for carefully chosen parameters, the
unstable manifold intersects $K_a^+$ in a set with zero Lebesgue
measure 
on the manifold.

By the continuous dependence of invariant manifolds on parameter,
one can take a parameter $a'<a^*$ such that $W^u(P)$ is unbounded
for $a'<a<a^*$. Let $a^{**}$ denote the smallest with this
property. Our parameter set $\Delta$ is contained in
$(a^{**},a^*]$. Benedicks and Carleson \cite{BC91}, Mora and Viana
\cite{MV93} constructed a set of $a$-values near $2$,
corresponding to maps for which the closure of $W^u(P)$ is a
nonuniformly hyperbolic strange attractor. Their parameter sets
are at the left of $a^{**}$. Figure 2 indicates a landscape in the
$(a,b)$-plane (as usual, $b$ controls the closeness to the
quadratic family, see (\ref{form0})). In the orientation
preserving case, $a^{**}$ corresponds to the tangency between
$W^u(P)$ and $W^s(Q)$.


\begin{figure}[t]\label{fig1}
\begin{center}
\input{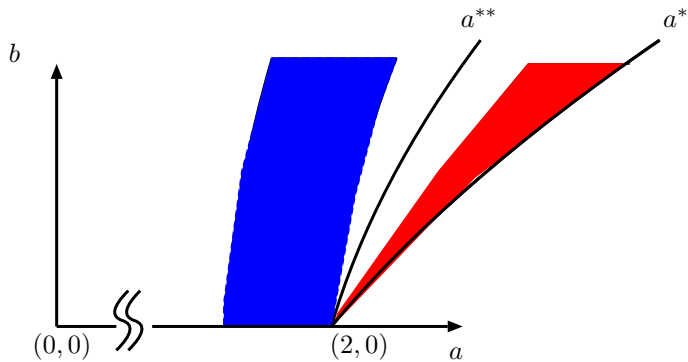}
\end{center}
\caption{Landscape in the $(a,b)$-plane near $(2,0)$. Parameter sets 
constructed in \cite{BC91,MV93,WY01} corresponding to
nonuniformly hyperbolic strange attractors are contained in 
the blue region. Our parameter set is contained in the red region.
The dynamics are uniformly hyperbolic at the right of the $a^*$-line.}
\end{figure}


In view of the theorem, one might speculate that maps in
$\{f_a\colon a\in\Delta\}$ would retain some weak form of
hyperbolicity, as a memory of the uniform hyperbolicity before the
bifurcation. For the moment, we do not know if the uniform
hyperbolicity is prevalent at $a^*$. To our knowledge, the only presently
known result on the prevalence of hyperbolicity in internal
tangency bifurcations is due to Rios \cite{R01}, on arcs of
surface diffeomorphisms destroying type 3 horseshoes
 (horseshoes with three symbols \cite{PT1}).

To prove the theorem, we build on and develop the machinery for
the analysis of strongly dissipative H\'enon maps
\cite{BC91,BV01,BY93,MV93,WY01}. Excluding undesirable parameters
inductively, we construct the parameter set $\Delta$ having $a^*$
as a full density point. We then investigate the dynamics of
$f\in\{f_a\colon a\in\Delta\}.$

A parameter exclusion argument in the spirit of Jakobson
\cite{Jak81}, Benedicks and Carleson \cite{BC85,BC91} was first
brought into the study of homoclinic bifurcations by Palis and
Yoccoz \cite{PY2,PY3}. As we mentioned in the beginning, the
underlying basic set at the bifurcation is used in a crucial way
there, and the same approach does not work in our context of
internal tangency bifurcation. In order to prove that $K_a^+$ has
zero Lebesgue measure, we develop the volume control argument of
Benedicks and Viana \cite{BV01}.

The rest of this paper consists of six sections and one appendix.
In Section 2 we analyze one fixed map, collecting
results from \cite{BC91,BV01,MV93,WY01} and \cite{T1} as far as we
need them.
In Section 3 we recall the procedure in \cite{T1} for finding suitable
critical approximations, used as guides for orbits falling in critical regions.

The parameter set $\Delta$ is constructed in Section 4. This part
closely follows the previous construction of the parameter set in
\cite{T1}, modulo the assertion that $a^*$ is a full density point
of $\Delta$. It is at this point where the characteristic of the
first bifurcation is crucial. We show that the
map $f_{a^*}$ behaves as if it is a ``two-dimensional Misiurewicz
map", in the sense that {\it every} critical approximation of it
 is non-recurrent. Then it is possible,
as in the one-dimensional case
\cite{BC85,Jak81}, to arrange the induction construction in such
a way that less and less proportions of parameters in
$[a^*-\varepsilon,a^*]$ are excluded as $\varepsilon\to +0$, and
the total fractions of $\Delta$ in the intervals get closer to
one. Consequently, $\Delta$ must have $a^*$ as a full density
point.

For the remaining three sections we consider the dynamics of one
fixed map $f\in\{f_a\colon a\in\Delta\}.$ In Section 5 we identify
an well-organized geometric structure of the unstable manifold,
close to the one identified by Wang and Young \cite{WY01}. Using
this structure, in Section 6 we analyze the dynamics on the
unstable manifold. Combining a classical large deviation argument
\cite{BC91,BY93} with a continuity argument from the first
bifurcation, we prove that $K^+$ intersects the unstable manifold
in a set with zero Lebesgue measure. In Section 7 we study the
dynamics on $K^+$. A careful adaptation of the volume control
argument \cite{BV01} together with the conclusion of Section 6
shows that $K^+$ cannot have positive two-dimensional Lebesgue
measure.






\section{Preliminaries}
In this section we analyze one fixed map $f$, collecting
results from \cite{BC91,BV01,MV93,WY01} and \cite{T1} as far as we
need them.

\subsection{H\'enon-like families}
We deal with a parameterized family 
$(f_a)$ of
diffeomorphisms on $\mathbb R^2$ such that $f=f_a$ has the form
\begin{equation}\label{form0}f_a\colon (x,y)\mapsto
(1-ax^2,0)+b\cdot \Phi(a,b,x,y),\end{equation}
where $(a,b)$ is close to $(2,0)$ and $\Phi$ is bounded, continuous,
$C^4$ in $(a,x,y)$.

Although $f$ is globally defined on $\mathbb R^2$, it is possible
to localize our consideration to a compact domain defined as
follows. If $f$ preserves orientation, let $W^u=W^u(Q)$.
Otherwise, let $W^u=W^u(P)$.
 Let $R_0$ denote the compact domain bounded by
$W^u$ and $W^s(Q)$, as indicated in Figure 1 in the case $a=a^*$.
By a result of
\cite{CLR}, points outside of $R_0$ escape to infinity either by
positive or negative iterations. Hence $K\subset R_0$ holds. Let
$D_0=\{(x,y)\notin R_0\colon x\geq\sqrt{2}\}$. It can be read out
from \cite{CLR} that $K^+\subset D_0\cup R_0$ holds. By the
obvious uniform hyperbolicity on $D_0$,
$K^+\cap D_0$ has zero Lebesgue measure. 
Therefore, for the proof of the theorem, it suffices to show that
$K^+\cap R_0$ has zero Lebesgue measure. To this end, the next
lemma allows us to focus on the dynamics inside $R_0$.

\begin{claim}
$K^+\cap R_0=\bigcap_{n\geq0}f^{-n}R_0$.
\end{claim}

\begin{proof}
Let $z\in K^+\cap R_0$. Suppose that
 $z\notin f^{-n}R_0$ holds for some $n>0.$
Let $n_0$ denote the smallest integer with this property. Then
$f^{n_0+1}z\in D_1$, where  $D_1$ is the set of points $(x,y)$
which is at the left of $W^s_{\rm loc}(Q)$ and $|y|\leq\sqrt{b}$.
As $D_1\cap K^+=\emptyset$, $z\notin K^+$ holds, which is a
contradiction. Consequently, $K^+\cap R_0\subset
\bigcap_{n\geq0}f^{-n}R_0$ holds. The reverse inclusion is
obvious.
\end{proof}

To structure the dynamics inside $R_0$, we construct {\it critical
points} and use them as {\it guides}. Unlike the attractor context
\cite{BC91, MV93, WY01}, the construction of critical points has
to take into consideration possible leaks out of $R_0$ under
iteration, and unbounded derivatives at infinity is a bit
problematic. To bypass this problem, we work with a new family
$(\tilde f_{a,b})$ which is obtained by modifying the quadratic
map $x\to1-ax^2$, and $\Phi$ in (\ref{form0}) so
that the following holds:
\medskip

\noindent (M1) $f=\tilde f$ on $R_0$ and $\tilde fD_1\subset
D_1$;

\noindent (M2) if $z\in R_0$ and $\tilde fz\notin R_0$, then for
any $n\geq1$ and a nonzero tangent vector $v$ at ${\tilde f}^nz$
with ${\rm slope}(v)\leq \sqrt{b}$, ${\rm slope}(D\tilde f v)\leq
\sqrt{b}$ and $\Vert D\tilde fv\Vert \geq 2\Vert v\Vert$;

 \noindent (M3) there exists a constant $C_0>0$ such that $\Vert
\partial^i f\Vert \leq C_0$ and $|\det D\tilde f|\leq C_0b$ on $D_1\cup
R_0\cup fR_0$ $(1\leq i\leq 4)$, where $\partial^i$ denotes any
partial derivative in $a,x,y$ of order $i$.


\subsection{Hyperbolic behavior}
Constructive constants are $\alpha,M,\delta,$ chosen in this order.
The $\alpha,\delta$ are small, and $M$ is a large integer.
Having chosen all of them, we choose sufficiently small $b$.
The letter $C$ denotes any generic constants
which depend only on $(f_a)$ restricted to $[-2,2]^2$.

From this point on, let us denote $\tilde f$ by $f$. We start with
basic properties of $f$. For $\delta>0$, define
$I(\delta)=\{(x,y)\in R_0\colon |x|<\delta\}.$ The next lemma
establishes a uniform hyperbolicity outside of $I(\delta)$. Not
only for orbits staying inside $R_0$, the hyperbolicity estimates
hold for orbits which leak out of $R_0$.
\begin{lemma}\label{outside} For any $\lambda_0\in(0,\log2)$ and $\delta>0$, the
following holds for $(a,b)$ close to $(2,0)$. Let $z\in R_0$ be
such that
 $z$, $fz,\cdots,f^{n-1}z\notin I(\delta)$,
and let $v$ be a tangent vector at $z$ with ${\rm
slope}(v)\leq \sqrt{b}$. Then:
\smallskip

\noindent{\rm (a)} ${\rm slope}(Df^nv)\leq \sqrt{b} \text{ and
}\Vert Df^nv\Vert \geq \delta e^{\lambda_0 n}\Vert v\Vert$;
\smallskip

\noindent {\rm (b)} if, in addition, $f^nz\in I(\delta)$, then
$\Vert Df^nv\Vert\geq e^{\lambda_0 n}\Vert v\Vert.$
\end{lemma}

\begin{proof}
If $z,fz,\cdots,f^{n-1}z\in R_0,$ then (a) (b) follow from the
closeness of $f$ to the top quadratic map. Otherwise, the orbit
splits into the part $z,fz,\cdots,f^{k-1}z$ $(k<n)$ in $R_0$, and
the rest out of $R_0$. (b) is vacuous because of $f^nz\notin
I(\delta)$. We have
 ${\rm slope}(Df^kv)\leq \sqrt{b}
\text{ and }\Vert Df^kv\Vert \geq \delta e^{\lambda_0 k}\Vert
v\Vert$. Combining these with (M2) we obtain (a).
\end{proof}

\subsection{Quadratic behavior}
In the next lemma we assume $\gamma$ is a {\it horizontal curve},
that is,  a $C^2$-curve such that the slopes of its tangent
directions are $\leq1/10$ and the curvature is everywhere $\leq
1/10$. For $z\in\gamma$, let $t(z)$ denote any unit vector tangent
to $\gamma$ at $z$. In addition, we assume there exists
$\zeta\in\gamma$ such that ${\rm slope}(Dft(\zeta))\geq
C\sqrt{b}$. Let $e$ denote any unit vector tangent to $f\gamma$ at
$f\zeta$. Split $Dft(z)=A(z)
\left(\begin{smallmatrix}1\\0\end{smallmatrix}\right)+B(z)e$. Let
us agree that $a\approx b$ indicates that $C^{-1}\leq a/b\leq C$
holds for some $C\geq1$.
\begin{lemma}{\rm (\cite{T1} Lemma 2.2.)}\label{quadratic}
For all $z\in\gamma\cap I(\delta),$ $|z-\zeta|\approx |A(z)|$
 and $|B(z)|\leq C\sqrt{b}$.
\end{lemma}

\begin{remark}\label{verti}
{\rm This lemma implies the following, that is often used in
what follows. 
A $C^2$-curve of the form $$\{(x(y),y)\colon |y|\leq
\sqrt{b},|x'(y)|\leq C\sqrt{b}, |x''(y)|\leq C\sqrt{b}\}.$$ is
called  a {\it vertical curve}. Any vertical curve near $f\zeta$
is tangent to $f\gamma$ and the tangency is quadratic, or else it
intersects $f\gamma$ exactly at two points. }
\end{remark}

\subsection{Most contracting directions}\label{contracting}
Some versions of results in this section were obtained in
\cite{BC91,MV93}. Our presentation follows \cite{WY01}. Let $M$ be
a $2\times2$ matrix. Denote by $e$ the unit vector
 (up to sign)
such that $\Vert Me\Vert\leq\Vert Mu\Vert$ holds for any unit
vector $u$. We call $e$, when it exists, the {\it most contracting
direction} of $M$.

For a sequence of matrices $M_1$, $M_2\cdots$, we use $M^{(i)}$ to
denote the matrix product $M_i\cdots M_2M_1$, and $e_i$ to denote
the mostly contracting direction of $M^{(i)}$. \medskip

\noindent{\bf Hypothesis for Sect.2.2.} The matrices $M_i$ satisfy
$|\det M_i|\leq Cb$ and $\Vert M_i\Vert\leq C_0$.
\begin{lemma}\label{fo}$($\rm \cite{WY01} Lemma 2.1.$)$
Let $i\geq 2$, and suppose that $\Vert M^{(i)}\Vert\geq\kappa^i$
and $\Vert M^{(i-1)}\Vert\geq\kappa^{i-1}$ for some $\kappa\geq
b^{1/10}$. Then $e_i$ and $e_{i-1}$ are well-defined, and satisfy
$$\Vert e_i\times
e_{i-1}\Vert\leq\left(\frac{Cb}{\kappa^2}\right)^{i-1}.$$
\end{lemma}

\begin{cor}\label{focor1}{$($\rm \cite{WY01} Corollary 2.1.$)$}
If $\Vert M^{(i)}\Vert\geq\kappa^{i}$ for $1\leq i\leq n$, then:
\smallskip

\noindent{\rm (a)} $\Vert e_n-e_1\Vert\leq \frac{Cb}{\kappa^2}$;

\noindent{\rm (b)} $\Vert M^{(i)}e_n\Vert\leq
\left(\frac{Cb}{\kappa^2}\right)^i$ holds for $1\leq i\leq n$.
\end{cor}

Next we consider for each $i$ a parametrized family of matrices
$M_i(s_1,s_2,s_3)$ such that $\Vert\partial^j \det
M_i(s_1,s_2,s_3)\Vert\leq C_0^ib$, and $|\partial^j
M_i(s_1,s_2,s_3)|\leq C_0^i$ for each $0\leq j\leq 3$. Here,
$\partial^j$ represents any one of the partial derivatives of
order $j$ with respect to $s_1,s_2$, or $s_3$.

\begin{cor}\label{focor2}$($\rm \cite{WY01} Corollary 2.2.$)$
Suppose that $\Vert M^{(i)}(s_1,s_2,s_3)\Vert\geq\kappa^i$ for
$1\leq i\leq n$. Then for $j=1,2,3$ and $2\leq i\leq n$,
\begin{equation}\label{focor21}
|\partial^j(e_{i}\times e_{i-1})|\leq
\left(\frac{Cb}{\kappa^{2+j}}\right)^{i-1},
\end{equation}
\begin{equation}\label{focor21}
\Vert\partial^j (M^{(i)}e_i)\Vert\leq
\left(\frac{Cb}{\kappa^{2+j}}\right)^{i}.
\end{equation}
\end{cor}

Let $e_1(z)$ denote the most contracting direction of $Df(z)$ when
it makes sense. From the form of our map (\ref{form0}), $e_1(z)$
is defined for all $z\notin I(\sqrt{b})$. In view of [\cite{MV93}
pp. 21], we have
\begin{equation}\label{slo}
{\rm slope}(e_1)\geq C/\sqrt{b}\quad\text{and} \quad \Vert
\partial e_1\Vert\leq C\sqrt{b}.\end{equation}

We say $z$ is {\it $\kappa$-expanding up to time $n$}, or simply
{\it expanding}, if there exists a tangent vector $v$ at $z$ and
$\kappa\geq b^{1/10}$ such that for every $1\leq i\leq n$,
\begin{equation*}\label{expand}\Vert Df^iv\Vert\geq \kappa^i\Vert
v\Vert.\end{equation*} With a slight abuse of language, we also
say $v$ is $\kappa$-expanding up to time $n$. For $n\geq1$, let
$e_n(z)$ denote the most contracting direction of $Df^n(z)$ when
it makes sense. From Corollaries \ref{focor1}, \ref{focor2} and
(\ref{slo}) we get
\begin{cor}\label{e1}
If $z$ is $\kappa$-expanding up to time $n$, then ${\rm
slope}(e_n)\geq C/\sqrt{b}\text{ and }\Vert
\partial e_n\Vert\leq \frac{Cb}{\kappa^3}.$
\end{cor}

\subsection{Long stable leaves}\label{long}
 In the next
lemma, a $C^2$-distance $d_{C^2}$ between two vertical curves is
measured by regarding them as $C^2$-functions on
$[-\sqrt{b},\sqrt{b}]$.
\begin{lemma}\label{stable}{\rm(cf.\cite{MV93} Section 6.)}
Let $\kappa\geq \delta^{15}$. If $z$ is $\kappa$-expanding up to time
$n$, then for every $1\leq i\leq n$, the maximal integral curve
of $e_i$ through $z$ contains a vertical curve, denoted by $\Gamma_i(z)$. In
addition, for every $1< i\leq n$,
$d_{C^2}(\Gamma_i(z),\Gamma_{i-1}(z))\leq
\left(\frac{Cb}{\kappa^4}\right)^{i-1}$.
\end{lemma}

By a {\it long stable leaf of order $i$ through $z$} we mean the
curve $\Gamma_i(z)$ as in the statement.

\begin{remark}
{\rm In the construction of long stable leaves,
the relation between
the lengths of leaves
and the value of $\kappa$ is crucial \cite{MV93}.
In \cite{BV01}, long stable leaves of length $\approx 1/5$ are used.
To this end, they require $\kappa\geq e^{-20}$.
For our purpose, long stable leaves of length $\approx 2\sqrt{b}$ suffices.
Hence, $\kappa\geq\delta^{15}$ suffices.}
\end{remark}

\begin{lemma}\label{leaf}{\rm (cf.\cite{BV01} Proposition
2.4.)} Let $\kappa\geq\delta^{15}$. If $z$ is $\kappa$-expanding, then
there exists a vertical curve $\Gamma(z)$
through $z$ such that:

\noindent{\rm (a)} $|f^n\xi-f^n\eta|\leq\left(\frac{Cb}{\kappa}\right)^{n}$ for all $\xi$,
$\eta\in\Gamma(z)$ and $n\geq1$;

\noindent{\rm (b)} if $z_1,z_2$ are $\kappa$-expanding, then
${\rm angle}(t_{\Gamma}(\xi_1),t_{\Gamma}(\xi_2))\leq C\sqrt{b}|\xi_1-\xi_2|,$
where $t_{\Gamma}(\xi_i)$ denotes any unit vector tangent to
$\Gamma(z_i)$ at $\xi_i$, $i=1,2$.
\end{lemma}
We call a {\it long stable leaf through $z$} the curve
$\Gamma(z)$ as in the statement, and a {\it stable leaf}
any compact curve having some iterate contained in a long stable
leaf.

Let us record one consequence of Lemma \ref{leaf}.
Let $\Gamma(z_1)$, $\Gamma(z_2)$ be two long stable leaves
and $\xi_1,\eta_1\in\Gamma(z_1)$. Let $\xi_2$, $\eta_2$
denote the points in $\Gamma(z_2)$ whose $y$-coordinate
coincides with that of $\xi_1$ and $\eta_1$ correspondingly.
Gronwall's inequality gives \begin{equation}\label{leaf2}
|\xi_1-\xi_2|\leq e^{C\sqrt{b}}|\eta_1-\eta_2|.\end{equation}

\subsection{Recovering expansion}\label{recoverysec}
Let $\gamma$ be a horizontal curve and $n\geq M.$ We
say $\zeta\in\gamma$ is a
{\it critical approximation of order
$n$} on $\gamma$ if:

\noindent{\rm (i)} $\Vert Df^i(f\zeta)\Vert\geq 1/10$ for $1\leq
i\leq n$;

\noindent{\rm (ii)} $e_n(f\zeta)$ is tangent to $Dft(\zeta)$,
where $t(\zeta)$ is any unit vector tangent to $\gamma$ at
$\zeta$.
\medskip

\noindent{\it Notation.} For $z\in I(\delta)$ and $i\geq1$,
let $w_i(z)=Df^{i-1}(fz)
\left(\begin{smallmatrix}1\\0\end{smallmatrix}\right)$.
\medskip

We now introduce three conditions, which are taken as inductive
assumptions in the construction of the parameter set $\Delta$. Let
$\lambda:=\lambda_0/2$, where $\lambda_0$ is the one in Lemma
\ref{outside}. A critical approximation $\zeta$ of order $n$ on
$\gamma$ {\it has a good critical behavior} if:
\medskip

\noindent{(G1)} $\|w_{i}(\zeta)\|\geq e^{\lambda (i-1)}$ for
$1\leq i\leq 20n$;

\noindent{(G2)} $\|w_{j}(\zeta)\| \geq e^{-2\alpha
i}\|w_i(\zeta)\|$ for $1\leq  i<j\leq 20n$;

\noindent{(G3)} there exists a monotone increasing function
$\chi\colon[M,20n]\cap\mathbb N \circlearrowleft$ such that for
each $j\in[M,20n]$, $(1-\sqrt\alpha)j\leq \chi(j)\leq j$ and
$\|w_{\chi(j)}(\zeta)\|\geq \delta \|w_{i}(\zeta)\|$ holds for
$0\leq i< \chi(j)$.
\medskip

\noindent
{\bf Hypothesis for the rest of Sect.\ref{recoverysec}:} $\zeta$ is a
critical approximation of order $n$ on $\gamma$, with a
good critical behavior.
\medskip

For $M\leq k\leq 20n-1$, let
$$D_k(\zeta)=e^{-3\alpha k} \cdot\min_{1\leq i\leq k}\min_{i\leq j\leq k+1}
\frac{\|w_{j}(\zeta)\|^2}{\|w_{i}(\zeta)\|^3}.$$
Represent the long stable leaf of order $n$ through $f\zeta$
as a graph of a
function $\Gamma_n(f\zeta)=\{(x_n(y),y)\colon |y|\leq\sqrt{b}\}$.
Let $$V_k=\{(x,y)\colon |x-x_n(y)|\leq
D_k(\zeta)/2,|y|\leq\sqrt{b}\}.$$
Take a monotone increasing function
$\chi$ satisfying
condition (G3).
Let $v$ denote any nonzero vector tangent to
$\gamma$ at $z$. If $fz\in V_k\setminus V_{k+1}$, then we say
$v$ is in {\it admissible position relative to $\zeta$}. Define a
{\it bound period} $p=p(\zeta,z)$ by
\begin{equation*}p=\chi(k),\end{equation*}
and a {\it fold period} $q=q(\zeta,z)$ by
$$q=\min\left\{i\in[1,p)\colon|\zeta-z|
^{\beta} \cdot \|w_{j+1}(\zeta)\|\geq1\text{ for every } i\leq
j<p\right\},$$ where
\begin{equation}\label{theta}
\beta=\frac{2\log C_0}{\log 1/b}.\end{equation} It is easy to
check that $q$ is well-defined, by (G1-3) and the assumption on
$z$. If $fz\in V_{20n-1}$, then we say $v$ is in {\it critical
position relative to $\zeta$}.

\begin{prop}\label{recovery}{\rm (\cite{T1} Proposition 2.2.)}
Let $\gamma$, $\zeta$, $z$, $v$ be as above.

\noindent{\rm (i)} If $v$ is in admissible position relative to
$\zeta$ and $fz\in V_k\setminus V_{k+1}$, then:

{$\text{\rm (a)}$}
 $\log|\zeta-z|^{-\frac{3}{\log C_0}}\leq p\leq\log
|\zeta-z|^{-\frac{3}{\lambda}};$

{$\text{\rm (b)}$} $q\leq C\beta p$;

{$\text{\rm (c)}$} $|f^i\zeta-f^iz|\leq e^{-2\alpha p}$ for $1\leq
i\leq p$;

{$\text{\rm (d)}$} $|\zeta-z|\Vert v\Vert\leq\Vert Df^qv\Vert\leq
| \zeta-z|^{1-\beta}\Vert v\Vert;$

{$\text{\rm (e)}$} $\displaystyle{\Vert Df^pv\Vert\geq\Vert v\Vert
\cdot|\zeta-z|^{-1+\frac{\alpha}{\log C_0}} \geq e^{\frac{\lambda
p}{3}}} \Vert v\Vert;$

{$\text{\rm (f)}$} $\Vert  Df^pv\Vert\geq (\delta/10)\Vert
Df^iv\Vert$ for $0\leq i<p;$


\noindent{\rm (ii)} If $v$ is in critical position relative to
$\zeta$, then $\Vert Df^nv\Vert\leq e^{-8\lambda n}\Vert v\Vert$.
\end{prop}

A proof of this proposition follows the line that is now well
understood \cite{BC91,MV93,WY01}. We split $Dfv$ into the
direction of $\left(
\begin{smallmatrix}1\\0\end{smallmatrix}\right)$
and that of $e_n(fz)$, iterate them separately, and put
them together at the expiration of the fold period.

\section{Existence of binding points}
 To deal with returns to
the region $I(\delta)$, we look for suitable critical
approximations and use them as guides to keep further evolution in
track. Such critical approximations, if exists, are called {\it
binding points}. In this section we recall the procedure in \cite{T1} for
finding binding points.

\subsection{Creation of new critical approximations}\label{cre}
By a {\it $C^2(b)$-curve} we mean a $C^2$-curve such that the slopes
of all its tangent vectors are $\leq \sqrt{b}$ and the curvature
is everywhere $\leq \sqrt{b}$.
The next two lemmas are used to create new critical approximations around the
existing ones. For corresponding versions, see: \cite{BC91}
p.113, Lemma 6.1; \cite{MV93} Sect.7A, 7B; \cite{WY01} Lemma 2.10,
2.11.

\begin{lemma}\label{induce}
Let $\gamma$ be a $C^2(b)$-curve in $I(\delta)$ parameterized by
arc length and such that $\gamma(0)$ is a critical
approximation of order $n$. Suppose that:

\noindent {\rm (i)} $\gamma(s)$ is defined for $s\in[-
b^{\frac{n}{4}}, b^{\frac{n}{4}}]$;

\noindent {\rm (ii)} there exists $m\in[n/3,20n]$ such that $\Vert
Df^i(f\gamma(0))\Vert\geq 1$ for $1\leq i\leq m$.

\noindent There exists $s_0\in[- b^{\frac{n}{4}},
b^{\frac{n}{4}}]$ such that $\gamma(s_0)$ is a critical
approximation of order $m$ on $\gamma$. 
\end{lemma}

Next we consider two $C^2(b)$-curves
 $\gamma_1$, $\gamma_2$ in $I(\delta)$
parametrized by arc length, in a way that the $x$-coordinate of
$\gamma_1(0)$ coincides with that of $\gamma_2(0)$. Let
$t_\sigma(s)$ denote any unit vector tangent to $\gamma_\sigma$ at
$\gamma_\sigma(s)$, $\sigma=1,2$.

\begin{lemma}\label{update}
Let  $\gamma_1$, $\gamma_2$ be as above and suppose that:

\noindent {\rm (i)} $\gamma_1(s)$, $\gamma_2(s)$ are defined for
$s\in[-\varepsilon^{\frac{n}{2}},\varepsilon^{\frac{n}{2}}]$,
$\varepsilon\leq C_0^{-5}$;

\noindent {\rm (ii)} $\gamma_1(0)$ is a critical approximation of
order $n$ on $\gamma_1$ and
 $\Vert Df^i(f\gamma_1(0))\Vert\geq 1$ for
$1\leq i\leq n$;

\noindent {\rm (iii)} $|\gamma_1(0)-\gamma_2(0)|\leq
\varepsilon^n\text{ and } {\rm angle}(t_1(0), t_2(0))\leq
\varepsilon^n$.

{\noindent}There exists
$s_0\in[-\varepsilon^{\frac{n}{2}},\varepsilon^{\frac{n}{2}}]$ such that
$\gamma_2(s_0)$ is a critical approximation of order $n$ on
$\gamma_2$.
\end{lemma}

\subsection{Hyperbolic times}\label{ht}
Let
\begin{equation}\label{MN}\theta=\alpha^3,\quad\kappa_0=C_0^{-10}.
\end{equation}
Let $v$ be a tangent vector at $z$ and let $m\geq1.$
We say $v$ is {\it $r$-regular up to time $m$}
if for $0\leq i<m$,
$$\Vert Df^mv\Vert\geq r\delta\Vert Df^iv\Vert.$$
We say $\mu\in[0,m]$ is an  {\it $m$-hyperbolic time} of $v$ if $
Df^\mu v$ is $\kappa_0^{\frac{1}{2}}$-expanding up to time
$m-\mu$. Results related to the next lemma can be found in
[\cite{BC91} Lemma 6.6], [\cite{MV93} Lemma 9.1], [\cite{WY01}
Claim 5.1].
\begin{lemma}\label{htlem}{\rm ({\rm \cite{T1} Lemma 2.12};
Abundance of well-distributed hyperbolic
 times)}
 Let $m\geq \log(1/\delta)$ and
suppose that a tangent vector $v$ at $z$ is $1/100$-regular
up to time $m$. There exist $s\geq2$ and a sequence $\mu_1<
\mu_2<\cdots< \mu_s$ of $m$-hyperbolic times of $v$ such that:
\smallskip

\noindent{{\rm (a)}} $\Vert Df^{\mu_j}v\Vert$ is
$\kappa_0^\frac{1}{4}$-expanding up to time $m-\mu_j$;

\noindent{\text{\rm (b)}} $1/16\leq (m-\mu_{j+1})/(m-\mu_{j})\leq 1/4$
for $1\leq j\leq s-1$;

\noindent{{\rm (c)}}
 $0\leq \mu_1<m/2$ and
$m-\log(1/\delta)\leq\mu_s\leq m-\log(1/\delta)/2$.
\end{lemma}

\subsection{Nice critical approximations}\label{critical}
Let $\zeta$ be a critical approximation
of order $n$ on a horizontal curve $\gamma$. We say $\zeta$
is {\it nice} if:
\smallskip

\noindent {(C1)} $\Vert Df^i(f\zeta)\Vert\geq 1$ for $1\leq
i\leq n$;

\noindent{(C2)} $f^{-i}\zeta\in [-2,2]\times[-\sqrt{b},\sqrt{b}]$
for $1\leq i\leq[\theta n]$;

\noindent{(C3)} let $u$ denote any unit vector at $f^{-[\theta n
]}\zeta$ such that $Df^{[\theta n]}u$ is tangent to $\gamma$. Then
$u$ is $\kappa_0^{\frac{1}{3}}$-expanding and $1/100$-regular,
both up to time $[\theta n]$.
\medskip

\noindent A nonzero vector $v$ is {\it in tangential position
relative to $\zeta$}
  if there exists a horizontal curve
which is tangent to both $v$ and $Df^{[\theta n]}u$.

\medskip

Let \begin{equation}N=\left[\frac{\log
1/\delta}{\theta}\right],\end{equation} where the square bracket
denotes the integer part. 
\medskip

\noindent{\bf Hypothesis for the rest of Sect.3:} $m$, $n$ are
integers with $m\geq\log(1/\delta)$, $n\geq N$, and:
\medskip

\noindent $\bullet$ each nice critical approximation $\zeta$ of
order $\leq n$ has a good critical behavior;

\noindent $\bullet$ a tangent vector $v$ at $z$ is $1/10$-regular
up to time $m$, and $f^mz\in I(\delta)$.
\medskip

\subsection{Binding procedure}\label{algo}
Under the above hypothesis, we describe how to choose a binding
point relative to which $Df^mv$ is in tangential position. In view
of Lemma \ref{htlem}, fix once and for all a sequence
$\mu_1<\mu_2<\cdots<\mu_{s}$ of $m$-hyperbolic times of $v$
satisfying
\begin{equation}\label{hip1} m-\mu_1\leq\theta n,\quad
\frac{1}{2}\log(1/\delta) \leq m-\mu_s\leq\log(1/\delta),\quad
\frac{1}{16}\leq\frac{m-\mu_{i+1}}{m-\mu_{i}}\ {\rm for}\ 1\leq
i<s.\end{equation} Correspondingly, fix once and for all a
sequence $n\geq n_1>\cdots>n_s>n_{s+1}>\cdots>n_{s_0}:=M$ of
integers such that
\begin{equation}\label{hip2}m-\mu_i=[\theta n_i]\ {\rm for}\ 1\leq i\leq
s,\end{equation}
\begin{equation}\label{hip3}
n_{i+1}=n_{i}-1\quad{\rm for }\ s\leq i< s_0.
\end{equation}

We construct a family of $C^2(b)$-curves tangent to $Df^mv$,
arranged in such an organized way that Lemmas \ref{induce} and
\ref{update} may be used inductively. This produces a critical
approximation on each of the $C^2(b)$-curves. We choose one of
them as a binding point. In this way we obtain the following
statement.

\begin{lemma}\label{zyunzo}$(${\rm \cite{T1} Proposition 3.1.}$)$
There exist $i\in[1,s]$ and a critical approximation $\zeta_i$ of
order $n_i$ such that $Df^mv$ is in tangential position relative
to $\zeta_i$.
\end{lemma}
\noindent{\it Sketch of the proof.} One way to find such $n_i$ and
$\zeta_i$ are described as follows. Let $l_i$ denote the straight
segment of length $\kappa_0^{3\theta n_i}$ centered at
$f^{\mu_i}z$ and tangent to $ Df^{\mu_i}v$. Then $\gamma_{i}
:=f^{\mu_i}l_i$ is a $C^2(b)$-curve extending to both sides around
$f^mz$ to length $\geq \kappa_0^{4\theta n_i}$. Lemma
\ref{induce}, Lemma \ref{update} and the hypothesis of $f$ allow
us to show the following: if $Df^mv$ is in critical position
relative to a critical approximation of order $n_i$ on $\gamma_i$,
then there exists a critical approximation of order $n_{i-1}$ on
$\gamma_{i-1}$ relative to which $Df^mv$ is in tangential
position. A recursive use of this argument yields the conclusion.
\qed

\begin{definition}\label{bind}
{\rm Let $i_0\in[1,s]$ denote the largest integer such that there
exists a critical approximation of order $n_{i_0}$ relative to
which $Df^mv$ is in tangential position. We call any such critical
approximation a {\it binding point for }
$Df^{m}v$.}\end{definition}


Let $\zeta$ denote any binding point for $Df^mv$. By the
definitions in Sect.\ref{recoverysec}, there are two mutually
exclusive cases:
\smallskip

\noindent {{\rm (a)}} $i_0=1$, and $f^mz$ is in critical
position relative to $\zeta$;
\smallskip

 \noindent {{\rm (b)}} $Df^mv$ is in
admissible position relative to $\zeta$.


In case (a), the contraction estimate in (ii) Proposition
\ref{recovery} is in place. In case (b), all the estimates in (i)
Proposition \ref{recovery} are in place: the loss of expansion
and regularity suffered from the return are recovered at the end
of the bound period.

In case (b), one can repeat the binding procedure in the following
manner. Write $m=m_1$. Let $p_1$ denote the bound period. (e,f)
Proposition \ref{recovery} implies that $v$ is $1/10$-regular up
to time $m_1+p_1$. Let $m_2\geq m_1+p_1$ denote the smallest such
that $f^{m_2}z\in I(\delta)$. By Lemma \ref{outside}, $v$ is
$1/10$-regular up to time $m_2$. Subsequently one may repeat the
binding procedure once again, replacing $m\to m_2$, $f^mz\to
f^{m_2}z$, $Df^mv\to Df^{m_2}v$.

In this way, one can (if (a) does not occur) define integers
$$m_1<m_1+p_1\leq m_2<m_2+p_2\leq m_3<\cdots$$
inductively as follows: for $k\geq1$, let $p_k$ be the bound
period of $f^{m_k}z$, and let $m_{k+1}$ be the smallest $j\geq
m_k+p_k$ such that $f^jz\in I(\delta)$. (Note that an orbit may
return to $I(\delta)$ during its bound periods, i.e. $(m_k)$ are
not the only return times to $I(\delta)$.) This decomposes the
orbit of $z$ into segments corresponding to time intervals
$(m_k,m_k+p_k)$ and $[m_k+p_k,m_{k+1}]$, during which we describe
the orbit of $z$ as being``bound" and ``free" states respectively;
$m_k$ are called times of {\it free returns}.

\begin{remark}\label{dig}
{\rm Let us consider the case where the above hypothesis is
satisfied for every $n\geq N$. Then, the binding procedure allows
us to keep in track the evolution of any complete orbit in $W^u$,
decomposing it into bound and free segments. However, this
procedure is not well-adapted to our phase-space construction in
later sections, because:
\medskip

\noindent$\bullet$ the choice of binding points relies only on the
individual orbit under consideration and

neglects a global information on $W^u$;

\noindent$\bullet$ Critical approximations eligible as binding
points are not unique.
\medskip

These issues will be resolved in Section 5, for parameters in
$\Delta$ constructed in the next section.}
\end{remark}

\section{Parameter exclusion}
In this section we construct the parameter set $\Delta$ in the
theorem, having $a^*$ as a full density point. The construction is
done by induction:
$\Delta=\bigcap_{n\geq0}\Delta_{n}$,
where
$\Delta_{n}$ is constructed at step $n$, excluding from $\Delta_{n-1}$
 all those undesirable parameters for which some critical
approximation may not have good critical behavior
 up to time $20n$.

\subsection{Critical approximations of $f_{a^*}$ are non-recurrent}
\label{errant}
The construction of $\Delta$ and a measure estimate of it closely follow
\cite{T1}, in which a positive
measure set of parameters was constructed corresponding to
H\'enon-like maps with nonuniformly hyperbolic behavior.
One key difference from \cite{T1}
is the assertion that $a^*$ is a full density point of
$\Delta$. A key ingredient for this is the next proposition,
which states that the orbit
of every critical approximation of $f_{a^*}$ is non-recurrent.

\begin{prop}\label{good1}For every critical approximation $\zeta$ of
$f_{a^*}$ of order $n$, $f_{a^*}^i\zeta\in \{(x,y)\in\mathbb
R^2\colon |x|\geq9/10\}$ holds for every $1\leq i<20n$.
\end{prop}
We postpone a proof of this proposition to Sect.\ref{proo}.

\subsection{Definition of parameter sets}\label{parameterset}

Choose sufficiently small $\varepsilon_0$ and $b$ so that for any
$f\in\{f_a\colon a\in [a^*-\varepsilon_0,a^*]\}$, any critical
approximation $\zeta$ of $f$ and $1\leq i<20N$,
$f^i\zeta\in\{(x,y)\in\mathbb R^2\colon |x|\geq9/10\}$. This
requirement is feasible by the fact that any critical
approximation is contained in $I(\sqrt{b})$. Set
$\Delta_{n}=[a^*-\varepsilon_0,a^*]$ for $1\leq n\leq N.$
\medskip

Let $n>N$, $a\in \Delta_{n-1}$ and suppose that every nice
critical approximation of $f_a$ of order $<n$ has a
good critical behavior. Let
$20(n-1)\leq m<20n$.
 We say a nice critical approximation $\zeta$ of $f_a$ of order
$\geq n$ satisfies $(G)_m$ if:
\smallskip

\noindent (i) there is an well-defined decomposition of the orbit
$w_1(\zeta),w_2(\zeta),\cdots,w_m(\zeta)$ into bound and free segments,
as described at the end of Sect.\ref{algo};

\noindent (ii) let $n_1-1<n_2-1<\cdots<n_s-1\leq m-1$ denote all the free
return times of $f\zeta$, with $z_{1},\cdots,z_{s}$ the
corresponding binding points. They are of order $<n$ and
\begin{equation}\label{br}
\sum_{i=1}^s\log|f^{n_i}\zeta-z_{i}|\geq -\alpha m.\end{equation}

For $n> N$, define $\Delta_{n}$ to be the set of all $a\in
\Delta_{n-1}$ for which every nice critical approximation of order
$\geq n$ satisfies $(G)_{20n-1}$. In other words,
\[\Delta_{n-1}\setminus \Delta_n=\left\{
\begin{array}{c}
a\in \Delta_{n-1}\colon \text{ $(G)_m$ fails
for some $20(n-1)\leq m<20n$} \\
\quad\quad\quad\quad\text{\quad and some nice critical approximation of
order $\geq n$ of $f_a$}
\end{array}
\right\}.\]

\begin{remark}
{\rm 
Let $\zeta$ be a critical approximation.
If $n-1$ is a free return time of $w_1(\zeta)$,
then for simplicity, let us call $n$ a free return time of $\zeta$. 
This terminology
is in a slight disagreement with that introduced at
the end of Sect.\ref{algo}.}
\end{remark}

The next proposition indicates that, for parameters in $\Delta_n$,
critical approximations of order $n$ can be used as binding
points, and thus allows us to proceed to the definition of
$\Delta_{n+1}$.
\begin{prop}\label{brprop}
Let $n>N$, $a\in \Delta_{n}$ and let $\zeta$ be a nice critical
approximation of order $\geq n$ of $f_a$. Then:

\noindent {\rm (a)} $\|w_i(\zeta)\|\geq e^{\lambda(i-1)}$ for
$1\leq i\leq 20n$;

\noindent {\rm (b)} $\|w_j(\zeta)\|\geq e^{-2\alpha i
}\|w_i(\zeta)\|$ for $1\leq i<j\leq 20n$;

\noindent {\rm (c)} if $\zeta$ is of order $n$, then it has a good
critical behavior.
\end{prop}

\begin{proof}
The proof is almost identical to that of [\cite{T1} Proposition
5.1]. Here we suppose $\zeta$ is of order $n$, and just give a
sketch of how to define the function $\chi$ in $(G3)$. For 
details, see \cite{T1}.


Let $j\in [M,20n]$ and $h_0:=j$. Define a finite sequence
$h_1>\cdots>h_{t(j)}$ of free return times of $\zeta$ inductively
as follows. Let $\hat h_{k+1}$ denote the largest free return time
before $h_k$, when it makes sense. Let $p_{k+1}$ denote the
corresponding bound period. If
\begin{equation}\label{sae}h_k-\hat h_{k+1}-p_{k+1}\leq
(1/\lambda_0)\log (10\delta),\end{equation} then let
$h_{k+1}=\hat h_{k+1}$. In all other cases, $h_{k+1}$ is
undefined, namely $k=t(j)$. Define $\chi(j)=h_{t(j)}.$ Obviously,
$\chi(j)\leq j$ holds.
If $(1-\sqrt{\alpha})j\leq \chi(j)$ did not hold, (\ref{sae})
 would imply that
the total number of bound iterates in the interval
$[(1-\sqrt{\alpha})j,j]$ were bigger than a constant multiple of
$\sqrt{\alpha}j$. While by condition (G), the total number of
bound states in the interval is smaller than a constant multiple
of $\alpha j$. If $\alpha$ is small, then these two estimates are
not compatible. \end{proof}

To estimate the measure of $\Delta_{n-1}\setminus \Delta_{n}$, we
first decompose it into a finite number of subsets, based on
certain combinatorics on itineraries of critical approximations.
We then estimate the measure of each subset separately, and unify
them at the end. In the next two subsections we introduce two
integral components of the combinatorics.

\subsection{Deep returns}
 Let $f\in\{f_a\colon a\in
\Delta_{n-1}\setminus\Delta_n\}$. Let $\zeta$ be a nice critical
approximation of $f$ of order $\geq n$. Let $\nu<20n$ be a free
return time of $\zeta$, with the binding point $z$. If $\nu$ is
not the first return time to $I(\delta)$, then let
$n_1<\cdots<n_{t}$ denote all the free return times of $\zeta$
before $\nu$. For $1\leq s\leq t$, let  $z_{s}$ denote 
the corresponding binding
point and $p_s$ the bound period. 
Write $n_{t+1}=\nu$ and $z_{t+1}=z$. We say $\nu$ is a
{\it deep return time}, if it is the first return time to
$I(\delta)$, or else for $1\leq s\leq t$,
\begin{equation}\label{inessential}
\sum_{j=s+1}^{t+1}2\log|f^{n_j}\zeta-z_{j}|\leq
\log|f^{n_s}\zeta-z_{s}|.
\end{equation}


For each
$n_s$, let
$$\sigma_{n_s}(\zeta)=\frac{|f^{n_s}\zeta-z_s|^{\frac{10}{9}}}
{\|w_{n_s}(\zeta)\|}.$$ For each
$i\in[1,\nu)\setminus\bigcup_{1\leq s\leq t}[n_s,n_s+p_s-1]$, let
$$\sigma_i(\zeta)=\frac{\|w_{i+1}(\zeta)\|}{\|w_i(\zeta)\|^2}.$$
 Define
$$\Theta_{\nu}(\zeta)=\kappa_0\cdot\left
[\sum_{i=1}^{\nu-1}\sigma_i(\zeta)^{-1}\right]^{-1}.$$ It is
understood that the sum runs over all $i$ such that $f^i\zeta$ is
free.

\begin{lemma}{\rm (\cite{T1} Lemma 5.2.)}\label{sample4}
For the above $f,\zeta,\nu,z$, if $\nu$ is a deep return time of
$\zeta$, then $$\|w_{\nu}(\zeta)\||\Theta_\nu(\zeta)|\geq
|f^\nu\zeta-z |^{\frac{1}{2}}.$$
\end{lemma}

\subsection{Position of nice critical approximations}\label{address}

For each $\mu\geq \theta M>1$, fix a subdivision of $\mathbb
R\times\{\sqrt{b}\}$ into right-open horizontals of equal length
$\kappa_0^{\mu}$. We label all of them intersecting $H:=[-2,2]
\times\{\sqrt{b}\}$ with
$l=1,2,3,\cdots$, from the left to the right. By a {\it
$\mu$-address} of a point $x$ on $H$ we mean the integer $l$ which
is a label of the horizontal containing $x$.

In general, let $\zeta$ be a nice critical approximation of order $n$.
The long stable leaf through $f^{-[\theta
n]}\zeta$ of order $[\theta n]$ intersects $H$ exactly at one
point. Let $A(\zeta)$ denote the $[\theta n]$-address of
the point of the intersection.

Let $\zeta$ be a nice critical approximation of order
$n\geq N$ on a horizontal curve $\gamma$. By definition, there
exists a tangent vector $u$ at $f^{-[\theta n]}\zeta$ for which
(C3) in Sect.\ref{critical} holds. Let $\mu$ be any $[\theta
n]$-hyperbolic time of $u.$ We call $\mu$ a {\it hyperbolic time
of $\zeta$.} The long stable leaf through $f^{\mu-[\theta
n]}\zeta$ of order $[\theta n]-\mu$ intersects $H$ exactly at one
point. Let $A(\zeta,\mu)$ denote the $([\theta n]-\mu)$-address of
the point of the intersection.

\subsection{Decomposition of the exclude parameter set
at step $n$}\label{dev} Fix positive integers $m\in[20(n-1),20n)$,
$s$, $t$, $R$. Fix the following combinatorics:
\smallskip

\noindent $\bullet$ sequences
$(\mu_1,\cdots,\mu_s),(x_1,\cdots,x_s)$ of $s$ positive integers;

\noindent $\bullet$ sequences $(\nu_1,\cdots,\nu_t)$, $(
n_1,\cdots,n_t)$, $(r_1,\cdots,r_t)$, $(y_1,\cdots,y_t)$ of $t$
positive integers. \medskip

Let $E_n(*)=E_n(m,s,t,R,\cdots)$ denote the set of all $a
\in\Delta_{n-1}\setminus\Delta_n$ for which there exists a nice
critical approximation $\zeta$ of $f_{a}=f$ of order $n'\geq n$
such that the following holds:

\smallskip

\noindent(Z1) $(G)_{m-1}$ holds, and $(G)_m$ fails;

\noindent (Z2) $\nu_1<\cdots<\nu_t= m$ are all the deep return
times in the first $m$ iterates of $\zeta$, with $z_1,\cdots, z_t$
the corresponding binding points;

\noindent(Z3) for each $k\in[1,t]$, the order of $z_k$ is $n_k<n$.
If $\nu_k<m$, then $|f^{\nu_k}\zeta-z_k|\in
[e^{-r_k},e^{-r_k+1})$. If $\nu_k=m$, which means $k=t$ and
$\nu_t=m$, then $r_t$ is defined as follows. If
$|f^{m}\zeta-z_t|>e^{-\alpha m}$, then $r_t$ is such that
$|f^{m}\zeta-z_t|\in [e^{-r_t},e^{-r_t+1})$ holds. Otherwise,
$r_t=\alpha m$;

\noindent(Z4) $\mu_1<\cdots<\mu_{s}$ is a minimal sequence of
hyperbolic times of $\zeta$ satisfying
\begin{equation}\label{samui}\frac{1}{2}\leq
\frac{[\theta n']-\mu_s}{\log(1/\delta)}\leq1,\quad [\theta
n']-\mu_1\geq\theta n,\quad \frac{1}{16}\leq\frac{[\theta
n']-\mu_{i+1}}{[\theta n']-\mu_{i}}\leq\frac{1}{4}\ {\rm for }\
1\leq i<s.\end{equation}
 Lemma \ref{htlem} ensures the existence of
such a sequence;

\noindent(Z5) $x_i=A(\zeta,\mu_i)$;

\noindent (Z6) $y_k=A(z_k)$.
\medskip

If $a\in E_n(*)$, then any nice critical approximation of $f_{a}$ of
order $\geq n$ for which (Z1-6) hold is called {\it
responsible} for $a$. The parameter set $E_n(*)$ is called an {\it $n$-class}.
By definition, any parameter excluded from $\Delta_{n-1}$ belongs
to some $n$-class. We estimate the measure of
$\Delta_{n-1}\setminus\Delta_n$ by estimating a contribution from
each $n$-class first, and then counting the total number of
$n$-classes.

\subsection{Digestive remarks on the combinatorics}
Let us remark on the meanings of the conditions in the
definition of $E_n(*)$. (Z1,Z2,Z3) are conditions
on the forward orbits of responsible critical approximations. 
(Z2) indicates that {\it we do exclusions
of parameters only at deep return times}. (Z4,Z5) are conditions
on the backward orbits of responsible critical approximations. 
(Z4) indicates that only the backward orbit segments of length 
comparable to $\theta n$ 
are taken into consideration. (Z6) is a condition on binding points at each deep 
return time.
(Z4,Z5,Z6) are used to deal with the following two problems intrinsic
to two-dimension.
\medskip

\noindent{\it $\bullet$ Infinitely many responsible critical
approximations.} The first problem is that critical approximations
responsible for a single parameter are far from unique, and even
infinite. All of them have to be taken into
consideration in the measure estimate of $E_n(*)$. 
(Z4, Z5) are used to deal with this problem. They allow us to
reduce our consideration to a finite number of parameter-dependent
orbits, called {\it deformations}, introduced in Sect.\ref{quas}.
\medskip

\noindent{\it $\bullet$ Infinitely many binding points.} Nice
critical approximations eligible as binding
points are far from unique, due to the very definition
of binding points in Sect.\ref{algo}. (Z6) allows us to deal with
this problem, with the help of deformations as well.
\medskip

\subsection{Full Lebesgue density at the first bifurcation parameter}
\label{sgre} We conclude that $\Delta$ has $a^*$ as a full Lebesgue
density point.
Let $|\cdot|$ denote the one-dimensional Lebesgue measure.
For a compact interval $I$ centered at $x$ and $r>0,$ let $ r\cdot
I$ denote the interval of length $r|I|$ centered at $x$. The main
step is a proof of the next
\begin{prop}\label{kuso}{\rm (Covering by intervals)}
Let $m\in[20(n-1),20n)$, $s$, $t$, $R$ be positive integers. For
any $n$-class $E_n(m,s,t,R,\cdots,)=E_n(*)$, for any
$\varepsilon\in(0,\varepsilon_0)$, $k\in[1,t]$, there exist
a finite number of pairwise disjoint intervals $\{J_{k,i}\}_{i}$
with the following properties:
\smallskip

\noindent{\rm (a)} $E_n(*)\cap[a^*-\varepsilon,a^*]\subset \bigcup
_{i}e^{-r_k/3}\cdot J_{k,i};$

\noindent{\rm (b)} if $t>1$, then for each $k\in[2,t]$ and
$J_{k,i}$ there exists $J_{k-1,j}$ such that ${J}_{k,i}\subset
2e^{-r_{k-1}/3}\cdot{J}_{k-1,j}$;

\noindent{\rm (c)} $\sum_i|J_{1,i}|\leq 3\varepsilon$.\end{prop}

This sort of covering originates in the works Tsujii
\cite{T93a,T93b}, and has been used in \cite{T1} for the
construction of positive measure set of parameters corresponding
to maps with nonuniformly hyperbolic behavior. For our purpose
we need to develop
it further.
\medskip

Proposition \ref{kuso} gives $|E_n(*)\cap
[a^*-\varepsilon,a^*]|\leq 3\varepsilon e^{-\frac{1}{3} R},$ where
$R = r_1 + r_2 \cdots + r_t$. To conclude, we need to count the
number of all feasible $n$-classes. The counting argument in \cite{T1}
shows $$\sharp((\mu_1,x_1),\cdots,(\mu_s,x_s))\leq
C^{-\theta n}$$ and
$$\sharp(\nu_1,\cdots,\nu_t)\sharp(r_1,\cdots,r_t)
\sharp(n_1,\cdots,n_t)\sharp(y_1,\cdots,y_t)\leq e^{\tau(\delta)
n+C\theta\alpha^{-1} R},$$ where  $\tau(\delta)\to0$ as
$\delta\to0$. [\cite{T1} Lemma 5.3] gives $r_1+\cdots+r_t\geq
\alpha m/2.$ Taking contributions from all $n$-classes into
consideration,
\begin{align*}
|(\Delta_{n-1}\setminus \Delta_{n})\cap[a^*-\varepsilon,a^*]|
&\leq\varepsilon\sum_{m,s,t}\sum_{R\geq\alpha m/2}
\sum_{r_1+\cdots+r_t=R} |E_n(*)\cap[a^*-\varepsilon,a^*]|\\&\leq
\varepsilon e^{\tau(\delta)n}\sum_{R\geq\alpha n}
\exp\left(-\frac{R}{6}\right) \leq\varepsilon e^{-\alpha n/8}.
\end{align*}

\begin{figure}[t]
\begin{center}
\input{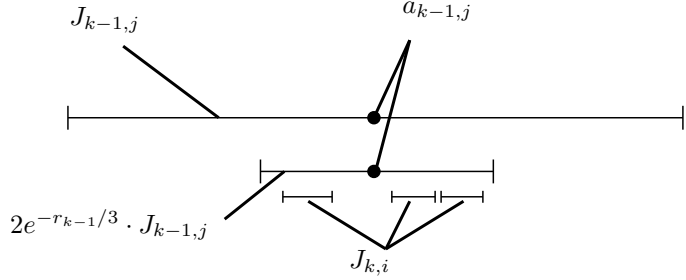}
\caption{Organization of $J_{k,i}$-intervals}
\end{center}
\end{figure}

Let
\begin{equation}\label{n0}
n_0(\varepsilon)=\frac{1}{2\log C_0}\log\left(\frac{
2\varepsilon}{\kappa_0}\right).
\end{equation}
The next lemma indicates that no parameter is deleted from
$[a^*-\varepsilon,a^*]$ at least up to step $[n_0(\varepsilon)/20]$, namely
$[a^*-\varepsilon,a^*]\subset\Delta_n$ holds for every $0\leq
n\leq [n_0(\varepsilon)/20].$

\begin{lemma}\label{good2}
Let $a_0\in[a^*-\varepsilon,a^*]$, and let $\zeta_0$ be a nice critical
approximation of $f_{a_0}$ of order $\xi$. Then $f_{a_0}^{n}\zeta_0\notin
I(\delta)$ holds for every $1\leq
n<\min\left(n_0(\varepsilon),20\xi\right)$.
\end{lemma}

Therefore
\begin{align*}
|\Delta\cap[a^*-\varepsilon,a^*]|&=|\Delta_0\cap[a^*-\varepsilon,a^*]|
-\sum_{n=1}^\infty
|\left(\Delta_{n-1}\setminus\Delta_n\right)\cap[a^*-\varepsilon,a^*]|\\
&=\varepsilon-\sum_{n>[n_0(\varepsilon)/20]}
|\left(\Delta_{n-1}\setminus\Delta_n\right)\cap[a^*-\varepsilon,a^*]|
\geq\varepsilon\left(1- \sum_{n>[n_0(\varepsilon)/20]} e^{-\alpha
n}\right).\end{align*} Since $n_0(\varepsilon)\to\infty$ as
$\varepsilon\to0,$ we obtain
$\lim_{\varepsilon\to0}\varepsilon^{-1}
|\Delta\cap[a^*-\varepsilon,a^*]|=1$ as desired. \qed

\subsection{Parameter dependence of nice critical
approximations}\label{quas} The rest of this section is entirely
devoted to the proof of Proposition \ref{kuso} and Lemma
\ref{good2}. A key ingredient is a {\it deformation of a quasi
critical approximation}, developed in [\cite{T1} Section 4,5] for
dealing with the parameter dependence of nice critical approximations.

We begin by relaxing the definition of nice critical
approximations as follows.
Let $\zeta$ be a critical approximation of order
$n$ on a horizontal curve $\gamma$. 
Let $u$ denote any unit vector at $f^{-[\theta n]}\zeta$
such that $Df^{[\theta n]}u$ is tangent to $\gamma$.
We say $\zeta$ is a {\it quasi critical
approximation of order $n$ on $\gamma$} if
$u$ is $\kappa_0^{\frac{1}{2}}$-expanding
up to time $[\theta n]$. 
\medskip

\noindent{\bf Hypothesis for the rest of Sect.\ref{quas}:}
Let $\hat a\in [a^*-\varepsilon_0,a^*]$.
Write $f$ for $f_{\hat a}$. Let $\gamma$ be a $C^2(b)$-curve in
$I(\delta)$. Let $\zeta$ be a quasi critical approximation of
order $n$ on $\gamma$, with $u$ the same meaning as above. Assume:
\smallskip

$\bullet$ $\Vert Df^i(f\zeta)\Vert\geq 1$ for $1\leq i\leq
n$;

 $\bullet$  $u$ is $\kappa_0^{\frac{1}{3}}$-expanding
and $\delta/160$-regular, both up to time $[\theta
n]$.
\medskip

 Let $r$ denote the point of intersection between
$H$ and the long stable leaf of order
$[\theta n]$ through $\xi$. Let $l\subset H$ denote the horizontal
of length $2\kappa_0^{3\theta n}$ centered at $r$. By [\cite{T1}
Lemma 4.1], $f_{\hat a}^{[\theta n]}l$ is a $C^2(b)$-curve, and
there exists a quasi critical approximation of order $n$ on it,
denoted by $\zeta(\hat a)$ for which $|\zeta-\zeta(\hat
a)|\leq(Cb)^{\frac{\theta n}{4}}$ holds.
In addition,
this picture persists, for a small variation of
parameters within the interval
\begin{equation}\label{continu}
I_{n}(\hat a)=[\hat a-\kappa_0^{n},\hat a+
\kappa_0^{n}].\end{equation} By [\cite{T1} Lemma 4.2], for all
$a\in I_{n}(\hat a)$, $f_a^{[\theta n]}l$ is a $C^2(b)$-curve.
By [\cite{T1} Proposition 4.1], there exists a quasi critical
approximation of order $n$ of $f_a$ on it, which we denote by
$\zeta(a)$.

\begin{definition}
{\rm The map $a\in I_n(\hat a)\to\zeta(a)$ is called a {\it
deformation} of $\zeta$.}
\end{definition}

The next lemma states that the ``speed" of the deformation as $a$
sweeps $I_n(\hat a)$ is
uniformly bounded. We use `` $\cdot$ " to denote the differentiation
with respect to $a$.
\begin{lemma}\label{atsui}{\rm (\cite{T1} Proposition 4.2.)}
The deformation $a\in I_n(\hat a)\to \zeta(a)$ of $\zeta$ is $C^3$ and for all
$a\in I_n(\hat a)$,
$\max\left(\Vert\dot\zeta(a)\Vert,\Vert\ddot\zeta(a)\Vert\right)
\leq\kappa_0^{10\log\delta}.$
\end{lemma}


\subsection{Evolution of critical curves}
In the next proposition we assume 
$\hat a\in\Delta_{n-1}$, $0<\nu<20n$ and $\hat\zeta$ is a nice critical
approximation of $f_{\hat a}$ of order $\geq n$, for which 
$(G)_{\nu-1}$ holds and $f_{\hat a}^\nu\hat\zeta$ is free. 
Define
$$J_\nu(\hat a,\hat\zeta)=[\hat a-\Theta_\nu(\hat\zeta),
\hat a+\Theta_\nu(\hat\zeta)].$$

\begin{prop}\label{mato}{\rm [\cite{T1} Section 5]}
There exist an integer $m$ and
a quasi critical approximation $\zeta$ of order $m$ of $f_{\hat a}$
such that:

\noindent {\rm (i)} $|f^{\nu}_{\hat a}\hat\zeta-f^{\nu}_{\hat a}\zeta|\leq
(Cb)^{\frac{1}{2}\theta \nu}$;

\noindent {\rm (ii)} for the deformation $a\in I_m(\hat a)\to \zeta(a)$ of
 $\zeta$, write $f_a^\nu\zeta(a)=\zeta_\nu(a).$
Then:

{\rm (a)} $J_{\nu}(\hat\zeta,\hat a)\subset I_m(\hat a)$:

{\rm (b)} the set $\{\zeta_\nu(a)\colon a\in J_\nu(\hat a,\hat\zeta)\}$
is a horizontal curve;

{\rm (c)} $ \Vert\zeta_\nu(a)-\zeta_\nu(b)\Vert \approx \Vert
w_{\nu}(\hat\zeta)\Vert|a-b|\ll1$ for all $a,b\in
J_{\nu}(\hat a,\hat\zeta)$.

\end{prop}

\subsection{Proof of Proposition \ref{kuso}.}\label{ve}
We choose each $J_{k,i}$ so that it has the form
$J_{k,i}=J_{\nu_k}(a_{k,i},\zeta_{k,i})$, where $a_{k,i}\in
E_n(*)\cap[a^*-\varepsilon,a^*]$ and $\zeta_{k,i}$ is some
responsible critical approximation of $f_{a_{k,i}}$. In what
follows we describe how to choose $(a_{k,i},\zeta_{k,i})_{i}$.
\medskip

\noindent{\it The first level.} Start with $k=1$. We describe how to choose
$(a_{1,i},\zeta_{1,i})_{i}$ such that (a) holds with $k=1$. First,
choose arbitrary $a_{1,1}\in E_n(*)\cap[a^*-\varepsilon,a^*]$. Let
$\zeta_{1,1}$ denote any responsible critical approximation of
$f_{a_{1,1}}$. We show
\begin{equation}\label{dc1} E_n(*)\cap (J_{1,1}\setminus
e^{-r_1/3} J_{1,1})=\emptyset.\end{equation} If $ J_{1,1}$ covers
$E_n(*)$, then the desired inclusion follows. Otherwise, choose
$a_{1,2}\in E_n(*)-{J}_{1,1}$. We claim that
\begin{equation}\label{dc2}
J_{1,1}\cap{J}_{1,2}=\emptyset.\end{equation} If
$J_{1,1}\cap{J}_{1,2}$ covers $E_n(*)$, then the desired inclusion
follows. Otherwise, choose $a_{1,3}\in E_n(*)-{J}_{1,1}\cup
J_{1,2}$. Repeat this. As the length of these intervals are
uniformly bounded from below, there must come a point when the
inclusion is fulfilled.

Below we sketch the proofs of (\ref{dc1}) and (\ref{dc2}). To ease
notation, write $a_i:=a_{1,i}$, $\zeta_i:=\zeta_{1,i}$ and
$J_i=J_{1,i}$, $i=1,2$.
\medskip

\noindent{\it Sketch of the proof of (\ref{dc1}).}  Choose an
integer $m$, a quasi critical approximation $\zeta$ of $f_{a_{1}}$
of order $m$, and its deformation $a\in I_m(a_{1}) \to\zeta(a)$
for which the conclusions of Proposition \ref{mato} hold up to
time $\nu_1$. In fact, (Z4), (Z5) allow us to choose such a
deformation so that the following holds:

\smallskip

\noindent $\bullet$
$|f_{a_{1}}^{\nu_1}\zeta_{1}-f_{a_{1}}^{\nu_1}\zeta(a_{1})|\leq
e^{-r_1};$

\noindent $\bullet$ if $a\in  J_{1}\cap E_n(*)$ and $x$ is {\it any}
responsible critical approximation of $f_a$, then
$|f_a^{\nu_1}x-f_{a}^{\nu_1}\zeta(a)|\ll e^{-r_1}.$
\smallskip

\noindent The second item states that, although responsible
critical approximations for a single parameter $a$ are not unique,
all of their positions at time $\nu_1$ are well-approximated by
that of $f_{a}^{\nu_1}\zeta(a)$. \medskip

Now, let $z_1$ denote the binding point of order $n_1$ for
$f_{a_{1}}^{\nu_1}\zeta_{1}$ and let $a\in I_{n_1}(a_1)\to
z_1(a)$ denote its deformation. (Z6) allows us to show that this
deformation satisfies:
\smallskip

\noindent $\bullet$ $|z_1-z_1(a_1)|\leq e^{-r_1}$;

\noindent $\bullet$ if $a\in  J_{1}\cap E_n(*)$ and $x$ is {\it any}
responsible critical approximation of $f_a$, with $y$ a binding
point for $f_a^{\nu_1}x$, then $|y-z_1(a)|\ll e^{-r_1}.$
\smallskip

\noindent The second item states that, although binding points are
not unique, they are well approximated by $z_1(a)$.

These four conditions altogether imply (\ref{dc1}). To see this, suppose
that this is not the case and let $a\in J_{1,1}\setminus
e^{-r_1/3}\cdot J_{1,1}$, $a\in E_n(*)$. Let $x$ denote {\it any}
critical approximation responsible for $a$. Let $y$ denote {\it
any} binding point for $f_a^{\nu_1}x$. The triangle inequality
gives
\begin{align*}|f_a^{\nu_1}x-y|&\geq|f_a^{\nu_1}\zeta(a)-
f_{a_1}^{\nu_1}\zeta(a_1)|-
|f_{a}^{\nu_1}\zeta(a)-f_{a}^{\nu_1}x|-
|f_{a_1}^{\nu_1}\zeta(a_1)-f_{a_1}^{\nu_1}\zeta_1|\\
&-|f_{a_1}^{\nu_1}\zeta_1-z_1| -|z_1-z_1(a)|-|z_1(a)-y|,
\end{align*}
where, for the last term, $z_1(a)$ makes sense, because of
$J_{1}\subset I_{n_1}(a_1)$. On the first term, Proposition
\ref{mato} and Lemma \ref{sample4} give
$$|f_{a_{1}}^{\nu_1}\zeta(a_{1})-
f_{a}^{\nu_1}\zeta(a)|\approx\|w_{\nu_1}(\zeta_{1})\|\cdot|a_{1}-a|
\gg e^{-r_1}.$$ The remaining four terms are $\leq e^{-r_1}$. It
follows that $|f_a^{\nu_1}x-y|\gg e^{-r_1}$. This yields a
contradiction to the assumption that $x$ is responsible for $a$.
Hence $a\notin E_n(*)$ holds and we get (\ref{dc1}).\medskip

\noindent{\it Sketch of the proof of (\ref{dc2}).} In the
discussions to follow, we need to introduce {\it critical
parameters} \cite{T1}. For the purpose of this
we make the following assumption and observation. Let
$\hat a\in E_n(*)$ and let $\hat\zeta$ denote any critical
approximation responsible for $a$. Let $z$ denote any binding
point for $f_{\hat a}^{\nu_k}\hat\zeta$, and let $a\in
I_{n_k}(\hat a)\to z(a)$ denote its deformation. Take an integer
$m$, a quasi critical approximation $\zeta$ of $f_{\hat a}$ of
order $m$, and its deformation $a\in I_{m}(\hat a)\to\zeta(a)$ for
which the conclusions of Proposition \ref{mato} hold up to time
$\nu_k$. The ``speed" of $z(a)$ as $a$ sweeps the interval
$I_{n_k}(\hat a)$ is bounded from above by in Lemma \ref{atsui}.
On the other hand, the ``speed" of $\zeta_{\nu_k}(a)$ as $a$
sweeps the interval $J_{\nu_k}(\hat a,\hat\zeta)$ is much
faster, by Proposition \ref{mato}. From the proposition,
$J_{\nu_k}(\hat a,\hat\zeta)\subset I_{n_k}(\hat a)$ holds. Hence,
the comparison of the speeds and Lemma \ref{sample4} together
imply that there exists a unique parameter $c_0\in
e^{-r_k/3}\cdot J_{\nu_k}(\hat a,\hat\zeta)$ such that the
$x$-coordinate of $\zeta_{\nu_k}(c_0)$ coincides with that of
$z(c_0)$.

\begin{definition}{\rm The $c_0$ is called
a {\it critical parameter in $J_{\nu_1}(\hat
a,\hat\zeta)$}.}
\end{definition}

A proof of (\ref{dc2}) is outlined as follows. Let $c_0,c_0'$
denote the critical parameters in $J_{1,1}$, $J_{1,2}$
respectively. Suppose that (\ref{dc2}) does not hold. Then, from a
distortion argument, $|J_{1,1}|\approx |J_{1,2}|$ follows. As
$a_{1,2}\notin J_{1,1}$, this implies $c_0\neq c_0'$. In addition,
it is possible to extend the domain of definition of the
deformation of $\zeta_{1,1}$ to the larger interval $J_{1,1}\cup
J_{1,2}$, so that all the above properties of the deformation
continue to hold. As $a_{1,2}\notin J_{1,1}$, the argument used in
the proof of (\ref{dc1}) gives $a_{1,2}\notin E_n(*)$. This is a
contradiction. Hence (\ref{dc2}) holds.
\medskip

\noindent{\it From level $k-1$ to $k$.} 
Having chosen $(a_{k-1,i},\zeta_{k-1,i})_i$ and the corresponding
intervals $(J_{k-1,i})_i$,
we choose $(a_{k,j},\zeta_{k,j})_j$ as follows. For each
$J_{k-1,i}$, in the same way as the proof of (\ref{dc1}) it is
possible to choose a finite number of parameters
$a_{k,1},a_{k,2},\cdots$ in $E_n(*)\cap[a^*-\varepsilon,a^*]\cap
e^{-r_{k-1}/3}\cdot{J}_{k-1,i}$ such that the corresponding
intervals $J_{k,1},J_{k,2},\cdots$ are pairwise disjoint and
altogether cover $E_n(*)\cap e^{-r_{k-1}/3}\cdot{J}_{k-1,i}.$ Now
the issue is to show the inclusion
$\bigcup_j{J}_{k,j}
\subset 2e^{-r_{k-1}/3}
 \cdot{J}_{k-1,i}.$
This is a consequence of the fact that the center $a_{k,j}$ of
each $J_{k,j}$ belongs to $e^{-r_{k-1}/3}\cdot{J}_{k-1,i}$, and
any $J_{k,j}$ does not contain the critical parameter in
$J_{k-1,i}$. \qed
\smallskip

\begin{lemma}\label{compare}
For every $i$, $\Theta_{\nu_1}(\zeta_{1,i})\leq2\varepsilon.$
\end{lemma}
As the intervals $(J_{1,i})_i$ are pairwise disjoint and intersect
$[a^*-\varepsilon, a^*]$, Lemma \ref{compare} gives
$\sum_i|J_{1,i}|\leq 3\varepsilon$. This proves (c).
\medskip

It is left to prove Lemma \ref{compare}. We use the following
which can be proved by slightly extending the arguments in
Sect.\ref{errant} and using the definition of quasi critical
approximations.
\begin{claim}\label{trivial}
Let $\zeta$ be a quasi critical approximation of order $n$ of
$f_{a^*}$. There exists a critical point $z$ of $f_{a^*}$ such
that $|\zeta-z|\leq (Cb)^{\frac{1}{2}\theta n}.$
\end{claim}

\noindent{\it Proof of Lemma \ref{compare}.} Take an integer $m$,
a quasi critical approximation $\zeta$ of $f_{a_{1,i}}$ of order
$m$, and its deformation $a\in I_m(a_{1,i})\to\zeta(a)$ for which
the conclusions of Proposition \ref{mato} hold up to time $\nu_1$.
If $|J_{1,i}|> 2\varepsilon,$ then $a^*\in J_{1,i}$ holds, because
of $a_{1,i}\in[a^*-\varepsilon,a^*]$. Then $\zeta(a^*)$ makes
sense and we have
$|f_{a_{1,i}}^{\nu_1}\zeta_{1,i}-\zeta_{\nu_1}(a^*)|\leq
|f_{a_{1,i}}
^{\nu_1}\zeta_{1,i}-\zeta_{\nu_1}(a_{1,i})|+|\zeta_{\nu_1}(a_{1,i})
-\zeta_{\nu_1}(a^*)|\ll1.$ As $\nu_1$ is a return time, $f_{
a_{1,i}}^{\nu_1}\zeta_{1,i}\in I(\delta)$ holds. It follows that
$\zeta_{\nu_{1}}(a^*)$ is near $I(\delta)$. On the other hand,
Proposition \ref{good1} and Claim \ref{trivial} together imply
$\zeta_{\nu_1}(a^*)\in\{(x,y)\colon |x|\geq 4/5\}$. A
contradiction arises. \qed

\subsection{Proof of Lemma \ref{good2}.} 
We argue by induction on $n$. 
Let $1<k\leq \min\left(n_0(\varepsilon),20\xi\right)$
and assume $f_{a_0}^i\zeta_0\notin I(\delta)$ for
$1\leq i\leq k-1$. 
Then $f_{a_0}^{k}\zeta_0$ is free.
Take an integer $m$, a quasi critical approximation $\zeta$ of
$f_{a_0}$ of order $m$, and its deformation $a\in I_m(\hat a)\to
\zeta(a)$ for which the conclusions of Proposition \ref{mato} hold
up to time $k$. The definition of $J_k(a_0,\zeta_0)$
and (\ref{n0}) give
$$|J_k(a_0,\zeta_0)|\geq \kappa_0C_0^{-2k}\geq
2\varepsilon.$$ As $a_0\in[a^*-\varepsilon,a^*]$, $a^* \in
J_k(a_0,\zeta_0)$ holds. Hence, $\zeta(a^*)$
makes sense and we have
$|f_{a_0}^k\zeta_0-\zeta_k(a^*)|\leq|f_{a_0}^k\zeta_0-\zeta_k(a_0)|+
|\zeta_k(a_0)-\zeta_k(a^*)|\ll1.$ Proposition \ref{good1} and
Claim \ref{trivial} give $\zeta_k(a^*)\notin \{(x,y)\colon |x|\leq
4/5\}$. Hence $f_{a_0}^{k}\zeta_0\notin I(\delta)$ holds, 
recovering the assumption of the induction. \qed

\subsection{Proof of Proposition \ref{good1}}\label{proo}
Write $f$ for $f_{a^*}$. Let $r$ denote the point of the quadratic
tangency near $(0,0)$. Let $S$ denote the lenticular compact
domain in $I(\delta)$ bounded by the segment in $W^u$ and the
parabola in $W^s(Q)$ containing $r$ (cf. Figure 1). By (M1), all
points in $fS$ do not return to $R_0$ under positive iteration,
and thus they are expanding. By Proposition \ref{fo}, $fS$ is
foliated by long stable leaves. Note that the leaf through $fr$
contains the boundary of $R_0$.

Temporarily, let us adopt the following definition. Let $\gamma$
be a $C^2(b)$-curve in $W^u(Q)$ stretching across $I(\delta)$. We
say $\zeta\in \gamma$ is a {\it critical point on $\gamma$} if
$z\in S$, and the long stable leaf through $fz$ is tangent to
$W^u(Q)$ at $fz$. For the proof of Proposition \ref{good1}, we
approximate any critical approximation by a critical point. Since
the orbit of every critical point do not return to $R_0$, the
conclusion of the proposition follows.

\begin{lemma}\label{sox}
Let $\gamma$ be a $C^2(b)$-curve in $W^u(Q)$ stretching across
$I(\delta)$. There exists a unique critical point on $\gamma.$ In
addition, for every $n\geq M$ there exists a critical
approximation of order $n$ on $\gamma$ within the distance
$(Cb)^\frac{n}{4}$ from the critical point.
\end{lemma}
\begin{proof}
By Remark \ref{verti}, any long stable leaf at the right of
$\Gamma(fr)$ intersects $f\gamma$ at two points, or else it is
tangent to $f\gamma$ and the point of tangency is quadratic. There
exists only one leaf for which the latter holds, for otherwise two
distinct leaves intersect each other, a contradiction to the
remark below Lemma \ref{leaf}. The pull-back of the point of
tangency is a critical point on $\gamma,$ denoted by $\zeta$.
Hence, the first statement holds.

Take $z\in\gamma$ with $|\zeta-z| =b^{\frac{n}{4}}$, and write
$fz=(x_0,y_0)$. Represent the two long stable leaves as graphs of
functions on $[-\sqrt{b},\sqrt{b}]$: $\Gamma_n(z)=\{(x(y),y)\}$
and $\Gamma_n(f\zeta)=\{(\tilde x(y),y)\}$. Since the Hausdorff
distance between $\Gamma_n(f\zeta)$ and $\Gamma(f\zeta)$ is $\leq
(Cb)^{n}$, Lemma \ref{quadratic} gives $|x(y_0)-\tilde
x(y_0)|=|x_0-\tilde x(y_0)|\approx b^{\frac{n}{2}}$. Since $e_n$
is Lipschitz, it follows that $|x(y)-\tilde x(y)|\approx
Cb^{\frac{n}{2}}$ for all $y \in[-\sqrt{b},\sqrt{b}]$. Hence
$f^{-1}\Gamma_n(fz)$ intersects $\gamma$ at two points within
$(Cb)^{\frac{n}{4}}$ from $\zeta$. This and Remark \ref{verti}
together imply the second statement.
\end{proof}


Let $\zeta_0$ denote the critical point which is closest to $Q$ in
the Riemannian distance in $W^u(Q)$. Let $G$ denote the segment in
$W^u(Q)$ with endpoints $Q$, $f\zeta_0$. A proof of the next lemma
is given in Appendix A.1.
\begin{lemma}\label{c2}
For every $n\geq0$, any component of $f^nG\cap I(\delta)$ is a
$C^2(b)$-curve.
\end{lemma}

We are in position to finish the proof Proposition \ref{good1}. If
$|f^{-[\theta n]}\zeta-fr|\leq1/10$, then let $m=[\theta n]-1$.
Otherwise, let $m=[\theta n].$ Then $f^{-m}\zeta$ is expanding.
Let $z$ denote the point of intersection between the long stable
leaf of order $m$ through $f^{-m}\zeta$ and $G$. It is possible to
take a curve $\gamma$ in $G$ extending both sides around $z$ to
length $b^{\frac{m}{3}}.$ For otherwise the contraction along the
long stable leaf gives $f^mQ\in I(\delta)$, a contradiction
because $Q$ is a fixed point and $Q\notin I(\delta)$. By the
definition of $m$, $\gamma$ avoids the turn near $f\zeta_0$, and
hence is $C^2(b)$. Then $f^m\gamma$ is a $C^2(b)$-curve extending
both sides around $f^mz$ to length $\geq b^{\frac{m}{2}}$. By
Lemma \ref{update}, there exists a critical approximation $\bar z$
of order $n$ on $f^m\gamma$ such that
 $|\zeta-\bar z|\leq(Cb)^{\frac{\theta n}{4}}$ holds.
By  Lemma \ref{sox}and Lemma \ref{c2}, there exist a
$C^2(b)$-curve $\gamma'$ in $W^u$ containing $f^m\gamma$ and
stretching across $I(\delta)$, and a critical point $\zeta''$ on
$\gamma'$ such that $|\bar z-\zeta''|\leq (Cb)^\frac{\theta
n}{4}$. It follows that $|f^i\zeta-f^i\zeta''|\leq
(Cb)^\frac{\theta n}{5}$ for $1\leq i<20n$. As the orbit of
$\zeta''$ is out of $R_0$, the claim holds. \qed
\medskip

\noindent{\bf Standing hypothesis for the rest of the paper:}
$f\in\{f_a\colon a\in\Delta\cap(a^{**},a^*]\}$. Here,
$a^{**}$ is the one defined in Introduction.

\section{Dynamics on the unstable manifold}
In this section we develop a one-dimensional analysis on the
unstable manifold $W^u$. In Sect.\ref{setc}, we define a {\it
critical set} $\mathcal C$ in $W^u$, as a set of accumulation
points of critical approximations, and use it as a spine to structure
the dynamics.
 Each element of $\mathcal C$ is called a {\it critical
point}. In Sect.\ref{rec}, \ref{cripa} we prove some key estimates
on critical points. In Sect.\ref{geometry} we identify a geometric
structure of $W^u$ near the critical set.
\medskip

\noindent{\it Notation.} For $z\in W^u$, let $t(z)$ denote any
unit vector tangent to the unstable manifold at $z.$ The
boundaries of $R_0$ in $W^u$ is called {\it unstable sides}, and
denoted by $\partial R_0.$ Let $\partial R_n:=f^n(\partial R_0)$.

\subsection{The critical set}\label{setc}
In the case $W^u=W^u(Q)$, fix a fundamental domain $\mathcal F$ in
$W^u_{\rm loc}(Q)$. For $z\in\mathcal F$, define a sequence
$n_1<n_1+p_1\leq n_2<n_2+p_2\leq n_3<\cdots$ inductively as
follows: $n_1$ is the smallest such that $f^{n_1}z\in I(\delta)$
and $p_1$ is the bound period of $f^{n_1}z$; $n_{k}\geq
n_{k-1}+p_{k-1}$ is the smallest such that $f^{n_k}z\in
I(\delta)$, and $p_k$ is the bound period of $f^{n_k}z$. From the
fact that $Q$ is a fixed saddle, it follows that this sequence is
defined indefinitely, or else there exists an integer $m$ such
that $Df^mt(z)$ is in critical position relative to critical
approximations of arbitrarily high order. If the latter case
occurs, we let $f^mz\in\mathcal C$. Since each such point is
isolated in $W^u$, $\mathcal C$ is a countable set. In the case
$W^u=W^u(P)$, $\mathcal C$ is constructed in the same way, with
$Q$ replaced by $P$.

\begin{prop}\label{criticaset}
For each $\zeta\in\mathcal C$ we have:

\smallskip

\noindent{{\rm (a)}} $\|w_{n}(\zeta)\|\geq e^{\lambda (n-1)}$ for
$n\geq1$;

\noindent{{\rm (b)}} $\|w_{j}(\zeta)\| \geq e^{- 2\alpha
i}\|w_i(\zeta)\|$ for $1\leq i<j$;

\noindent{{\rm (c)}} there exists a monotone increasing function
$\chi\colon\mathbb N \circlearrowleft$ such that for each $n$,
$(1-\sqrt\alpha)n\leq\chi(n)\leq n$ and
$\|w_{\chi(n)}(\zeta)\|\geq \delta \|w_{k}(\zeta)\|$ for $1\leq k<
\chi(n)$;

\noindent{{\rm (d)}} the long stable leaf through $f\zeta$ is
tangent to $W^u$ at $f\zeta$ and the tangency is quadratic.
\end{prop}


\begin{proof}
By definition, for each $\zeta\in\mathcal C$ there exists a
strictly increasing sequence $m_1<m_2<\cdots$ of integers and a
sequence $\zeta_{m_1},\zeta_{m_2},\cdots$ of critical
approximations with good critical behavior, such that
$\zeta_{m_\ell}$ is of order $m_\ell$, and
$\zeta_{m_\ell}\to\zeta$ as $\ell\to\infty$. (a) (b) (c) are
direct consequences of this convergence. By the definition of
$\mathcal C$ and (ii) in Proposition \ref{recovery}, $t(\zeta)$ is
contracted exponentially by positive iterations. Thus $t(f\zeta)$
is tangent to $\Gamma(f\zeta)$. By Remark \ref{verti},
this tangency is quadratic, and (d) holds. 
\end{proof}




\subsection{Recovering expansion}\label{rec}
In this and the next subsection we assume that $\zeta\in\mathcal
C$ is on a horizontal curve $\gamma$ in $I(\delta)$, namely,
$\Gamma(f\zeta)$ is tangent to $f\gamma$ at $f\zeta$. We state a
version of Proposition \ref{recovery} which is proved similarly.
The difference is that $\zeta$ is no longer an approximation and a
``genuine'' critical point, and thus the estimates are available
entirely on $\gamma.$

As before, write $\Gamma(f\zeta)=\{(x(y),y)\colon
|y|\leq\sqrt{b}\}$, and for each $k\geq M$, let $V_k=\{(x,y)\colon
|x-x(y)|\leq D_k(\zeta)/2,|y|\leq\sqrt{b}\}.$
 If $fz\in V_k\setminus V_{k+1},$ define a {\it bound
period} $p=p(\zeta,z)$ by $$p=\chi(k),$$ and a {\it fold period}
$q=q(\zeta,z)$ by $$q=\min\left\{i\in[1,p)\colon|\zeta-z| ^{\beta}
\cdot \|w_{j+1}(\zeta)\|\geq1\text{ for } i\leq j<p\right\}.$$

\begin{prop}\label{recovery0}
 Let $z\in \gamma\setminus\{\zeta\}$ and let $t(z)$ denote any
unit vector tangent to $\gamma$ at $z$. Then:

{\rm (a)}
 $p\leq\log
|\zeta-z|^{-\frac{3}{\lambda}};$

{\rm (b)} $q\leq C\beta p$;

{\rm (c)} $|f^i\zeta-f^iz|\leq e^{-2\alpha p}$ for $1\leq i\leq
p;$

{\rm (d)} $|\zeta-z|\leq\Vert Df^qt(z)\Vert\leq |
\zeta-z|^{1-\beta};$

{\rm (e)} $\displaystyle{\Vert
Df^pt(z)\Vert\geq|\zeta-z|^{-1+\frac{\alpha}{\log C_0}} \geq
e^{\frac{\lambda p}{3}}};$

{\rm (f)} $\Vert  Df^pt(z)\Vert\geq (\delta/10)\Vert
Df^it(z)\Vert$ for $0\leq i<p;$

{\rm (g)} $\|Df^it(z)\|\approx |\zeta-z|\|w_i(\zeta)\|$
 for $q\leq i\leq p$;

{\rm (h)} $\|Df^it(z)\|<1$
 for $1\leq i\leq q$.
\end{prop}

\subsection{Critical partitions}\label{cripa}
Using the family $(V_k)_k$ of vertical strips, we construct a {\it
critical partition} of $\gamma$ as follows. By Remark \ref{verti},
$\gamma\cap f^{-1}(V_k\setminus V_{k+1})$ consists of two
components, one at the right $\zeta$ and the other at the left.
For simplicity, let us denote both by $\gamma_k$. If $f\gamma_{k}$
does not intersect the vertical boundary of $V_k$, then we take
$\gamma_{k}$ together with the adjacent $\gamma_{k+1}$. We cut
each $\gamma_{k}$ into $[e^{3\alpha k}]$-number of curves of equal
length, and denote them by $\gamma_{k,s}$ $(s=1,2,\cdots)$.

A proof of the next lemma is given in Appendix A.2.
\begin{lemma}\label{boun}
For each $\gamma_{k,s}$ we have:

\noindent{\rm (a)} $f^{\chi(k)}\gamma_{k,s}$ is a $C^2(b)$-curve
of length
 $\geq e^{-4\alpha k}$;

\noindent{\rm (b)} 
For all
$\xi,\eta\in\gamma_{k,s}$,
$$\log\frac{\Vert Df^{\chi(k)}t(\xi)\Vert}{
\Vert Df^{\chi(k)}t(\eta)\Vert}\leq
C|f^{\chi(k)}\xi-f^{\chi(k)}\eta|^{C\alpha}.$$
\end{lemma}

\begin{center}
\begin{figure}
\input{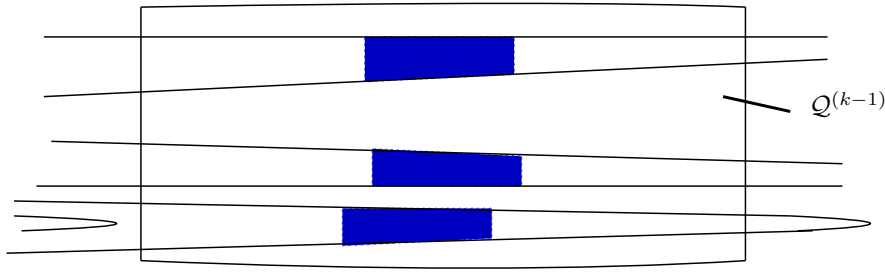}\label{criticalset}
\caption{The relation between $\mathcal C^{(k-1)}$ and $\mathcal C^{(k)}$.
The shaded regions are components of $\mathcal C^{(k)}$.}
\end{figure}
\end{center}

\subsection{Geometry of critical regions.}\label{geometry}
We identify a geometric structure of critical regions,
close the one depicted in (\cite{WY01} Sect.1.2). 
Let $\mathcal C^{(0)}=
\{(x,y)\in R_0 \colon |x|\leq\delta\}$.
 \begin{prop}\label{geo} There exists a nested sequence $\mathcal C^{(0)}\supset\mathcal
C^{(1)} \supset\mathcal C^{(2)}\supset\cdots$ such that the
following holds for $k=0,1,2,\cdots$:
\smallskip

\noindent {\rm (S1)} $\mathcal C^{(k)}$ has a finite number of
components called $\mathcal Q^{(k)}$ each one of which is
diffeomorphic to a rectangle. The boundary of $\mathcal Q^{(k)}$
is made up of two $C^2(b)$-curves of $\partial R_k$ connected by
two vertical lines: the horizontal boundaries are
$\approx\min(2\delta,\kappa_0^{k})$ in length, and the Hausdorff
distance between them is $\mathcal O(b^{\frac{k}{2}})$;
\smallskip

\noindent{\rm (S2)} On each horizontal boundary $\gamma$ of each
component $\mathcal Q^{(k)}$ of $\mathcal C^{(k)}$, there is a
critical point located within $\mathcal O(b^{\frac{k}{4}})$ of the
midpoint of $\gamma$.
\smallskip

\noindent{\rm (S3)} $\mathcal C^{(k)}$ is related to $\mathcal
C^{(k-1)}$ as follows: $\mathcal Q^{(k-1)}\cap R_{k}$ has at most
finitely many components, each of which lies between two $C^2(b)$
subsegments of $\partial R_{k}$ that stretch across $\mathcal
Q^{(k-1)}$ as shown in FIGURE 7. Each component of
$\mathcal Q^{(k-1)}\cap R_{k}$ contains exactly one component of
$\mathcal C^{(k)}$.
\smallskip

\noindent{\rm (S4)} Let $\Xi^{(k)}$ denote the set of critical
points on the horizontal boundaries of $\bigcup_{j=0}^k\mathcal
C^{(j)}.$ Then $\mathcal C=\bigcup_{k\geq0}\Xi^{(k)}.$
\end{prop}

The rest of this section is entirely devoted to an inductive proof
of (S1), (S2), (S3). (S4) is a direct consequence of this. In
Section \ref{log}, we first describe a structure of the induction,
to make clear how to proceed from one to the next step. In Section
\ref{indu0} we treat an initial step of the induction. In Section
\ref{proce} we treat a generic step.

\subsection{Structure of induction}\label{log}
(S1), (S2) for $k=0$ are trivial. (S3) for $k=0$ is an empty
condition. Let us say that {\it $\partial R_0$ is controlled up to
time $0$ by $\Xi^{(0)}$.} Using the critical partition in
Sect.\ref{cripa}, we assign to all points in $\partial R_0\cap
I(\delta)$ their binding points in $\Xi^{(0)}$ and bound periods.
This makes sense to refer to points in $\partial R_1$ as being
free or bound.

\begin{definition}\label{indu1}
{\rm Let $j\geq1$ and assume:
\smallskip

\noindent{$(I)_{j-1}$:} \ (S1-3) hold for $0\leq k\leq j-1$, and
$\partial R_0$ is controlled up to time $j-1$ by $\Xi^{(j-1)}$.
\smallskip

\noindent Under this assumption, we say:
\smallskip

\noindent $\bullet$ a segment in $\partial R_{j}$ is a {\it free
segment} if all points on it are free;

\noindent $\bullet$ a maximal free segment in $\partial R_{j}$ is
a free segment in $\partial R_{j}$ which is not contained in any
other free segment in $\partial R_{j}$;

\noindent $\bullet$ a bound segment in $\partial R_{j}$ is any
connected component of $\partial R_{j}\setminus\{\text{maximal
free segment in $\partial R_{j}$}\}.$}
\end{definition}

In the sequel we need two curvature-related estimates.
\begin{lemma}\label{feer}
Any free segment in $\partial R_{j}$ is a $C^2(b)$-curve.
\end{lemma}
\begin{proof}
Let $\gamma$ be a free segment in $\partial R_j$. Then $1\geq
C\delta\Vert Df^{-n}(z)t(z)\Vert$ holds for all $z\in\gamma$ and
$n>0$. Hence, the curvature of $\gamma$ is $\leq\sqrt{b}$, by the
 curvature
estimate in
[\cite{T1} Lemma 2.4] and the boundedness of the curvature of
$W^u_{\rm loc}$. The inequality for $n=-1$ implies that the slopes
of the tangent directions of $\gamma$ are $\leq\sqrt{b}$.
\end{proof}

\begin{lemma}\label{curv}
For any free segment $\gamma$ and $n\geq0$, the curvature of
$f^{-n}\gamma$ is everywhere $\leq 5^{3n}$.
\end{lemma}
\begin{proof}
 For $z\in\gamma$, let $\kappa_{-n}(z)$ denote the curvature of $f^{-n}\gamma$
at $f^{-n}z$. If $f^{-n}z$ is free, then
$\kappa_{-n}(z)\leq\sqrt{b}$, by Lemma \ref{feer}. Otherwise, let
$m<-n$ denote the largest integer such that $f^{m}z$ is a free
return. [\cite{T1} Lemma 2.4] and $\kappa_{m}(z)\leq \sqrt{b}$
give
\begin{align*}\kappa_{-n}(z)&\leq\sqrt{b}(Cb)^{-n-m}
\frac{\Vert Df^{m}t(z)\Vert^3} {\Vert Df^{-n}t(z)\Vert^3}
+\sum_{i=1}^{-n-m}(Cb)^i\frac{\Vert Df^{-n-i}t(z)\Vert^3}{\Vert
Df^{-n}t(z)\Vert^3}.\end{align*} Since $z$ is free, $\Vert
Df^{-n-i}t(z)\Vert\leq 10/\delta$, and thus for $1\leq i\leq-n-m$,
$$\frac{\Vert Df^{-n-i}t(z)\Vert}{\Vert Df^{-n}t(z)\Vert}
\leq 10\cdot 5^{n}\delta^{-1}.$$ Replacing all these in the above
inequality, we obtain $\kappa_{-n}(z)\leq 5^{3n}$.
\end{proof}

\begin{definition}\label{cos}
{\rm Suppose that (S1-3) hold for every $0\leq k\leq j$. We say
{\it $\partial R_0$ is controlled up to time $j$ by $\Xi^{(j)}$},
if for any maximal free segment $\gamma$ in $\partial R_{j}$ there
exist a horizontal curve $\tilde\gamma$ which contains $\gamma$
and a critical point $\zeta\in\Xi^{(j)}$ on $\tilde\gamma$.}
\end{definition}

At step $j-1$ of the induction, we show the implication
$(I)_{j-1}\Longrightarrow (I)_j.$ Then, for all points in
$\partial R_j\cap I(\delta)$ which are free, we assign their
binding points as follows. For a maximal free segment $\gamma$ in
$\partial R_j$, take $(\tilde\gamma,\zeta)$ as in Definition
\ref{cos}. We use $\zeta$ as a common binding point for points in
$\gamma\cap I(\delta)$. Their bound periods are given by
considering the critical partition of $\tilde\gamma$. This makes
sense to refer to points in $\partial R_{j+1}$ as being free or
bound.

\subsection{From step $0$ to step $N$}\label{indu0}
Let $1\leq j\leq N$ and suppose $(I)_{j-1}$. The bound parts of
$\partial R_j$ do not come back to $\mathcal C^{(0)}$, and
$\partial R_j\cap I(\delta)$ consists of $C^2(b)$ curves, each of
which admits a critical point. Define $\mathcal
C^{(j)}=R_j\cap\mathcal C^{(0)}$. $(I)_j$ obviously holds.

\subsection{From step $2^mN$ to $2^{m+1}N$}
\label{proce} The same argument cannot be continued indefinitely,
because bound segments return to $I(\delta)$. To deal with these
returns, we need the help of critical points.

\begin{lemma}\label{mega}
For each $\zeta\in \mathcal C$ there exist positive integers
$n_1<n_1+p_1\leq n_2<n_2+p_2\leq n_3<\cdots$ such that, for each
$n_l$, $f^{n_l}\zeta\in I(\delta)$, and there exists a critical
approximation $\hat z_{l}$ relative to which $w_{n_l}(\zeta)$ is
in admissible position, with $|f^{n_l}\zeta-\hat z_{l}|\geq e^{-
\alpha n_l}.$
\end{lemma}
The integers $n_1,n_2,\cdots$ are called {\it free return times}
of $\zeta$.
\begin{proof}
We argue by induction. First, let $n_1=\min \{n>0\colon
f^n\zeta\in I(\delta)\}$. As $I(\delta)$ is open, $n_1=\min
\{n>0\colon f^n\zeta_{m_\ell}\in I(\delta) \}$ holds for all
sufficiently large $\ell$. Let $z_{m_\ell}$ denote the binding
point for $f^{n_1}\zeta_{m_\ell}$, with a bound period
$p_{m_\ell}$. Passing to subsequences, we may assume that both
converge as $\ell\to\infty$. Define $\hat z_{1}$, $p_{1}$ to be
the corresponding limits.

Given $(n_k, \hat z_k, p_k)$, define $n_{k+1}=\min \{n\geq
n_k+p_k\colon f^n\zeta\in I(\delta) \}.$ Passing to subsequences
again, we may assume that $f^{n_{k+1}}\zeta_{m_{\ell'}}$ is a free
return to $I(\delta)$, with a binding point $z_{m_{\ell'}}$ and a
bound period $p_{m_{\ell'}}$, both converging as $\ell\to\infty$.
Define $\hat z_{k+1}$, $p_{k+1}$ to be the corresponding limits.
\end{proof}

\begin{definition}
{\rm Let $\zeta\in\mathcal C$, with $n_1,n_2,\cdots$ and $\hat
z_{1},\hat z_2,\cdots$ as in Lemma \ref{mega}. We say $\zeta$ is
{\it controlled up to time $n$ by $\Xi^{(k)}$} if, for each
$n_l\leq n$ there exists $z_l\in\Xi^{(k)}$ such that $|z_{l}-\hat
z_l| =\mathcal O(b^{\frac{\theta \xi}{5}})$, where $\xi$ is the
order of $\hat z_l$. Such $z_l$ is called a {\it binding point}
for $\zeta$.}
\end{definition}


Clearly, every $\zeta\in\mathcal C$ is controlled up to time $2N$
by $\Xi^{([\theta N])}$. To proceed from step $2^mN$ to step
$2^{m+1}N$, it suffices to show
\begin{lemma}\label{contor}
Let $m\geq0$. Suppose that $(I)_{2^mN}$ holds, and that every
$\zeta\in\mathcal C$ is controlled up to time $2^{m+1}N$ by
$\Xi^{([2^{m}\theta N])}$. Then:
\smallskip

\noindent{\rm (a)} {\rm (I)$_{k}$} holds for $2^m N<k\leq
2^{m+1}N$;

\noindent{\rm (b)} every $\zeta\in\mathcal C$ is controlled up to
time $2^{m+2}N$ by $\Gamma^{([2^{m+1}\theta N])}$.
\end{lemma}

\noindent{\it Proof of {\rm (a)}.} Assume $(I)_{j-1}$ for some
$2^mN< j\leq 2^{m+1}N$. Then $\Xi^{(j-1)}$ makes sense. We prove
$(I)_{j}$ in three steps.
\smallskip

\noindent{\it Step 1: Treatment of bound segments in $\partial
R_j$.} Let $d$ denote the Hausdorff distance.
\begin{lemma}\label{bo}
Let $B$ be a bound segment in $\partial R_j$. There exist $N<l<j$
and $\zeta\in\Xi^{(j-1)}$ such that $f^{l}\zeta$ is free and
$d(f^{l}\zeta,B)\leq e^{-2\alpha l}$. 
\end{lemma}
\begin{proof}
We define a sequence $z_0,\cdots,z_s$ in $\Xi^{(j-1)}$ and a
sequence $n_0,\cdots,n_s$ of positive integers inductively as
follows. By the definition of bound segments, there exists
$0<n_0\leq k$ such that $f^{-n_0}B$ contains a critical point in
$\Xi^{(j-n_0)}$, denoted by $z_0$. If $f^{n_0}z_0$ is bound, let
$n_1<n_0$ denote the free return time of $z_0$ with bound period
$p_1$, such that $n_1<n_0<n_1+p_1$. Let $z_1$ denote the
corresponding binding point, which is in $\Xi^{([\theta
n_1])}\subset\Xi^{(j-1)}$ by the assumption of induction. If
$f^{n_0-n_1}z_1$ is bound, then
 let $n_2<n_0-n_1$ denote the free return time of $z_1$ with bound period
$p_2$, such that $0<n_2<n_0-n_1<n_2+p_2$. Let $z_2$ denote the
binding point, which is in $\Xi^{([\theta
n_2])}\subset\Xi^{(j-1)}$, and so on.

We must reach some $n_s$ and $z_s$ such that
$f^{n_0-n_1-\cdots-n_s}z_{s}$ is free. By the inductive
assumption, each $z_i$ is controlled up to time $k-1$. Hence, for
each $i=1,\cdots,s$ we have $p_{i}<\frac{4\alpha}{\lambda}
p_{i-1}$. We have
\begin{align*} d(B,f^{n_0-n_1-\cdots-n_s}z_{s})\leq&d(B,
f^{n_0}z_0)+|f^{n_0}z_0-f^{n_0-n_1}z_1|\\
&+\sum_{k=2}^s|f^{n_0-n_1-\cdots-n_{k-1}}z_{k-1}-f^{n_0-n_1-
\cdots-n_k}z_k|\\
\leq& \sum_{k=0}^s 2e^{-2\alpha p_k}\leq 3e^{-2\alpha p_s}\leq
3e^{-2\alpha(n_0-n_1-\cdots-n_s)},
\end{align*}
where we have used (c) Proposition \ref{recovery0} for the second
inequality. As $z_{s-1}$ is bound at time
$n_0-n_1-\cdots-n_{s-1}$,
 $n_0-n_1-\cdots-n_{s-1}<p_s$ holds.
 Hence $n_0-n_1-\cdots-n_{s}< p_s$ and the last inequality holds.
Take $l=n_0-n_1-\cdots-n_s$ and $\zeta=z_s$. The argument shows
$N<l$.
\end{proof}

\begin{cor}\label{oo}
For any bound segment $B$ in $\partial R_{j}$ and
 $\alpha j\leq i<j$, $B\cap \mathcal C^{(i)}=\emptyset$.
\end{cor}
\begin{proof}
Take $l<j$ and $\zeta\in\Xi^{(j-1)}$ such that the conclusion of
Lemma \ref{bo} holds. If $f^l\zeta\in I(\delta)$, then let
$z\in\Xi^{([\theta l])}$ denote the binding point. We have
$d(B,z)\geq|f^l\zeta-z|- {\rm diam}(B)\geq e^{-\alpha l}-
6e^{-2\alpha l}\geq e^{-2\alpha l}.$ This implies $B\cap \mathcal
C^{([\alpha l])}=\emptyset$, and the claim holds. If
$f^l\zeta\notin I(\delta)$, then let $O=(0,0)$. If $l$ is large so
that $d(B,O)\geq|f^l\zeta-O|- {\rm diam}(B)\geq \delta-
2e^{-2\alpha l}\geq\delta/2$ holds, then the claim follows,
because $j>0$. If $l$ is so small that the last inequality does
not hold, then $f^l\zeta$ is near $fI(\delta)$, which is away from
$I(\delta)$.
\end{proof}

\noindent{\it Step 2: Construction of $\mathcal C^{(j)}$.}
 Let $\mathcal
Q^{(j-1)}$ denote any component of $\mathcal C^{(j-1)}$ which
intersects $\partial R_j$. By Corollary \ref{oo}, bound segments
in $\partial R_j$ do not intersect $\mathcal C^{(j-1)}$. Hence,
each component of $\mathcal Q^{(j-1)}\cap R_j$ is bounded by two
free segments stretching across $\mathcal Q^{(j-1)}$. The next
lemma ensures that it is possible to construct $\mathcal C^{(j)}$
so that (S2) (S3) hold.

\begin{lemma}\label{crea}
For any free segment $\gamma$ in $\partial R_j$ stretching across
$\mathcal Q^{(j-1)}$, there exists a critical point on $\gamma$
within $\mathcal O(b^{\frac{j}{4}})$ of the midpoint of $\gamma$.
\end{lemma}

\begin{proof}
By the closeness and the disjointness of the boundaries of
$\mathcal Q^{(j-1)}$, their tangent directions are close enough,
for Lemma \ref{update} to yield a critical approximation $\zeta_0$
of order $m_0:=j$ on $\gamma$, within $\mathcal
O(b^{\frac{j}{3}})$ of the midpoint of $\gamma$.

We inductively construct a sequence $\zeta_0,\zeta_1,\cdots,$ of
nice critical approximations on $\gamma$, of order
$m_0<m_1<\cdots$, such that: (a) $m_{i+1}\in[5m_i/4,20m_i)$; (b)
$|\zeta_i-\zeta_{i+1}|\leq (Cb)^{\frac{m_i}{2}}$. The limit of the
sequence $(\zeta_i)_i$ is a critical point with the desired
property.

Given $\zeta_i$ of order $m_i$ for some $i\geq0$, $\zeta_{i+1}$ is
constructed as follows. Let $\mu_1<\mu_2< \cdots$ denote an
infinite sequence of positive integers such that $\mu_{j+1}\leq
16\mu_{j}$ for $j=1,2,\cdots,$ and $\Vert
Df^{k-\mu_j}t(\zeta_i)\Vert\geq\kappa_0^{\frac{1}{4}(\mu_j-k)}\quad\text{for
} 0\leq k\leq\mu_j.$ Lemma \ref{htlem} ensures the existence of
such a sequence. Given $j(i)$ such that $\mu_{j(i)}\leq 20\theta
m_i<\mu_{j(i)+1}$, define $m_{i+1}$ to be the smallest such that
$[\theta m_{i+1}]=\mu_{j(i)}$. We have $\theta m_{i+1}\geq
\mu_{j(i)+1}/16\geq 5\theta m_i/4.$ (a) allows us to use Lemma
\ref{induce}, to create a critical approximation of order
$m_{i+1}$, denoted by $\zeta_{i+1}$. (b) is a consequence of Lemma
\ref{induce}.

Since $\gamma$ is a free segment,
$$|f^{-[\theta m_{i+1}]}\zeta_{i}-f^{-[\theta
m_{i+1}]}\zeta_{i+1}|\leq 10\delta (Cb)^{\frac{m_i}{2}}.$$ Lemma
\ref{curv}
implies, 
for $1\leq j\leq[\theta m_{i+1}]$,
$$\Vert Df^jt(f^{-[\theta m_{i+1}]}\zeta_{i+1})\Vert\geq\frac{1}{2}\kappa^{\frac{j}{4}}\geq
\kappa^{\frac{j}{3}}.$$ In other words, $t(f^{-[\theta
m_{i+1}]}\zeta_{i+1})$ is $\kappa_0^{\frac{1}{3}}$-expanding up to
time $[\theta m_{i+1}]$. Also, it is $1/10$-regular up to time
$[\theta m_{i+1}]$, because $\gamma$ is a free segment.
Consequently, (C3) in Sect.\ref{critical} holds and $\zeta_{i+1}$
is a nice critical approximation of order $m_{i+1}$ on $\gamma$.
This completes the construction of $(\zeta_i)_i$ and also the
proof of Lemma \ref{crea}.
\end{proof}

\noindent{\it Step 3: Verification of (I$)_j$.} To show the
assertion on the Hausdorff distance in (S1), we regard the
horizontal boundaries of the component of $\mathcal C^{(j-1)}\cap
R_j$ containing $\mathcal Q^{(j)}$ as graphs of functions
$\gamma_1$, $\gamma_2$ defined on an interval $I$ of length
$2\kappa_0^{j-1}$. Let $L(x)=|\gamma_1(x)-\gamma_2(x)|.$ (S1$)$
gives $L^{\frac{1}{2}}(x)\leq (Cb)^{\frac{j-1}{4}}<{\rm
length}(I).$ Moreover $|\gamma_1'(x)-\gamma_2'(x)|\leq
L^{\frac{1}{2}}(x)$ holds, for otherwise $\gamma_1$ intersects
$\gamma_2$. By this and the $C^2(b)$-property, $L(y)\geq L(x)
-(L^{\frac{1}{2}}(x)-C\sqrt{b}|x-y|)|x-y|$ holds for $x,y\in I$,
which is $\geq L(x)/2$ provided  $|x-y|\leq L^{\frac{2}{3}}(x)$.
Hence, ${\rm area}(\mathcal Q^{(j)})\geq L^{\frac{5}{3}}(x)/2$
holds. If $L(x)\geq b^{\frac{j}{2}},$ then ${\rm area}(\mathcal
Q^{(j)})\geq b^{\frac{5j}{6}}/2$, which yields a contradiction to
${\rm area}(\mathcal Q^{(j)})<{\rm area}(R_j)\leq (Cb)^j$.
\medskip

We show that $\partial R_0$ is controlled up to time $j$. Let
$\gamma$ denote any maximal free segment in $\partial R_j$
intersecting $I(\delta)$. We indicate how to choose the horizontal
curve $\tilde\gamma$.

If $\gamma\cap\mathcal Q^{(j-1)}\neq\emptyset$, then $\gamma$
stretches across a component $\mathcal Q^{(j-1)}$, and there
exists a critical point on $\gamma$, by Lemma \ref{crea}. In this
case, we take $\tilde\gamma=\gamma$. If $\gamma\cap\mathcal
Q^{(k-1)}=\emptyset$, let $k_0<k-1$ denote the largest such that
$\mathcal C^{(k_0)}\cap \gamma\neq\emptyset.$ Let $\mathcal
Q^{(k_0)}$ denote the component intersecting $\gamma.$ Let
$\mathcal Q^{(k_0+1)}$ denote any component of $\mathcal
C^{(k_0+1)}$ in $\mathcal Q^{(k_0)}$. Since the bound segments are
small, there exists a horizontal curve $\tilde\gamma$ which
contains $\gamma$ and a critical point on $\tilde\gamma$.
\medskip

\noindent{\it (Proof of (b)).} Let $\zeta\in\mathcal C$,
$2^{m+1}N<n_l\leq 2^{m+2}N$ and suppose that $n_l$ is a free
return time of $\zeta$. Let $\hat z_l$ denote the binding point of
order $\xi$, as in Lemma \ref{mega}. If $f^{-[\theta \xi]}\hat
z_l\notin fI(\delta)$, then the long stable leaf of order $[\theta
\xi]$ through $f^{-[\theta \xi]}\hat z_l$ intersects $\partial
R_0$ at one point, which we denote by $x$. Otherwise, the long
stable leaf of order $[\theta \xi]-1$ through $f^{-[\theta
\xi]+1}\hat z_{l}$ intersects $\partial R_0$ at one point, which
we denote by $x$. In either of the two cases,
$|f^{[\theta\xi]}x-\hat z_l|\leq(Cb)^{\theta \xi}$, and
$$\xi\leq C\alpha n_l<2^{m+1}N.$$
\begin{claim}
$f^{[\theta \xi]}x$ is free.
\end{claim}
\begin{proof}
Suppose the contrary. Let $B$ denote the bound segment containing
$f^{[\theta \xi]}x$, which is in $\partial R_{[\theta\xi]}$. By
Lemma \ref{bo}, $B\subset I(\delta)$ and there exists
$l<[\theta\xi]$, $z\in\Xi^{([\theta\xi]-1)}$ such that $f^lz$ is
free and $d(f^lz,B)\leq e^{-2\alpha l}$. Let $z'$ denote the
binding point for $f^lz$. It follows that $\zeta$ and $z'$ lie on
the same horizontal curve, a contradiction.
\end{proof}
Let $\gamma$ denote the maximal free segment containing
$f^{[\theta \xi]}x$. Lemma \ref{bo} implies that $\gamma$
stretches across $\mathcal Q^{([\theta \xi]-1)}$. By the
assumption of induction, there exists $z_l\in\Xi^{([\theta
\xi])}\subset\Xi^{([2^{m+1}\theta N])}$, located within $\mathcal
O(b^{\frac{[\theta\xi]}{4}})$ of the midpoint of
$\gamma\cap\mathcal Q^{([\theta\xi]-1)}$. By Lemma \ref{update},
there exists a critical approximation $z$ of order $\xi$ on
$\gamma$ such that $|f^{[\theta \xi]}x-z|= \mathcal O(
b^{\frac{\theta \xi}{2}})$. Lemma \ref{induce} implies
$|z_l-z|\leq (Cb)^{\frac{\theta \xi}{5}}$. Hence
 $|\hat z_l-z_l|\leq|\hat z_l-f^{[\theta\xi]}x|+
 |f^{[\theta\xi]}x-z|+|z-z_l|=\mathcal O(b^{\frac{\theta
 \xi}{5}}),$
 which means that $\zeta$ is controlled up to time $n_l$ by $\Xi^{([2^{m+1}\theta
N])}$. This completes the proof of Proposition \ref{geo}.

\subsection{Binding points in $\mathcal C$}
The following statement was obtained from the proof of Proposition
\ref{geo}.
\begin{cor}\label{cap}
For all $z\in\partial R_0\setminus\bigcup_{n\geq0}f^{-n}\mathcal
C$ there exist a sequence of integers $0\leq n_1<n_1+p_1\leq
n_2<n_2+p_2\leq\cdots$ and a sequence $\zeta_{1},\zeta_{2},\cdots$
of critical points such that for each $n_l$ we have: $f^{n_l}z\in
I(\delta)$; there exists a horizontal curve $\tilde\gamma$ which
contains the maximal free segment containing $f^{n_l}z$, and a
critical point $\zeta_{l}$ on $\tilde\gamma$; $p_l$ is the bound
period from the critical partition of $\tilde\gamma$.
\end{cor}

We use this corollary to resolve the problems mentioned in Remark
\ref{dig}, on the ambiguities of binding points. From this point
on, we call each $\zeta_l$ a {\it binding point} for the orbit of
$z$, and refer to $f^iz$ as {\it bound} if $n_l<i<n_l+p_l$ for
some $n_l$. Otherwise, we refer to $f^iz$ as {\it free}.



\section{The measure of $W^u\cap K^+$}
Let $|$ $\cdot$ $|$ denote the arc length measure on $W^u$ (we will also
denote by $|$ $\cdot$ $|$ the two-dimensional Lebesgue measure, but never
for both things simultaneously).
The aim of this section is to prove
\begin{prop}\label{zero}
$|W^u\cap K^+|=0.$
\end{prop}
The main step in the proof of this proposition is to show the next

\begin{lemma}{\rm (Growth to a fixed size)}\label{escape}
Let $\omega_0$ be an element of a critical partition constructed
in Section \ref{rec}, or a free segment not intersecting
$I(\delta)$ and stretching across one of the components of
$I(2\delta)\setminus I(\delta)$. If $\omega_0\cap K^+$ has
positive Lebesgue measure, there exist a collection $\mathcal Q$
of pairwise interior-disjoint curves in $\omega_0$ and a stopping
time function $S\colon\mathcal Q\to\mathbb N$ such that:
\smallskip

\noindent{\rm (a)} for a.e. $z\in \omega_0\cap K^+$, there exists
$\omega\in\mathcal Q$ containing $z$;

\noindent{\rm (b)} for each $\omega\in\mathcal Q$,
$f^{S(\omega)}\omega$ is a free segment not intersecting
$I(\delta)$ and stretching across one of the components of
$I(2\delta)\setminus I(\delta)$. The distortion of
$f^{S(\omega)}|\omega$ is uniformly bounded;

\noindent{\rm (c)} there exists $c>0$ depending only on the length
of $\omega_0$ such that for $n>0$,
\begin{equation}\label{exptail}|\{S>n\}|\leq ce^{-C n}.\end{equation}
Here, $\{S>n\}$ denotes the union of all $\omega\in\mathcal Q$
such that $S(\omega)>n$.
\end{lemma}

A large part of this section is devoted to the proof of Lemma
\ref{escape}. In Section \ref{srb} we define and describe the
combinatorics of the partition $\mathcal Q$ and the stopping time
$S$. In Section \ref{dre} we estimate the size of a curve with a
given combinatorics, and combine it with a counting argument, and
prove Lemma \ref{escape}.
 In Section \ref{sbrs} we show that stable manifolds with "good shapes"
are more or less dense. Combining this topological result with
Lemma \ref{escap} we complete the proof of Proposition \ref{zero}.

\subsection{Combinatorial structure}\label{srb}
Let $\omega_0$ be a free segment in $W^u$ as in Lemma
\ref{escape}. For each $n\geq0$, considering $n$-iterates we
construct a partition $\mathcal P_n$ of $\omega_0$, and its subset
$\mathcal Q_n$. Each element of $\mathcal P_n$ is a countable
union of elements of $\mathcal P_{n+1}$. Each element of $\mathcal
Q$ is an element of some $\mathcal Q_n$. If $\omega\in\mathcal
Q\cap\mathcal Q_n$, then $S(\omega)=n$ holds.

If $\omega_0$ is an element of a critical partition, let $p_0$
denote the bound period. Otherwise, namely $\omega_0\cap
I(\delta)=\emptyset$, let $p_0=0$. Let $n_1=\min\{n\geq p_0\colon
f^n\omega\cap I(\delta)\neq\emptyset\}$. For every $0\leq n< n_1$,
set $\mathcal P_n=\{\omega_0\}$, the trivial partition of
$\omega_{0}.$

Let $n\geq n_1$. Given $\omega\in\mathcal P_{n-1},$ $\mathcal
P_{n}|\omega$ is defined as follows. The $n$ is either {\it
cutting time} or {\it non-cutting time} of $\omega$. If $n$ is a
cutting time of $\omega$, $f^n\omega$ is cut into pieces. A
pull-back of this partition defines $\mathcal P _{n}|\omega$. If
$n$ is a {\it non-cutting time of $\omega$}, let $\mathcal
P_{n}|\omega=\{\omega\}$.

We precisely describe when $n$ is a cutting or non-cutting time of
$\omega$. If $f^n\omega\cap I(\delta)=\emptyset$, or $f^n\omega$
is bound, then $n$ is a {\it non-cutting time of $\omega$}. If
$f^n\omega\cap I(\delta)\neq\emptyset$ and $f^n\omega$ is free,
Let $\gamma$ denote the horizontal curve given by Corollary
\ref{cap}. Namely,$\gamma$ contains $f^n\omega$, and there exists
a critical point on $\gamma$. There are two mutually exclusive
cases:

\begin{itemize}
\item $\omega_0$ contains at least one element of the critical
partition $\{\gamma_{k,s}\}$ of $\gamma$. In this case, $n$ is a
{\it cutting time} of $\omega_0$. We cut $\omega\cap I(\delta)$
into pieces, by intersecting it with the elements of
$\{\gamma_{k,s}\}$. The partition elements containing the boundary
of $\omega\cap I(\delta)$ are taken together with the adjacent
ones, so that all the resultant elements contains exactly one
element of $\{\gamma_{k,s}\}$. If the component of
$\omega\setminus I(\delta)$ is $\geq\delta$ in length, then we
treat it as an element of our partition of $\omega$. Otherwise, we
take it together with the adjacent $\gamma_{k,s}$. Lemma
\ref{boun} goes through for each partition element, because it is
a union of at most $\log 2/(3\alpha)$- number of elements of the
critical partition of $\gamma$. This follows from (\ref{compte}).

\item  $\omega_0$ contains no element of
$\{\gamma_{k,s}\}$. In this case, $n$ is a non-cutting time of
$\omega$.
\end{itemize}

Let $\mathcal P_n'$ denote the collection of all elements of
$\mathcal P_n$ intersecting $K^+$.
\begin{lemma}\label{izo}
If $\omega\in\mathcal P_{n-1}'$ and $f^n\omega$ is free, then for
all $\xi,\eta\in\omega$,
$$\log\frac{\Vert Df^nt(\xi)\Vert}{\Vert Df^n
t(\eta)\Vert}\leq \frac{C}{\delta}|f^n\xi-f^n\eta|^{C\alpha}.$$ If
$f^n\xi,f^n\eta\in I(2\delta)$, then the factor $\delta$ can be
dropped.
\end{lemma}
\begin{proof}
Let $k<n$ and suppose that $f^k\omega$ is free. The time
interval $[k,n]$ is decomposed into bound and free segments.
Applying Proposition \ref{recovery0} to each bound segment and
Lemma \ref{outside} to each free segment, we have $\Vert
Df^{n-k}(z)t(z)\Vert\geq \delta e^{\frac{\lambda}{3}(n-k)}$ for
all $z\in f^{k}(\omega)$. Since $f^k\omega$ and $f^n\omega$
are $C^2(b)$, it then follows that
\begin{equation}\label{choy}|f^k\xi-f^k\eta|\leq
\delta^{-1}e^{-\frac{\lambda}{3}(n-k)}|f^n\xi-f^n\eta|.\end{equation}

Let $n_1<\cdots<n_s<n_{s+1}:=n$ denote all the free returns in the
first $n$-iterates of $\omega$, with $p_j$ the corresponding bound
period. By Lemma \ref{boun} and (\ref{choy}), $$\sum_{j=0}^s
\log\frac{\Vert Df^{p_j}t(f^{n_j}\xi)\Vert} {\Vert
Df^{p_j}t(f^{n_j}\eta)\Vert}\leq\sum_{j=0}^s
|f^{n_j+p_j}\xi-f^{n_j+p_j}\eta|^{C\alpha}\leq
\frac{C}{\delta^{C\alpha}}|f^n\xi-f^n\eta|^{C\alpha}.$$ By Lemma
\ref{outside}, $f^i\omega$ is a $C^2(b)$-curve outside of
$I(\delta)$, for $n_j+p_j\leq i<n_{j+1}$. Hence
\[\sum_{j=0}^{s}\sum_{i=n_j+p_j}^{n_{j+1}-1}\log
\frac{\Vert Dft(f^{i}\xi)\Vert}{\Vert Df t(f^{i}\eta)\Vert}
\leq\frac{C}{\delta}\sum_{j=0}^{s}\sum_{i=n_j+p_j}^{n_{j+1}-1}
|f^i\xi-f^i\eta|\leq\frac{C}{\delta}|f^n\xi-f^n\eta|.\] These two
inequalities yield the desired one. The last assertion follows
from (b) Lemma \ref{outside}.
\end{proof}

\subsection{Large deviation}\label{dre}
Let us say that $\omega\in\mathcal P_{n}'$ is an {\it escaping
element} if (i) $\omega\cap K^+\neq\emptyset$; (ii) $n$ is a
cutting time of the element of $\mathcal P_{n-1}'$ containing
$\omega$, and $f^n\omega\cap I(\delta)=\emptyset$. By
construction, if $\omega\in\mathcal P_{n}'$ is an escaping
element, then $f^n\omega$ is a free segment, not intersecting
$I(\delta)$ and stretching across one of the components of
$I(2\delta)\setminus I(\delta)$.

 Let
$\mathcal Q_n$ denote the collection of all escaping elements of
$\mathcal P_n'$ which are not contained in escaping elements in
$\bigcup_{0\leq k\leq n}\mathcal P_k'$. Define $\mathcal
Q=\bigcup_n\mathcal Q_n$. Define a stopping time function $S\colon
\mathcal Q\to\mathbb N$ by $S(\omega)=n$ for each
$\omega\in\mathcal Q_n$. Clearly, $\mathcal
P_n'\setminus\bigcup_{0\leq k\leq n}\mathcal
Q_k=\{\omega\in\mathcal Q\colon S(\omega)>n \}$ holds. We show
(\ref{exptail}). By construction, this would imply that the
elements of $\mathcal Q$ altogether cover $\omega_0\cap K^+$ up to
a set of zero Lebesgue measure, as desired in (a).

Let $n_1>0$ denote the cutting time of $\omega_0$. It is finite,
and depends only on the length of $\omega_0$. This implies that,
for $n\geq n_1$, any $\omega\in\mathcal
P_n'\setminus\bigcup_{0\leq k\leq n}\mathcal Q_k$ has an
well-defined itinerary that is described as follows. There exist a
sequence of integers $0< n_1<\cdots<n_s\leq n$ called {\it
essential free returns}, and an associated sequence
$\omega_{1}\supset\cdots\supset \omega_{s}\supset\omega$ such that
$\omega_{i}$ is the element of $\mathcal P_{n_i}'$ containing
$\omega,$ and $n_i$ is a cutting time of $\omega_{i-1}$, with
 $f^{n_i}
\omega_{i}\subset I(2\delta)$. Let $\zeta_{i}\in\mathcal C$ denote
the binding point for $f^{n_i}\omega_{i-1}$. Let $p_i$ denote the
bound period. By an {\it itinerary} of $\omega$ we mean the
sequence $(n_1,\pm p_1,\zeta_{1}),(n_2,\pm p_2,\zeta_{2}),
\cdots,(n_s,\pm p_s,\zeta_{s}),$ where $+,-$ indicates whether
$f^{n_i}\omega_{i}$ is at the right or left of $\zeta_{i}$.

From this point on we assume
\begin{equation}\label{stb}n\geq 2n_1.\end{equation}
By construction, $f^{n_{i}}\omega_{{i}}$ and
$f^{n_i+p_i}\omega_{i}$ are free segments.
 The following estimates are used in the proof:
$$|f^{n_{i}}\omega_{{i}}|\leq
e^{-\lambda p_{i}}\text{ and }|f^{n_i+p_i}\omega_{i}|\geq
e^{-4\alpha p_i}.$$ The first one follows from the definition of
the critical partition. The second one is from Lemma \ref{boun}.
Let $n_{s+1}>n$ denote the cutting time of $\omega_{s}$.
\begin{claim}\label{space0}
$n_{i+1}-n_i-p_i\leq \frac{20p_i}{\lambda}$ for every $1\leq i\leq
s$.
\end{claim}
\begin{proof}
Since $f^{n_{i+1}}\omega_{i}$ is also a free segment, in view of
Proposition \ref{recovery0} and Lemma \ref{outside} we have $3\geq
|f^{n_{i+1}}\omega_{i}|\geq \delta
e^{\frac{\lambda}{3}(n_{i+1}-n_i-p_i)}e^{-4\alpha p_i}.$
Rearranging gives $n_{i+1}-n_i-p_i\leq
\frac{3}{\lambda}\left(\log(1/\delta)+5\alpha p_i\right)
\leq\frac{20}{\lambda}p_i,$ where the last inequality follows from
$p_i\geq\frac{\log1/\delta}{2\log2}$.
\end{proof}
Summing the above
 inequality over all $1\leq i\leq s$ and then using (\ref{stb}), we have
\begin{equation}\label{Re}
n\leq \frac{40}{\lambda}\sum_{i=1}^sp_i.\end{equation} Write
$\omega=\omega_{s+1}$. Since $f^{n_{i+1}}\omega_{{i+1}}\subset
I(2\delta)$, the factor $\delta$ in Lemma \ref{izo} can be dropped
and
$$|\omega|\leq|\omega_s|=|\omega_{1}|\frac{|\omega_{2}|}{|\omega_{1}|}\cdots
\frac{|\omega_{s}|}{|\omega_{{s-1}}|}\leq 2^{s-1}\prod_{i=1}^{s-1}
\frac{|f^{n_{i+1}}\omega_{{i+1}}|}{|f^{n_{i+1}}\omega_{{i}}|} \leq
e^{-(\lambda-3\alpha)(p_1+\cdots+p_{s-1})}e^{-\lambda p_s}\leq
e^{-(\lambda-3\alpha)R},$$ where $R=\sum_{i=1}^s p_i$, which is
$\geq \frac{\lambda n}{40}$ by (\ref{Re}).
 Hence
$$\sum_{\omega\in\mathcal
P_n'\setminus\bigcup_{0\leq k\leq n}\mathcal
Q_k}|\omega|=\sum_{R}\sum_{\stackrel{\omega}{ p_1+\cdots+p_s=R}}|
\omega|\leq \sum_{R}\sum_s2^s\begin{pmatrix}
R\\s\end{pmatrix}e^{-(\lambda-3\alpha) R}\leq\sum_{R\geq \lambda
n/40} e^{-(\lambda-4\alpha) R}.$$ For the last inequality we have
used $s/R\leq C/\log(1/\delta)$ and
 $\left(\begin{smallmatrix}R\\s\end{smallmatrix}\right)\leq
e^{\beta(\delta)R}$, where $\beta(\delta)\to0$ as $\delta\to0,$
which follows from Stirling's formula for factorials. This
completes the proof of (\ref{exptail}) and that of Lemma
\ref{escape}. \qed

As a corollary we obtain

\begin{cor}{\rm (Abundance of stopping times)}\label{escap}
Let $\omega$ be an element of some critical partition. If
$\omega\cap K^+$ has positive Lebesgue measure, there exist a
sequence $\mathcal Q^{(1)},\mathcal Q^{(2)},\cdots$ of collections
of pairwise interior-disjoint curves in $\omega$, and a sequence
of stopping time functions $S_1,S_2\cdots,$ $S_k\colon\mathcal
Q^{(k)}\to\mathbb N$ such that:

\smallskip

\noindent {\rm (a)} for a.e. $z\in \omega\cap K^+$ there exists a
sequence $\omega^{(1)}\supset \omega^{(2)}\supset\cdots$ of curves
such that $\omega^{(k)}\in\mathcal Q^{(k)}$ for each $k\geq1$ and
 $\{z\}=\bigcap_{k\geq1}\omega^{(k)}$;

\noindent {\rm (b)}
 $0<S_1(\omega^{(1)})<S_2(\omega^{(2)})<\cdots$, and
$\log\frac{\Vert Df^{S_k(\omega^{(k)})}t (\xi)\Vert}{\Vert
Df^{S_k(\omega^{(k)})}t(\eta)\Vert}\leq C\delta^{-1}$ for all
$\xi,\eta\in \omega^{(k)}$;

\noindent {\rm (c)} $f^{S_k}\omega^{(k)}$ is a $C^2(b)$-curve,
stretching across one of the components of $I(2\delta)\setminus
I(\delta)$.
\end{cor}

\begin{proof}
Let $\mathcal Q$, $S$ be as in Lemma \ref{escape}, replacing
$\omega_0$ by $\omega$, and define $\mathcal Q_1=\mathcal Q$,
$S_1=S.$ Given $\mathcal Q_k$ and $S_{k}$, for each
$\omega\in\mathcal Q_k$ define a partition $\mathcal Q'$ of
$f^{S_{k}(\omega)}\omega$ and a stopping time function
$S'\colon\mathcal Q'\to\mathbb N$, replacing $\omega_0$ in Lemma
\ref{escape} by $f^{S_{k}(\omega)}\omega$. This defines $\mathcal
Q_{k+1}$ in the obvious way. For $\omega'\in\mathcal Q_{k+1}$,
define
$S_{k+1}(\omega')=S_{k}(\omega)+S'(f^{S_{k}(\omega)}\omega'),$ and
so on. The bounded distortion follows from Lemma \ref{izo}.
\end{proof}

\subsection{Proof of Proposition \ref{zero}}\label{sbrs}
The next lemma relies on a continuity argument within a small
parameter range containing the first bifurcation parameter $a^*$,
and is not valid for the parameter ranges treated in
 \cite{BC91,MV93,WY01}.
\begin{lemma}\label{intersec}
There exist $\varepsilon_1\in(0,a^*-a^{**})$ and
 $\sigma\in(0,1)$ such that for any $a\in [a^*-\varepsilon_1,
a^*]$ and any $C^2(b)$-curve $\gamma$ in $W^u$ stretching across
one component of $I(2\delta)\setminus I(\delta)$, $|\gamma\cap
K^+|\leq\sigma|\gamma|$.
\end{lemma}

We finish the proof of Proposition \ref{zero} assuming the
conclusion of the lemma. Assume $|W^u\cap K^+|>0$. Then one can
choose an element $\omega$ of some critical partition for which
$|\omega\cap K^+|>0$ holds. By Corollary \ref{escap} and Lemma
\ref{intersec}, for a.e. $z\in \omega\cap K^+$ there exists an
arbitrarily small neighborhood of $z$ in $W^u$ in which the set of
points which eventually escape from $R_0$ has a definite
proportion. It follows that $z$ is not a Lebesgue density point of
$\omega\cap K^+$.
 This yields a contradiction to
the Lebesgue density theorem.
\medskip

It is left to prove Lemma \ref{intersec}. The following elementary
observation is used, on the quadratic map
$g_2\colon[-1,1]\circlearrowleft$, $g_2(x)=1-2x^2$: $1/2$ is a
repelling fixed point, and the set of preimages $\bigcup_{n\geq0}
g_2^{-n}(1/2)$ is dense in $[-1,1]$, not containing $0$.

By a vertical curve we mean a curve such that the slopes of its tangent
directions are $\gg1$.
Let $l_0\subset W^s(Q)$ denote the segment in $W^u(P)$ which contains
$P$ and stretches across $R_0$. Clearly, $l_0$ is a vertical curve.
Iterating $l_0$ backward,
it is possible to choose
an integer $N_0$ independent of $b$, and to define a sequence
$l_{0},l_{1},\cdots,l_{N_0}$ of vertical curves in $W^s(P)$
which stretch across $R_0$, and with the property that
any $C^2(b)$-curve as in the statement of the lemma intersects
one of them in its middle third.
This
picture persists for all $a\in(a^{**},a^*)$ sufficiently close to
$a^*$. By the definition of $a^{**}$, $W^u(P)$ is not contained in
$[-2,2]^2$. By Inclination lemma, the conclusion holds. \qed

\section{Dynamics of Lebesgue typical points}
In this last section we show $\bigcap_{n\geq0} f^{-n}R_0$ has zero
Lebesgue measure, and completes the proof of the theorem. The main
step is a statistical argument, which enables us to show that {\it
the occurrence of infinitely many close returns is improbable}.
This sort of argument has been successfully undertaken by
Benedicks and Viana \cite{BV01} in the attractor context. We adapt
it to our non-attracting context, with the help of the geometric
structure of critical regions in Proposition \ref{geo}. In
addition, we dispense with any assumption on the Jacobian, which
was assumed in \cite{BV01,WY01}.

As a preliminary step, in Sect.\ref{controlp} we construct a
family long stable leaves near each critical point. In
Sect.\ref{shrink}, using these leaves we define a certain region,
and introduce {\it close return time}, as a kind of a first return 
time to this region. In Sect.\ref{clodec} we show
that the theorem follows from Proposition \ref{zeromeasure}, which
states that the occurrence of infinitely many close return times
is improbable.

For the proof of Proposition \ref{zeromeasure}, based on
preliminary geometric constructions in Sect.\ref{partition},
\ref{symb}, we construct in Sect.\ref{itin} an infinite nested
sequence $\Omega_0\supset\Omega_1\supset\cdots$. Each $\Omega_k$
is decomposed into rectangles, bordered by stable leaves and
pieces of $W^u$ and denoted by $R_{i_0\cdots i_k}$.  The sequence
$(i_0,\cdots,i_k)$ records the recurrent behavior of the iterates
of the rectangle to the critical set. Combining these geometric
ingredients with key analytic estimates in
Sect.\ref{uns},\ref{aread}, we complete the proof of Proposition
\ref{zeromeasure} in Sect.\ref{ch}.


\subsection{Construction of long stable leaves}\label{controlp}
For the purpose of stating the next proposition, we introduce a
{\it truncated distance $d_{\mathcal C}(\cdot)$ to $\mathcal C$}
as follows. Let $z\in W^u\setminus \bigcup_{n>0}f^n\mathcal C$ and
suppose that $z$ is free. If $z\notin I(\delta)$, then let
$d_{\mathcal C}(z)=1$. Otherwise, let $\zeta\in\mathcal C$ denote
the binding point for $z$ and let $d_{\mathcal C}(z)=|z-\zeta|$.
If $z$ is bound, then $d_{\mathcal C}(z)$ is undefined. For a free
segment $\omega$, let $d_{\mathcal
C}(\omega)=\min_{z\in\omega}d_{\mathcal C}(z)$.

The next proposition indicates the existence of a family of long
stable leaves near each critical value. In addition, these leaves
have a slow recurrence property to $\mathcal C$.

\begin{prop}\label{controlpoint}{\rm (Long stable leaves through slowly
recurrent points)} Let $\zeta$ be a critical point on a free
segment $\gamma$. For each element $\omega_0$ of the critical
partition of $\gamma$ there exists $z\in \omega_0$ such that
$d_{\mathcal C}(f^{n}z) \geq e^{-5\alpha n}$ holds for every $n>0$
such that $f^nz$ is free. In addition, the long stable leaf
through $fz$ exists.
\end{prop}

\begin{proof}
We divide the proof into three steps. First, we prove the
existence of $z\in\omega_0$ with the property as in the first
statement. Next, we give angle estimates. Finally, we show the
existence of long stable leaves through $fz$.
\medskip

\noindent{\it Step1. Construction of slowly recurrent points.} Let
$n_0=0$. Let $p_0$ denote the bound period of $\omega_0$. Let
$\mathcal P_0,\mathcal P_1,\mathcal P_2,\cdots$ denote the
sequence of partitions of $\omega_0$ constructed in the same way
as in Sect.\ref{srb}. We construct a (possibly finite) sequence
$p_0\leq n_1<n_2<\cdots$ and a nested sequence
$\omega_0\supset\omega_1\supset\omega_2\supset\cdots$ of curves
for which the following holds for every $k\geq0$. Obviously, any
point in the intersection $\bigcap_{k\geq0} \omega_k$ satisfies
the desired property:
\medskip

\noindent{$\bullet$} $\omega_k\in\mathcal P_{n_k}$, and for $0\leq
n\leq n_k$ such that $f^n\omega_k$ is free, $d_{\mathcal
C}(f^n\omega_k)\geq e^{-5\alpha n}$;

\noindent{$\bullet$} $n_{k+1}$ is a cutting time of $\omega_k$. If
there exists no cutting time of $\omega_k$, then $n_{k+1}$ is
undefined. 
\medskip

\noindent

The construction of the sequence is by induction that is described
as follows. Given $n_k$, $\omega_{k}\in\mathcal P_{n_k}$ such that
$f^{n_k}\omega_k\subset I(2\delta)$, with a bound period $p_k$,
define $n_{k+1}\geq n_k+p_k$ to be the cutting time of $\omega_k$.
Then $f^{n_{k+1}}\omega_k$ is a free segment of length $\geq
e^{-5\alpha n_{k+1}}$. Indeed, by Lemma \ref{boun},
$f^{n_k+p_k}\omega_k$ is a free segment of length $\geq
e^{-4\alpha p_k}$. Using Lemma \ref{outside} from time $n_k+p_k$
to $n_{k+1}$, $|f^{n_{k+1}}\omega_k|\geq |
f^{n_k+p_k}\omega_k|\geq  e^{-4\alpha n_{k+1}}.$ Hence, it is
possible to take an element $\omega_{k+1}\in\mathcal P_{n_{k+1}}$
such that $\omega_{k+1}\subset\omega_k$,
$f^{n_{k+1}}\omega_{k+1}\subset I(2\delta)$ and $d_{\mathcal
C}(f^{n_{k+1}}\omega_{k+1})\geq e^{-5\alpha n_{k+1}}$. To recover
the assumption of the induction, it suffices to show $d_{\mathcal
C}(f^n\omega_{k})\geq e^{-5\alpha n}$ for $n_k+p_k\leq n<n_{k+1}$
such that $f^n\omega_{k}$ is free. If $f^n\omega_k\cap
I(\delta)=\emptyset,$ then $d_{\mathcal C}(f^{n}\omega_{k})=1 \geq
e^{-5\alpha n}$. To treat the case where $n$ is a free return
time, we need
\begin{sublemma}\label{inessential}
Let $\tilde n_1<\cdots<\tilde n_s$ denote all the free return
times of $\omega_k$ in $[n_k+p_k,n_{k+1})$, with $\tilde
p_1,\cdots,\tilde p_s$ the corresponding bound periods. Then
$$\tilde p_1+\cdots+\tilde p_s\leq\frac{13\alpha p_k }{\lambda}.$$
\end{sublemma}

\begin{proof}
Splitting the time interval $[n_k+p_k,n_{k+1})$ into bound and
free segments, for all $z\in f^{n_k+p_k}\omega_k$ we have $\Vert
Df^{n_{k+1}-n_k-p_k}t(z)\Vert\geq  e^{\frac{\lambda}{3}(\tilde
p_1+\cdots+\tilde p_s)}.$ Combining this with
$|f^{n_k+p_k}\omega_k|\geq e^{-4\alpha p_k}$ from Lemma
\ref{boun}, we get $3>|f^{n_{k+1}}\omega_{k}|\geq
e^{\frac{\lambda}{3}(\tilde p_1+\cdots+\tilde p_s)-4\alpha p_k}.$
The first inequality is due to the elementary fact that the
forward iterates of $\omega_k$ cannot grow to a free segment of
length $>3$ without intersecting $I(\delta)$. Taking logs we
obtain the desired inequality.
\end{proof}

For each $\tilde n_i$ we have $d_{\mathcal C}(f^{\tilde
n_i}\omega)\geq e^{-\frac{\log C_0}{3}\tilde p_i}\geq
e^{-\frac{5\alpha \log C_0 }{\lambda}p_k}\geq e^{-5\alpha n}.$ The
last inequality follows from $p_k\leq\frac{3\alpha n_k }{\lambda}$
as in (a) Proposition \ref{recovery0}.
\medskip

\noindent{\it Step2. Angle estimates.} We introduce a useful
language along the way.

\begin{definition}
{\rm Let $z\in I(\delta)\setminus\mathcal C$. A tangent vector $v$
at $z$ is in {\it tangential position} relative to
$\zeta\in\mathcal C$ if there exists a horizontal curve $\gamma$
which is tangent to both $v$ and
$t(\zeta)$.}
\end{definition}

Let $z\in\omega_0$ have the property in Lemma \ref{controlpoint}.
Let $\theta_n= {\rm angle}(Df^{n}t(z),w_n(z))$. Let
$0=:n_0<n_1<n_2<\cdots$ denote all the free return times of $z$,
with $\zeta_0,\zeta_1,\zeta_2,\cdots$ the corresponding binding
points. The next lemma allows us to use $\zeta_k$ as a binding
point for $w_{n_k}(\zeta)$.
\begin{lemma}\label{behavior}
For every free return time $n_k>0$ of $z$, 
$\theta_{n_k}\leq (Cb)^\frac{n_k}{3}$ holds. In addition,
$w_{n_k}(z)$ is in tangential position relative to $\zeta_k$.
\end{lemma}

\begin{proof}
Let $p_k$ denote the binding period for $n_k$. The next three
angle estimates follow from [\cite{T1} Sublemma 3.2.]:
\begin{equation}\label{formu}\theta_{p_0}\leq \theta_{1}
(Cb)^{(p_0-1)/2} \frac{\Vert Dft(z)\Vert}{\Vert Df^{p_0}t(z)\Vert}
\frac{\|w_1(z)\|}{\Vert
w_{p_0}(z)\Vert}\leq(Cb)^{p_0/3};\end{equation}
\begin{equation}\label{formu1}\theta_{n_{k+1}}\leq \theta_{n_k+p_k}
(Cb)^{(n_{k+1}-n_k-p_k)/2} \frac{\Vert
Df^{n_k+p_k}t(z)\Vert}{\Vert Df^{n_{k+1}}t(z)\Vert} \frac{\Vert
w_{n_{k}+p_k}(z)\Vert}{\Vert w_{n_{k+1}}(z)\Vert}\quad {\rm for}\
k\geq0;\end{equation}
\begin{equation}\label{formu2}\theta_{n_{k}+p_k}\leq \theta_{n_k}
(Cb)^{p_k/2} \frac{\Vert Df^{n_k}t(z)\Vert}{\Vert
Df^{n_{k}+p_k}t(z)\Vert} \frac{\Vert w_{n_{k}}(z)\Vert}{\Vert
w_{n_{k}+p_k}(z)\Vert}\quad {\rm  for} \ k\geq1.\end{equation}
Using these, we prove the statement by induction on $k$. Take
$k=0$ in (\ref{formu1}). By (\ref{formu}) and Lemma \ref{outside},
the two fractions of the right-hand side are $\leq 1/\delta$ and
$\theta_{n_{1}}\leq (Cb)^{n_{1}/3}$ holds. This estimate and the
distance bound in Lemma \ref{controlpoint} implies that
$w_{n_1}(z)$ is in tangential position relative to $\zeta_1$.
Then, taking $k=1$ in (\ref{formu2}) we get
$\theta_{n_{1}+p_1}\leq(Cb)^{(n_{1}+p_1)/3} $. Taking $k=2$ in
(\ref{formu1}) we get $\theta_{n_{2}}\leq(Cb)^{n_{2}/3} $, and
that $w_{n_2}(z)$ is in tangential position relative to $\zeta_2$,
and so on.
\end{proof}

\noindent{\it Step 3. The existence of long stable leaves.} In
view of Lemma \ref{leaf}, it suffices to show that $fz$ is
expanding. In the next lemma, we assume
$1<\sigma<e^{\frac{\lambda}{4}}$.
\begin{lemma}\label{conduire}
For every $n\geq1$, $\Vert w_n(z)\Vert\geq \sigma^{n-1}.$
\end{lemma}
\noindent{\it Proof.} The inequality for $1\leq n\leq n(\lambda)$,
where $n(\lambda)<p_0$ only depends on $\lambda$, follows from the
closeness of our map $f$ to $(x,0)\to(1-2x^2,0)$. For
$n(\lambda)\leq n\leq p_0$, it follows from the exponential growth
and the bounded distortion along the orbit of $\zeta$. Let
$n>p_0$. Suppose $f^{n}z$ is free. Applying Lemma \ref{outside} to
each free segment and Proposition \ref{recovery} to each bound
segment, we have $\Vert w_n(z)\Vert\geq\delta
e^{\frac{\lambda}{3}(n-1)}$. If $n\leq n_1$, then the factor
$\delta$ can be dropped by Lemma \ref{outside} If $n>n_1$, then
using $\delta> d_{\mathcal C}(f^{n_1}z)\geq e^{-5\alpha n_1}$ we
have $\Vert w_n(z)\Vert\geq\delta e^{\frac{\lambda}{3}(n-1)}\geq
d_{\mathcal C}(f^{n_1}z)\Vert w_n(z)\Vert\geq
e^{(\frac{\lambda}{3}-5\alpha)(n-1)}\geq \sigma^{n-1}.$ If
$f^{n}z$ is bound, namely $n_k<n<n_k+p_k$ for some $n_k$, then
$\Vert w_n(z)\Vert\geq 5^{-(n_k+p_k-n)} \Vert
w_{n_k+p_k}(z)\Vert\geq
5^{-p_k}e^{\frac{\lambda}{3}(n_k+p_k-1)}\geq 5^{-\frac{3\alpha
n}{\lambda}}e^{\frac{\lambda}{3}n}\geq \sigma^{n-1}.$ For the
third inequality we have used $p_k\leq \frac{3\alpha
n_k}{\lambda}$ as in (a) Proposition \ref{recovery0}. This
completes the proof of Lemma \ref{conduire} and hence that of
Proposition \ref{controlpoint}.
\end{proof}

\subsection{Close return time}\label{shrink}
Let $\mathcal Q^{(k)}$ denote any component of $\mathcal C^{(k)}$.
Let $\zeta_0,\zeta_1$ denote the critical points on the horizontal
boundaries of $\mathcal Q^{(k)}$. Take curves $\gamma_0$,
$\gamma_1$ of length $\delta^{\frac{k}{10}}$ in the horizontal
boundaries of $\mathcal Q^{(k)}$ so that:  (i) $\gamma_0$ (resp.
$\gamma_1$) contains $\zeta_0$ (resp. $\zeta_1$) within $\mathcal
O(b^{\frac{k}{4}})$ of the midpoint of it; (ii) $\gamma_0$,
$\gamma_1$ are connected by two vertical lines. Let $\mathcal
B^{(k)}\subset\mathcal Q^{(k)}$ denote the region bordered by
$\gamma_0$ is connected to $\gamma_1$ by the two vertical lines
through their endpoints.

We construct a region $\mathcal B_0^{(k)}\subset\mathcal B^{(k)}$
as follows. Assume $\Gamma(f\zeta_0)$ is at the right of
$\Gamma(f\zeta_1)$. Choose a point $z\in\gamma_1$ for which
$\delta^{k}\leq |z-\zeta_1|\leq \delta^{\frac{k}{2}}$, and
$d_{\mathcal C}(f^{n}z)\geq e^{-5\alpha n}$ holds for every
$n\geq1$. Proposition \ref{controlpoint} ensures the existence of
such a point. By Remark \ref{verti},
$\Gamma(z)$ intersects $f\gamma_1$ exactly at two points.

By (\ref{leaf2}), the Hausdorff distance between $\Gamma(z)$ and
$\Gamma(f\zeta_0)$ is $\leq C|fz-f\zeta_1|+C
|f\zeta_1-f\zeta_0|\leq C\delta^{\frac{k}{2}}$. Hence,
$\Gamma(z)$ intersects $f\gamma_0$ at one point. 
By Remark \ref{verti}, $\Gamma(z)$ intersects
$f\gamma_0$ exactly at two points. Define $\mathcal B_0^{(k)}$
to be the region bordered by $\gamma_0,\gamma_1$ and the parabola
$f^{-1}\Gamma(z)$. By construction, the horizontal boundaries of
$\mathcal B_0^{(k)}$ extend both sides around $\zeta_0,\zeta_1$ to
length from $ \approx \delta^k$ to $\approx \delta^{\frac{k}{2}}$.
Let $\mathcal A^{(k)}$ denote the collection of all $\mathcal
B_0^{(k)}.$

\begin{definition}
{We say $z\in I(\delta)$ is {\it controlled up to time}
$\nu>0$ if $f^n z\notin \mathcal A^{(n)}$ holds for every $1\leq
n<\nu$. In addition, if $f^\nu z\in \mathcal A^{(\nu)},$ then we
say $z$ makes a {\it close return} at time $\nu$, and call $\nu$ a
close return time of $z.$}
\end{definition}

\subsection{Infinitely many close returns are improbable}\label{clodec}
Let $z\in I(\delta)$.
 Let
$\nu_1,\nu_2,\cdots$ be defined inductively as follows: $\nu_1$ is
a close return time of $z$; given $\nu_1,\cdots,\nu_k$, let
$\nu_{k+1}$ be the close return time of
$f^{\nu_1+\nu_2+\cdots+\nu_k}z\in I(\delta)$. If
$\nu_1,\cdots,\nu_k$ are defined in this way, we say $z$ {\it has
$k$ close return times.} If the sequence is defined is
indefinitely, we say $z$ has {\it infinitely close return times}.
Otherwise, we say $z$ {\it only finitely many close return times}.
We say $z$ is {\it controlled} if there is no close return time of
$z$.

Let $k_0$ be a large integer, to be specified later. Let
$\Omega_\infty$ denote the set of all $z\in\mathcal A^{(k_0)}$
which has infinitely many close return times. We have
$\Omega_\infty=\bigcap_{k\geq1}\Omega_k,$ where $\Omega_k$ denotes
the set of all $z\in\mathcal A^{(k_0)}$ which has $k$ close return
times. Obviously, $\Omega_k\subset\Omega_{k-1}$ holds.
\begin{prop}\label{zeromeasure}
$|\Omega_k|/|\Omega_{k-1}|\to0$ exponentially fast, as
$k\to\infty$. In particular, $\Omega_\infty$ has zero Lebesgue
measure.
\end{prop}

Let $\Lambda=\bigcap_{n\geq0}f^{-n}R_0.$ We show how
$\left|\Lambda\right|=0$ follows from this proposition. We argue
contradiction assuming $\left|\Lambda\right|>0.$ Lemma
\ref{outside} indicates that $\Lambda$ intersects
$\bigcup_{n\geq0}f^{-n}I(\delta)$ in a set with positive Lebesgue
measure. For almost every
$z\in\Lambda\cap\bigcup_{n\geq0}f^{-n}I(\delta)$, define
$m(z)\geq0$ to be the smallest such that $f^{m(z)}z$ is
controlled. Let us see $m(z)$ is well-defined. This is clear in
the case $z\notin\bigcup_{n\geq0}f^{-n}\mathcal A^{(k_0)}$.
Otherwise, take $i_0(z)\geq0$ such that $f^{i_0(z)}z\in\mathcal
A^{(k_0)}$. By Proposition \ref{zeromeasure}, one of the following
holds: either (i) $f^{i_0(z)}z$ is controlled, or else (ii)
$f^{i_0(z)}z$  has only finitely many close return times, denoted
by $\nu_1,\cdots,\nu_{k}$. By definition,
$f^{i_0+\nu_1+\cdots+\nu_{k}}z$ is controlled.

Let $V_j=\{z\in \Lambda\cap\bigcup_{n\geq0}f^{-n}I(\delta)\colon
m(z)=j\}.$ Take $j$ such that $|V_j|>0.$ By definition, any point
in $f^{j}V_j$ is controlled. The next lemma indicates that
$f^{j+1}V_j$ is foliated by long stable leaves.

\begin{lemma}\label{conduir} If $z\in I(\delta)$ is controlled
up to time $\nu$, then $\Vert w_n(z)\Vert\geq
\delta^{\frac{12n\log2}{\lambda}}$
holds for $1\leq n<\nu.$ 
\end{lemma}

\begin{proof}
We inductively define a sequence $0<n_1<n_1+p_1\leq
n_2<n_2+p_2\leq \cdots\leq n_s<n_s+p_s\leq\nu$ of integers and
critical points $\zeta_{1},\zeta_{2},\cdots,\zeta_s$ such that:
\noindent{\rm (i)} $f^{n_l}z\in I(\delta)$ for each $n_l$, and
$w_{n_l}(z)$ is in tangential position relative to $\zeta_l$, with
$p_l$ the bound period and $|f^{n_l}z-\zeta_{n_l}|\geq
\delta^{2n_l}$; \noindent{\rm (ii)} $n_{l+1}$ is the next time of
returns to $I(\delta)$ after $n_l+p_l$.

Given $n_l,\zeta_l$ and $p_l$, let $n_{l+1}\geq n_l+p_l$ denote
the smallest such that $f^{n_{l+1}}z\in I(\delta)$. By the
assumption, $f^{n_{l+1}}z\notin \mathcal A^{(n_{l+1})}$ holds. Let
$k$ denote the largest integer such that $f^{n_{l+1}}z\in\mathcal
C^{(k)}$, and let $\mathcal Q^{(k)}$ denote the component of
$\mathcal C^{(k)}$ containing $f^{n_{l+1}}z$. By (S3),
$f^{n_{l+1}}z$ is in tangential position relative to critical
points on the horizontal boundaries of the component of $\mathcal
C^{(k-1)}$ containing $\mathcal Q^{(k)}$. Choose one of them as
$\zeta_{l+1}$.

Suppose that $n_l<n<n_l+p_l$ holds. In the same way as in the
proof of Lemma \ref{conduire}, we have $\Vert w_n(z)\Vert\geq
4^{-p_l}$. Substituting $p_l\leq \frac{6n_l}{\lambda}\log
(1/\delta)$ into the exponent yields the desired inequality. For
all other $n$ it is immediate to show
the desired inequality, 
in the same way as
in the proof of Lemma \ref{conduire}. 
\end{proof}

Consider the projection $\pi\colon f^{j+1}V_j\to
\partial R_0$ along the long stable leaves. (\ref{leaf2})
says that $\pi$ is Lipschitz continuous. In particular,
$\pi(f^{j+1}V_j)$ has positive one-dimensional Lebesgue measure in
$W^u$. By the contraction along the leaves,
$\pi(f^{j+1}V_j)\subset K^+$ holds. This yields a contradiction to
Proposition \ref{zero}.

The rest of this paper is devoted to the proof of Proposition
\ref{zeromeasure}. Before proceeding, let us give some estimates
on close return times which will be used sometimes in the sequel.
Let $z\in\Omega_{k}$, and let $\nu_1,\cdots,\nu_{k}$ denote the
corresponding sequence of $k$ close return times of $z$. By
definition, for $1\leq l\leq k-1$, $f^{\nu_l}z\in \mathcal
A^{(\nu_l)}$ holds. Let $\zeta$ denote any critical point on the
horizontal boundary of the component of $\mathcal A^{(\nu_l)}$
containing $f^{\nu_l}z$. By definition, $|f^{\nu_l}z-\zeta|\leq
C\delta^{\frac{\nu_l}{2}}$ holds. Then
$$|f^{\nu_l+i}z-f^i\zeta|\leq C\delta^{\frac{\nu_l}{2}}4^{i}\ll
e^{-\alpha i}\quad{\rm for}\ 1\leq i\leq4\nu_l.$$ This implies
\begin{equation}\label{nu1}
\nu_{l+1}\geq 4\nu_{l}\quad {\rm for}\ 1\leq l<k.\end{equation} The
same reasoning gives $\nu_1\geq 4k_0$, and thus
\begin{equation}\label{nu}
\nu_{l}\geq 4^lk_0\quad{\rm for}\ 1\leq l\leq k.\end{equation}








\subsection{Partitions of rectangles}\label{partition}
By a {\it rectangle} $R$ we mean a compact region bounded by two
disjoint curves in $W^u$ and two disjoint stable leaves. The
boundaries of $R$ in $W^u$ are called {\it unstable sides}. The
boundaries in the stable leaves are called {\it stable sides}.

We define partitions of rectangles, using the families of long
stable leaves constructed in Section \ref{controlp}. To this end,
let us fix once and for all an enumeration $\mathcal
C=\{\zeta_m\}_{m=1}^\infty$ of all the critical points and let
$\gamma_m$ denote the maximal free segment containing $\zeta_m.$
We deal with a rectangle $R$ in $I(\delta)$ such that: \medskip

\noindent{\rm (R1)} the unstable sides of $R$ are made up of two
free segments, each contained in  $\gamma_{m_0}$ and
$\gamma_{m_1}$. In addition, $|\zeta_{m_0}-\zeta_{m_1}|\leq
(Cb)^{\frac{k}{2}}$ holds for some $k\geq1$;

\noindent{\rm (R2)} the unstable sides of $R$ extend to both sides
around $\zeta_{m_0},\zeta_{m_1}$ to length $\approx\delta^{k}$;

\noindent{\rm (R3)} $\Gamma(f\zeta_{m_0})$ is at the right of
$\Gamma(f\zeta_{m_1})$;

\noindent{\rm (R4)} there exists a long stable leaf
$\Gamma_\infty$ such that $f^{-1}\Gamma_\infty$ contains the
stable sides of $R$.
\medskip

One typical situation we have in mind is that two maximal free
segments in $\partial R_{\nu}$ stretch across $\mathcal
B_0^{(k)}$, where $k<\nu$. If this happens, then the region
bounded by the two maximal free segments and the stable sides of
$\mathcal B_0^{(k)}$ is a rectangle satisfying all the
requirements.

 By Lemma \ref{controlpoint}, in each element of the critical
partition of $\gamma_{m_1}$ there exists a point $z$ such that the
long stable leaf through $fz$ exists. Take just one such point
from each element of the partition and denote the associated
countable number of long stable leaves by $\Gamma_\Delta$,
$\Delta=-1,-2,-3,\cdots$ from the left to the right. We repeat
essentially the same construction for $\gamma_{m_0}$. The
difference is that, only those of the elements of the critical
partition of $\gamma_{m_0}$ come into play whose $f$-image is at
the right of $\Gamma(f\zeta_{m_1})$. We denote by $\Gamma_\Delta$
the associated countable number of long stable leaves at the right
of $\Gamma(f\zeta_{m_1})$, where $\Delta=1,2,3,\cdots$ from the
left to the right.

By Remark \ref{verti}, 
if $\Delta>0$, then $f^{-1}\Gamma_\Delta$
intersects the unstable side of $R$ containing $\zeta_{m_0}$
exactly at two points, one on the right of $\zeta_{m_0}$ and the
other on the left. 
If $\Delta<0$, then $f^{-1}\Gamma_\Delta$ intersects the stable
side of $R$ containing $\zeta_{m_0}$. By Remark \ref{verti} again, 
$f^{-1}\Gamma_\Delta$ intersects each of the unstable sides of $R$
exactly at two points. These observations and the Lipschitz
continuity of the tangent directions of the leaves as in
(\ref{leaf2}) altogether indicate that, the family of the long
stable leaves induces a partition of $R$. Each element of the
partition is a rectangle, bounded by the unstable sides of $R$ and
two neighboring parabolas, which are preimages of
$\Gamma_{\Delta}$, $\Gamma_{\Delta+1}$.

\begin{figure}
\begin{center}
\input{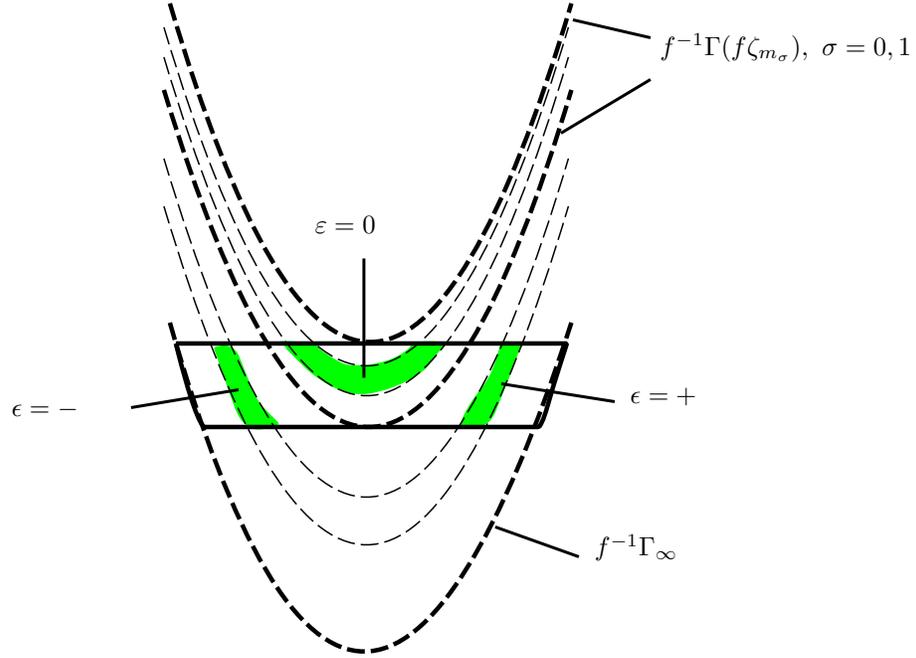}
\caption{A rectangle and its partition with long stable leaves}
\end{center}
\end{figure}

\subsection{Symbolic coding}\label{symb}
Each rectangle in the partition of $R$ constructed in Section
\ref{partition} is denoted by $R(\rho,\epsilon,\Delta,p)$. Here,
the meanings of $\rho,\epsilon,\Delta,p$ are as follows:
\medskip

\noindent$\bullet$ if the unstable sides of
$R(\rho,\epsilon,\Delta,p)$ intersect both $\gamma_{m_0}$ and
$\gamma_{m_1}$, then $\rho=m_1.$ Otherwise, $\rho=m_0;$

\noindent$\bullet$ if $\rho=m_0$, then $\epsilon=0$. If
$\rho=m_1$, then $\epsilon=+$ or $-$, depending on whether the
unstable sides of $R(\rho,\epsilon,\Delta,p)$ is at the ``right''
or the ``left'' of $\zeta_{m_0}$ and $\zeta_{m_1}$;

\noindent$\bullet$ the stable sides of
$f(R(\rho,\epsilon,\Delta,p))$ are contained in
$\Gamma_\Delta\cup\Gamma_{\Delta+1}.$

\noindent$\bullet$ $p=\max\{p(\zeta_{\rho},z)\colon
z\in\gamma_{\rho}\cap R(\rho,\epsilon,\Delta,p)\}.$
\medskip

The integer $p$ is called a {\it bound period} of
$R(\rho,\epsilon,\Delta,p)$.
 By the monotonicity of the function $z\mapsto p(\zeta_{\rho},z)$, the maximum is attained
 at one of the edges of the rectangle $R(\rho,\epsilon,\Delta,p).$
It is immediate to see:

\smallskip

(i) all points in $f(R(\rho,\epsilon,\Delta,p))$ are expanding up
to time $p_{i_0}-1$;
\smallskip

(ii) for all $\xi,\eta\in R(\rho,\epsilon,\Delta,p)$ and  $1\leq
i\leq p,$ $\Vert w_{i}(\xi)\Vert/\Vert w_{i}(\eta)\Vert\leq2$.

\begin{lemma}{\rm (Geometry of rectangles at the end of bound periods)}
\label{behavior} For all $z$ in the unstable sides of
$R(\rho,\epsilon,\Delta,p)$, $\Vert Df^{p}t(z)\Vert\geq
C\delta\Vert Df^{i}t(z)\Vert$ holds for every $0\leq i<p$. In
particular, the unstable sides of
$f^{p}(R(\rho,\epsilon,\Delta,p))$ are made up of two
$C^2(b)$-curves.
\end{lemma}


\begin{proof}
Let $\zeta$ denote the critical point on the unstable side of $R$
which contains $z.$ Let $p(\zeta,z), q(\zeta,z)$ denote the bound
and fold periods of $z$ with respect to $\zeta,$ as defined in
Sect.\ref{rec}. In view of (ii) as above and (g) Proposition
 \ref{recovery0},
\begin{equation}\label{parcourir}\Vert
Df^{i}t(z)\Vert\approx |\zeta-z| \cdot \|w_i(\zeta)\|\ \ {\rm for
}\ \ q(\zeta,z)\leq i\leq\max(p(\zeta,z),p).\end{equation}

Let $\xi_1,\xi_2,\xi_3,\xi_4$ denote the {\it edges} of the
rectangle, namely, the points which belong to both the stable and
the unstable sides of $R(\rho,\epsilon,\Delta,p).$ In the
discussion to follow, we assume that $\xi_1,\xi_2$ are on the same
unstable side of $R$, and $f\xi_{i},f\xi_{i+2}$ $(i=1,2)$ are
connected by the long stable leaf which defines the stable side of
$f(R(\rho,\epsilon,\Delta,p))$. \medskip

\noindent{\it Case 1: $\epsilon=0$.} In this case,
$\xi_1,\xi_2,\xi_3,\xi_4$ are on the same unstable side of $R.$ We
suppose that $\xi_1$ is closest to $\zeta$. Then
$p=\max(p(\zeta,\xi_1), p(\zeta,\xi_3))$ holds.
(\ref{leaf2}) and Lemma \ref{quadratic} give
$|\zeta-\xi_1|\approx|\zeta-\xi_3|.$ Hence, (a,b) Proposition
\ref{recovery0} gives
$$q(\zeta,z)\leq C\beta\max(
\log|\zeta-\xi_1|^{-1},\log|\zeta-\xi_3|^{-1})< p.$$ This means
that (\ref{parcourir}) holds for $q(\zeta,z)\leq i\leq p$ and
therefore
$$\frac{\Vert Df^{p}t(z)\Vert}{\Vert Df^{i}t(z)\Vert}\geq
C\frac{\|w_{p}(\zeta)\|} {\|w_{i}(\zeta)\|}\geq C\delta.$$ For
$1\leq i\leq q(\zeta,z)$,
 \begin{align*}\frac{\Vert Df^{p}t(z)\Vert}{\Vert
Df^{i}t(z)\Vert}&\geq\Vert Df^{p}t(z)\Vert \geq
C\delta|\zeta-z|\|w_{p(\zeta,z)}(\zeta)\| \geq
C\delta^{\frac{\alpha}{\log C_0}}>\delta.\end{align*} The first
inequality follows from (h) Proposition \ref{recovery0}. The
second inequality follows from $\|w_{p}(\zeta)\|\geq C\delta
\|w_{p(\zeta,z)}(\zeta)\|$. For the third inequality we have used
$|\zeta-z|\|w_{p(\zeta,z)}(\zeta)\|\geq|\zeta-z|^{-1+\frac{\alpha}{\log
C_0 }} \geq\delta^{-1+\frac{\alpha}{\log C_0}}$ which follows from
(e) Proposition \ref{recovery0}. Since the unstable sides of
$R(\rho,\epsilon,\Delta,p)$ are $C^2(b)$, these two inequalities
and the curvature estimate in [\cite{T1} Lemma
 2.3] together imply that
 the unstable sides of
$f^{p}(R(\rho,\epsilon,\Delta,p))$ are
$C^2(b)$.
\smallskip

\noindent{\it Case 2: $\epsilon=+$ or $-$.} In this case, $\xi_1$
and $\xi_3$ (resp. $\xi_2$ and $\xi_4$) are on different unstable
sides of $R$. We suppose that $\Gamma(f\xi_1)$ is at the right of
$\Gamma(f\xi_2)$, and that $\xi_1$ and $\zeta$ belong to the same
unstable side of $R.$ Let $\zeta'$ denote the other critical point
of $R$ on the unstable side of $R.$ Then
$p=\max\{p(\zeta,\xi_1),p(\zeta',\xi_3)\}$ holds. By (\ref{leaf2})
and Lemma \ref{quadratic} again, $|\zeta-z|\geq C|\zeta-\xi_1|$
and $|\zeta-z|\geq C|\zeta'-\xi_3|.$ Hence, $q(\zeta,z)\leq
C\beta\log|\zeta-z|^{-1}< p.$ This means that (\ref{parcourir})
holds for $q(\zeta,z)\leq i\leq p$. The rest of the argument is
analogous to that in Case 1.
\end{proof}

\begin{figure}[t]
\begin{center}
\input{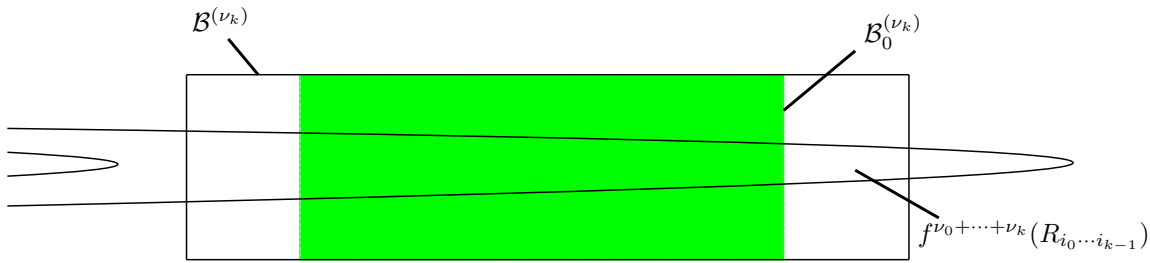}
\caption{Situation considered in Proposition \ref{rect2}}
\end{center}
\end{figure}


\subsection{Construction of partitions}\label{itin}
Putting the results in Sections \ref{partition}, \ref{symb}
together, for each $k\geq0$ we inductively construct a partition
of each $\Omega_k$ into a countable number of rectangles. Each
element of the partition of $\Omega_k$ will be denoted by
$R_{i_0\cdots i_{k}}$, where $(i_0,\cdots,i_{k})$ are itineraries
which record the behavior of the rectangle under iteration, up to
time $\nu_1+\cdots+\nu_k$.
\medskip

\noindent{\it Initial step}. Let $\Omega_0={\mathcal A}^{(k_0)}$.
Take a component of $\Omega_0$ and denote it by $R$. Following the
steps in Sect.\ref{partition}, define a partition of $R$ with the
family of long stable leaves. To each element of the partition,
assign the set of symbols according to the rule described in
Sect.\ref{symb}. Each element is denoted by $R_{i_0}$, where
$i_0=(\rho_0,\epsilon_0,\Delta_0,p_0)$ and
$R_{i_0}=R(\rho_0,\epsilon_0,\Delta_0,p_0)$. We repeat the same
construction for any component of $\Omega_0$.
\medskip

\noindent{\it General step}. Given the partition $\{R_{i_0\cdots
i_{k-1}}\}_{i_0,\cdots,i_{k-1}}$ of $\Omega_{k-1}$ for some
$k\geq1$, define $R_{i_0\cdots i_{k-1}}(\nu_k)= \{ z\in
R_{i_0\cdots i_{k-1}} :\text{$\nu_k$ is a close return time of
$f^{\nu_0+\cdots+\nu_{k-1}}z$}\}.$ Here and for the rest of this
section we adopt the next
\smallskip

\noindent{\it Convention.} $\nu_0=0$.
\smallskip

\noindent By definition,
$$\Omega_{k}=\bigcup_{(i_0,\cdots,i_{k-1})}\bigcup_{\nu_{k}}
R_{i_0\cdots i_{k-1}}(\nu_{k}).$$

\begin{prop}\label{rect2}{\rm
(Geometry of rectangles at close return times)} Let $z\in
f^{\nu_0+\cdots+\nu_{k-1}}R_{i_0\cdots i_{k-1}}$ and suppose
$f^{\nu_k}z\in\mathcal B_0^{(\nu_k)}\subset\mathcal B^{(\nu_k)}$.
Then the unstable sides of $f^{\nu_1+\cdots+\nu_k}R_{i_0\cdots
i_{k-1}}\cap \mathcal B^{(\nu_k)}$ are $C^2(b)$-curves stretching
across $\mathcal B^{(\nu_k)}$.
\end{prop}

We finish the construction of the partition of $\Omega_k$ assuming
the conclusion of the proposition. Take a component of
$f^{\nu_1+\cdots+\nu_{k}}R_{i_0\cdots i_{k-1}}(\nu_{k})$ and
denote it by $R$.
 By the proposition and
the geometric structure of critical regions in Proposition
\ref{geo}, on each unstable side of $R$ there exists a critical
point, within $\mathcal O(b^{\frac{\nu_k}{8}})$ of its midpoint.
In particular, $R$ meets all the requirements (R1-4) in
Sect.\ref{partition}.  Following the steps in
Sect.\ref{partition}, \ref{symb}, define a partition of $R$ with
the family of long stable leaves and assign to each element the
set of symbols. Let $R_{i_0\cdots i_{k-1}i_k}=
f^{-(\nu_1+\cdots+\nu_{k})}R(\rho_k,\epsilon_k,\Delta_k,p_k)$,
where $i_k=(\rho_k,\epsilon_k,\Delta_k,p_k,\nu_k)$. We repeat the
same construction for any component of
$f^{\nu_1+\cdots+\nu_{k}}R_{i_0\cdots i_{k-1}}(\nu_{k})$.

\medskip

\noindent{\it Proof of Proposition \ref{rect2}.} Let
$\Gamma_{\nu_k-1}(z)=\{(x(y),y)\colon |y|\leq\sqrt{b}\}$. Consider
the vertical strip $V=\{(x,y)\colon |x-x(y)|\leq
\delta^{\frac{\nu_k}{20}}, |y|\leq\sqrt{b}\}$.

\begin{lemma}\label{park}
$V$ does not intersect the stable sides of
$f^{\nu_0+\cdots+\nu_{k-1}+1}R_{i_0\cdots i_{k-1}}$.
\end{lemma}
\begin{proof}
Let $\sigma$ denote any stable side of
$f^{\nu_0+\cdots+\nu_{k-1}+1}R_{i_0\cdots i_{k-1}}$. By
construction, there exists $y\in W^u\cap\sigma$ such that
$d_\mathcal C(f^{n}y)\geq e^{-5\alpha n}$ holds whenever $f^ny$ is
free, and $\sigma\subset\Gamma(y)$.
 Suppose $V\cap \sigma\neq\emptyset$, and let $\xi\in
V\cap \sigma$. Let $\eta$ denote the point of intersection between
$\Gamma$ and the horizontal through $\xi$. The definition of $V$
gives $|\xi-\eta|\leq \delta^{\frac{\nu_k}{20}}$, and thus
$|f^{\nu_k-1}\eta-f^{\nu_k-1}\xi|\leq \delta^{\frac{\nu_k}{21}}$.
Since $\eta\in\Gamma$, $|f^{\nu_k}z-f^{\nu_k-1}\eta|\leq
(Cb)^{\nu_k-1}$ holds. Hence $|f^{\nu_k}z-f^{\nu_k-1}\xi|\leq
(Cb)^{\frac{\nu_k}{2}}$ follows. Meanwhile
$|f^{\nu_k-1}\xi-f^{\nu_k-1}y|\leq (Cb)^{\nu_k-1}$ holds, and the
assumption on $z$ gives $|\zeta-f^{\nu_k}z|\leq
C\delta^{\frac{\nu_k}{2}}$, where $\zeta$ is any critical point on
the unstable sides of $\mathcal B_0^{(\nu_k)}$. Therefore
$$|\zeta-f^{\nu_k-1}y|\leq|\zeta-f^{\nu_k}z|+
|f^{\nu_k}z-f^{\nu_k-1}\xi|+|f^{\nu_k-1}\xi- f^{\nu_k-1}y|\leq
\delta^{\frac{\nu_k}{22}}.$$ This estimate and the proof of
Corollary \ref{oo} together indicate that $f^{\nu_k-1}y$ is free.
Hence, Proposition \ref{controlpoint} gives a critical point
$\zeta'$ such that $|\zeta'-f^{\nu_k-1}y|\geq e^{-5\alpha\nu_k}.$
Then it is possible to choose a horizontal curve $\gamma$ such
that both $\zeta$ and $\zeta'$ are on $\gamma.$
This is a contradiction.
\end{proof}

By Lemma \ref{park}, $V$ cuts a segment in each unstable side of
$f^{\nu_0+\cdots+\nu_{k-1}+1}R_{i_0\cdots i_{k-1}}$, denoted by
$\gamma$. Let $\zeta'$ denote the critical point on the same
unstable side of $f^{\nu_0+\cdots+\nu_{k-1}}R_{i_0\cdots i_{k-1}}$
as that of $f^{-1}\gamma.$ Let $z'$ be an arbitrary point in
$\gamma.$ Let $p_{k-1}$ denote the bound period of
$f^{\nu_0+\cdots+\nu_{k-1}}R_{i_0\cdots i_{k-1}}$. The bounded
distortion gives $\Vert Df^{j}(fz)\Vert\approx \Vert
Df^{j}(z')\Vert\approx\Vert w_{j}(z)\Vert$ for $1\leq j<\nu_k,$
and thus for $p_{k-1}-1\leq j<\nu_k$,
\begin{equation}\label{er2} \Vert
Df^jt(z')\Vert\approx|\zeta'-f^{-1}z'|\cdot\Vert Df^j(z')\Vert
\approx |\zeta'-f^{-1}z'|\cdot\Vert w_j(z)\Vert.
\end{equation}
By Lemma \ref{behavior}, $f^{p_{k-1}-1}\gamma$ is $C^2(b)$. Then,
by [\cite{T1} Lemma 2.3.] and (\ref{er2}), the curvature of
$f^{\nu_k-1}\gamma$ is everywhere bounded from above by
$$(Cb)^{\nu_k-p_{k-1}}
\frac{\Vert w_{p_{k-1}}(z)\Vert^3} {\Vert
w_{\nu_k}(z)\Vert^3}\sqrt{b}+
\sum_{j=p_{k-1}}^{\nu_k-1}(Cb)^{\nu_k-j-1} \frac{\Vert
w_{j+1}(z)\Vert^3}{\Vert w_{\nu_k}(z)\Vert^3}.$$ Since $\Vert
w_{\nu_1}(z)\Vert\geq C\delta\Vert w_{j+1}(z)\Vert$ for
$p_{k-1}\leq j<\nu_k,$ it follows that the curvature is everywhere
$\leq\sqrt{b}$. (\ref{er2}) also implies that the slopes of the
tangent directions of $f^{\nu_k-1}\gamma$ are $\leq\sqrt{b}$.
Hence, $f^{\nu_k-1}\gamma$ is a $C^2(b)$-curve.

Parametrize $\gamma$ by arc length $s$. Using
$|\zeta-f^{-1}\gamma(s)|\geq C^{-p_{k-1}}$ for all $s$ and the
fact that the width of the strip $V$ is
$\delta^{\frac{\nu_k}{20}}$,
\begin{align*}
&\int \Vert Df^{\nu_k-1}t(\gamma(s))\Vert ds\geq C\Vert
w_{\nu_k}(z)\Vert\int|\zeta-f^{-1}(\gamma(s))| ds\geq
C^{-\nu_{k}}\delta^{\frac{\nu_k}{20}}\gg
\delta^{\frac{\nu_k}{10}}.
\end{align*}
This implies that $f^{\nu_k-1}\gamma$ stretches across $\mathcal
B^{(\nu_k)}$. \qed

\subsection{Unstable sides are roughly parallel}\label{uns}
A main step in the proof of Proposition \ref{zeromeasure} is an
estimate of the measure of the set
$$R_{i_0\cdots
i_{k-1}}(\nu_k)=\{z\in R_{i_0\cdots i_{k-1}}\colon\text{$\nu_k$ is
a close return time of $f^{\nu_0+\cdots+\nu_{k-1}}z$}\}.$$ This
subsection and the next are devoted to obtaining this estimate.
For the purpose of stating the next proposition we need some
definitions.
\medskip

\noindent $\bullet$ (New constants) Choose $C_1$, $C_2$ as
follows: $|\det Df|\geq C_1\text{ on }R_0;$ for all $\xi,\eta$ in
the unstable sides of any component of $\Omega_0$,
${\rm angle}(u(\xi),u(\eta))\leq C_2|\xi-\eta|.$ 
Let $C_3=C_0e^{\frac{6}{\log C_0}}$.
\medskip

\noindent $\bullet$ (Attachment of collars) 
For each $R_{i_0}\subset\Omega_0$, let $Q(R_{i_0})$ denote the
component of $\Omega_0$ containing $R_{i_0}$. Let $k\geq1$. For
each $R_{i_0\cdots i_{k}}\subset\Omega_k$,
 By Proposition
\ref{rect2}, there exists exactly one component $\mathcal
B^{(\nu_k)}$ of $\mathcal A^{(\nu_k)}$ containing
$f^{\nu_1+\cdots+\nu_k}R_{i_0\cdots i_{k}}$. Let $Q(R_{i_0 \cdots
i_{k}})$ denote the component of
$f^{-(\nu_1+\cdots+\nu_k)}\mathcal B^{(\nu_k)}\bigcap R_{i_0
\cdots i_{k-1}}$ containing $R_{i_0\cdots i_{k}}$.
\medskip



\noindent $\bullet$ For any $z$ in a free segment of $W^u$, let
$u(z)$ denote the unit vector tangent to $W^u$ at $z$ such that
the sign of the first component is positive.

\begin{prop}\label{lipschitz}
For every $j\geq 0$ and any $\xi$, $\eta$ in the unstable side of
$ f^{\nu_0+\cdots+\nu_{j}}Q(R_{i_0\cdots i_{j}})$,
\begin{equation}\label{lipeq}{\rm angle}(u(\xi),u(\eta))\leq
C_2C_3^{3\nu_{j}}|\xi-\eta|.\end{equation}
\end{prop}


\noindent{\it Proof of Proposition \ref{lipschitz}.}
We argue by
induction on $j$. The choice of $C_2$ and the convention $\nu_0=0$
give (\ref{lipeq}) for $j=0$. Let $k\geq1$ and assume
(\ref{lipeq}) for $j=k-1$.
\begin{lemma}\label{liplem}
For any $\xi$, $\eta$ in the unstable sides of
$f^{\nu_0+\cdots+\nu_{k-1}+1}Q(R_{i_0\cdots i_{k-1}})$,
$${\rm angle}(Dfu(\xi), Dfu(\eta))
\leq C_2C_3^{\nu_{k}}|f\xi-f\eta|.$$
\end{lemma}
\begin{proof}
Let $\theta_i={\rm angle}(Df^iu(\xi), Df^iu(\eta))$, $i=0,1.$ A
simple computation gives
$$\theta_1
\leq \frac{ Cb\theta_0+C|\xi-\eta|}{\|Dfu(\xi)\|\|Dfu(\eta)\|}.$$
Hence $\theta_1\ll1,$ provided $k_0$ is sufficiently large. We
have
\begin{align*}\theta_1
\leq CC_1^{-1}\left(|\xi-\eta|+{\rm
angle}(u(\xi),u(\eta))\right).\end{align*} The inequality follows from
the following elementary fact: for any nonzero vectors $u$, $v$
such that ${\rm angle}(u,v)\ll1$, ${\rm angle}(u,v)\leq
2|u-v|/\min\{\Vert u\Vert,\Vert v\Vert\}.$ (\ref{lipeq}) with
$j=k-1$ and $|\xi-\eta|\leq C_1^{-1}|f(\xi)-f(\eta)|$ give
\begin{align*}\label{induc}
|\xi-\eta|+{\rm angle}(u(\xi),u(\eta))\leq
2C_1^{-1}C_2C_0^{3\nu_{k-1}}|f\xi-f\eta|.\end{align*} Replacing
this in the previous inequality,
\begin{align*}\theta_1\leq
CC_1^{-2}C_2C_3^{3\nu_{k-1}}|f\xi-f\eta|\leq
C_2C_3^{\nu_{k}}|f\xi-f\eta|.\end{align*} The last inequality
holds for sufficiently large $k_0$, because of
$C_3^{4\nu_{k-1}}\leq C_3^{\nu_{k}}$ from (\ref{nu1}).
\end{proof}

For any $\xi$ on the unstable sides of
$f^{\nu_0+\cdots+\nu_{k-1}+1}Q(R_{i_0\cdots i_{k}})$, let
\begin{equation}\label{va}v(\xi)= \rho\cdot
Dfu(f^{-1}\xi),\end{equation} where $\rho>0$ is the normalizing
constant.
 If $k_0$ is
sufficiently large, then $v(\xi)$ has a large slope. By the
definition of $u(\cdot)$, the sign of the second component of
$v(\xi)$ is constant for all $\xi$.

By Proposition \ref{rect2} and the distortion control, the
contractive field $e_{\nu_k-1}$ is well-defined on
$f^{\nu_0+\cdots+\nu_{k-1}+1}Q(R_{i_0\cdots i_{k}})$. Fix once and
for all the orientation of $e_{\nu_k-1}$ so that the second
component of $e_{\nu_{k}-1}$ and that of $v(\xi)$ have the same
sign. Let $f_{\nu_{k}-1}$ denote the unit vector field orthogonal
to $e_{\nu_{k}-1}$. Split $v(\xi)=A(\xi)e_{\nu_{k}-1}(\xi)+
B(\xi)f_{\nu_{k}-1}(\xi).$

\begin{lemma}\label{eqlip0}
For any $\xi_1,\xi_2$ on the unstable sides of
$f^{\nu_0+\cdots+\nu_{k-1}+1}Q(R_{i_0\cdots i_{k}}),$
$$\max\{|A(\xi_1)-A(\xi_2)|,|B(\xi_1)-B(\xi_2)|\}\leq
2C_2C_3^{\nu_{k}}|\xi_1-\xi_2|.$$
\end{lemma}
\begin{proof}
The following elementary fact is used. For
$u_i=\left(\begin{smallmatrix}\cos\theta_i\\
\sin\theta_i
\end{smallmatrix}\right)$, $0\leq\theta_i\leq\pi$, $i=1,2,3,4$,
$$|{\rm angle}(u_1,u_2)-{\rm angle}(u_3,u_4)|
\leq {\rm angle}(u_1,u_3)+{\rm angle}(u_2,u_4).$$
This can be checked using ${\rm angle}(u_i,u_j)=|\theta_i-\theta_j|$
and the triangle inequality.

We have $A(\xi_i)=\langle e_{\nu_{k}-1}(\xi_i),v(\xi_i)\rangle
=\cos\left({\rm angle}(e_{\nu_{k}-1}(\xi_i),v(\xi_i))\right),$
 where
the bracket denotes the scholar product. Then ${\rm
angle}(e_{\nu_{k}-1}(\xi_i),v(\xi_i))\in[0,\pi]$ holds, which is
close to $0$. Considering $\cos^{-1}\colon[-1,1]\to[0,\pi]$ and
$|(\cos^{-1})'|\geq1$ we have $|A(\xi_1)-A(\xi_2)|\leq
|\cos^{-1}(A(\xi_1))- \cos^{-1}(A(\xi_2))|$, and
\begin{align*}
|\cos^{-1}(A(\xi_1))- \cos^{-1}(A(\xi_2))|&= |{\rm angle
}(e_{\nu_{k}-1}(\xi_1),v(\xi_1))-
{\rm angle}(e_{\nu_{k}-1}(\xi_2),v(\xi_2))\\
 &\leq{\rm angle}(v(\xi_1),v(\xi_2))+{\rm angle}
(e_{\nu_{k}-1}(\xi_1), e_{\nu_{k}-1}(\xi_2))\\&\leq
2C_2C_3^{\nu_{k}}|\xi_1-\xi_2|.
\end{align*}
The first factor in the second line is bounded by Lemma
\ref{liplem}. The second factor is bounded by Lemma \ref{fo}. In
the same way,  we have $B_i=\langle
f_{\nu_{k}-1}(\xi_i),v(\xi_i)\rangle =\cos\left({\rm
angle}(f_{\nu_{k}-1}(\xi_i),v(\xi_i))\right)$ and ${\rm
angle}(f_{\nu_{k}-1}(\xi_i),v(\xi_i))\in[0, \pi]$, which is
close to $\pi/2$. Then
\begin{align*}
|\cos^{-1}(B(\xi_1))- \cos^{-1}(B(\xi_2))|&= |{\rm angle
}(f_{\nu_{k}-1}(\xi_1),v(\xi_1))-
{\rm angle}(f_{\nu_{k}-1}(\xi_2),v(\xi_2))|\\
 &\leq{\rm angle}(v(\xi_1),v(\xi_2))+{\rm angle}(f_{\nu_{k}-1}
 (\xi_1),
f_{\nu_{k}-1}(\xi_2))\\&\leq 2C_2C_3^{\nu_{k}}|\xi_1-\xi_2|.
\end{align*}
For the last inequality we have used the orthogonality of
$f_{\nu_{k}-1}$ to $e_{\nu_{k}-1}$.
\end{proof}

\begin{lemma}\label{liplem6}
There is a $C^1$ vector field $\phi_0$ on
$f^{\nu_0+\cdots+\nu_{k-1}+1} Q(R_{i_0\cdots i_k})$ which is
tangent to the unstable sides of it, with $\Vert\phi_0\Vert\leq 2$
and $\Vert D\phi_0\Vert\leq 4C_2C_3^{2\nu_{k}}$.
\end{lemma}

\begin{proof}
Let $\zeta^{(1)}, \zeta^{(2)}$ denote the critical points on the
unstable sides of $f^{\nu_0+\cdots+\nu_{k-1}} Q(R_{i_0\cdots
i_{k-1}})$. We introduce a nearly orthogonal coordinate on the
rectangle which trivialize $\Gamma(f\zeta)$ and $\Gamma(f\zeta')$.
Namely, it is a $C^1$ coordinate $(\hat x,\hat y)$ on
$f^{\nu_0+\cdots+\nu_{k-1}+1}
 Q(R_{i_0\cdots i_{k-1}})$ such that:
\smallskip

(i) $9/10\leq\Vert\partial_{\hat x}\Vert\leq 10/9$,
$\Vert\partial_{\hat y}\Vert=1$, $\langle\partial_{\hat
x},\partial_{\hat y}\rangle=0$, $\langle\partial_{\hat
y},t(f\zeta_{1})\rangle=1$;

(ii) $\Gamma(f\zeta^{(1)})=\{\hat x=0\}$,
$\Gamma(f\zeta^{(2)})=\{\hat x=c\}$, where $c$ is a constant.
\smallskip

\noindent It is possible to choose such a coordinate, by the
properties of long stable leaves and (b) Lemma \ref{leaf}. Let
$T\colon (x,y)\to(\hat x,\hat y)$ denote the coordinate
transformation.

With respect to $(\hat x,\hat y)$-coordinate, we represent the
unstable sides of $f^{\nu_0+\cdots+\nu_{k-1}+1}
 Q(R_{i_0\cdots i_{k}})$ by
graphs of functions $\gamma_1, \gamma_2$, $\gamma_1(\hat
x)<\gamma_2(\hat x)$. For all $\xi$ in the unstable sides the
rectangle, let \begin{equation}\label{suv1} (\gamma_2(\hat
x)-\gamma_1(\hat x))\cdot v(\xi)=\tilde A(\xi)e_{\nu_{k}-1}(\xi) +
\tilde B(\xi)f_{\nu_{k}-1}(\xi),\end{equation} where $v(\xi)$ is
the one in (\ref{va}) and $T(\xi)=(\hat x,\hat y)$. In what
follows, we shall extend $\tilde A,\tilde B$ to $C^1$ functions on
the the entire $f^{\nu_0+\cdots+\nu_{k-1}+1} Q(R_{i_0\cdots
i_k})$, in such a way that $\max(\Vert \tilde A \Vert,\Vert \tilde
B\Vert)\leq 1$ and $\max(\Vert D\tilde A\Vert,\Vert D\tilde
B\Vert)\leq 3C_2C_3^{2\nu_{k}}$. For all $z$ in the rectangle,
define
\begin{equation}\label{suv2} \phi_0(z)=\tilde
A(z)e_{\nu_{k}-1}(z)+\tilde B(z)f_{\nu_{k}-1}(z).\end{equation}
Since $\Vert De_{\nu_{k-1}}\Vert$, $\Vert Df_{\nu_{k-1}}\Vert$ are
bounded by Lemma \ref{fo}, this yields the desired inequality.

To simplify notation, write $A$ for $A\circ T^{-1}$, and the same
for $B$, $\tilde A$, $\tilde B$. On the assumption that both
$\gamma_1(\hat x)$ and $\gamma_2(\hat x)$ make sense, we extend
$\tilde A$ affinely along the $\hat y$-direction. In other words,
for $\hat y\in[\gamma_1(\hat x),\gamma_2(\hat x)]$, define
\begin{equation}\label{A}\tilde A(\hat x,\hat y)=
\tilde A(\hat x,\gamma_1(\hat x))+ (\hat y-\gamma_1(\hat
x))\left(A(\hat x,\gamma_2(\hat x))-A(\hat x,\gamma_1(\hat
x))\right).\end{equation} In the same way, we extend $\tilde B$
affinely along the $\hat y$-direction. If, for instance,
$\gamma_1(\hat x)$ makes sense and $\gamma_2(\hat x)$ does not, we
enlarge the domain of definition of $\gamma_2$ so that
$\gamma_2(\hat x)$ makes sense. It is possible to show, using
 the long stable leaf of order $\nu_k-1$ through $\gamma_1(\hat x)$,
 that $\gamma_2(\hat x)$
 is sufficiently close to the unstable sides of
the rectangle, so that all the preceding arguments go through.

The definition gives $\max(\Vert \tilde A\Vert,\Vert \tilde
B\Vert)\leq \gamma_2(\hat x)-\gamma_1(\hat x)\ll1$. Lemma
\ref{eqlip0} and the choice of $(\hat x,\hat y)$-coordinate give
$\max(\Vert
\partial_{\hat y} \tilde A\Vert,\Vert
\partial_{\hat y} \tilde B\Vert)\leq 3C_2C_3^{\nu_k}$.
To evaluate the norms of $\hat x$-derivatives, we assume that
$\zeta^{(\sigma)}$ and $f^{-1}\gamma_\sigma(\hat x)$ belong to the
same unstable side, $\sigma=1,2$. Recall the symbolic coding
$i_{k-1}=(\rho_{k-1},\epsilon_{k-1},\Delta_{k-1},p_{k-1},\nu_{k-1})$.
In the case $\epsilon_{k-1}=+$ or $-$,
\begin{equation}\label{A1} \left|\frac{d\gamma_\sigma}{d\hat x}
(\hat x)\right|\leq \frac{C\sqrt{b}}{|f^{-1}(\hat
x,\gamma_\sigma(\hat x))-\zeta^{(\sigma)}|}\leq
e^{\frac{3p_{k-1}}{\log C_0}} \leq e^{\frac{3\nu_{k}}{\log
C_0}}.\end{equation} In the case $\epsilon_{k-1}=0$,
\begin{equation}\label{A4} \left|\frac{d\gamma_\sigma}{d\hat x}
(\hat x)\right|\leq \frac{C\sqrt{b}}{|f^{-1}(\hat
x,\gamma_\sigma(\hat x))-\zeta_{\rho_{k-1}}|}\leq
e^{\frac{3p_{k-1}}{\log C_0}} \leq e^{\frac{3\nu_{k}}{\log
C_0}}.\end{equation} In either of the two cases, Sublemma
\ref{eqlip0} gives
\begin{equation}\label{A2}\left|\frac{dA}{d\hat x}
(\hat x,\gamma_\sigma(\hat x))\right|\leq
3C_2C_3^{\nu_{k}}\left|\frac{d\gamma_\sigma}{d\hat x}(\hat
x)\right|\leq 3C_2C_3^{\nu_k}e^{\frac{3\nu_{k}}{\log
C_0}}.\end{equation}  As $\tilde A(\hat x,\gamma_\sigma(\hat
x))=(\gamma_2(\hat x)-\gamma_1(\hat x) )A(\hat
x,\gamma_\sigma(\hat x))$,
\begin{equation}\label{A3}
\left|\frac{d\tilde A}{d\hat x}(\hat x,\gamma_\sigma(\hat
x))\right|\leq C_2C_3^{\nu_k}e^{\frac{3\nu_{k}}{\log
C_0}}.\end{equation}
Differentiating (\ref{A}) with $\hat x$ and then using (\ref{A1})
(\ref{A4}) (\ref{A2}) (\ref{A3}), we obtain $\Vert\partial_{\hat
x} \tilde A\Vert\leq C_2C_3^{2\nu_k}$. In the same way we obtain
the desired
upper estimate of $\Vert\partial_{\hat x} \tilde B\Vert$. 
Transforming all these derivative estimates back to the original
$(x,y)$-coordinate, we obtain the desired estimates.
\end{proof}

We now introduce the projectivization $f_*$ of $Df$, given by
$f_*(\xi,v)=Df(\xi)v/\Vert Df(\xi)v\Vert,$ and define vector
fields $\phi_j$ on $f^{\nu_0+\cdots+\nu_{k-1}+j+1} Q(R_{i_0\cdots
i_k})$ for $1\leq j<\nu_{k},$ by push-forward under $f_*$:
$$\phi_j(z)=f_*(f^{-1}z,\phi_{j-1}(f^{-1}z)).$$

If $\epsilon_{k-1}=+$ or $-$, then for all $\xi,\eta$ in the
unstable side of $f^{\nu_0+\cdots+\nu_k}Q(R_{i_0\cdots i_k})$,
${\rm angle}(u(\xi),u(\eta))={\rm angle}(\phi_{\nu_k-1}(\xi),
\phi_{\nu_k-1}(\eta))$ holds. If $\epsilon_{k-1}=0$, then ${\rm
angle}(u(\xi),u(\eta))<{\rm angle}(\phi_{\nu_k-1}(\xi),
\phi_{\nu_k-1}(\eta))$ holds. Hence, \ref{lipeq}) for $j=k$ is a
direct consequence of the next
\begin{lemma}\label{viana}
For all $z\in f^{\nu_0+\cdots+\nu_k}Q(R_{i_0\cdots i_k})$, $\Vert
D\phi_{\nu_{k}-1}(z)\Vert\leq C_2C_3^{3\nu_{k}}.$
\end{lemma}
\begin{proof} The following estimates, proved in Appendix A.3, are used:
\begin{equation}\label{pro1}
\left|\partial_{v} f_*(\xi,v)\right|\leq
2\frac{|\det Df(\xi)|}{\Vert
Df(\xi)v\Vert^2}.
\end{equation}
\begin{equation}\label{pro2}
|\partial_\xi f_*(\xi,v)| \leq \frac{\Vert D^2f(\xi)\Vert\Vert
v\Vert}{\Vert Df(\xi)v\Vert}.
\end{equation}

Differentiating the formula of $\phi_j$ and using the result
recursively we get
\begin{align*}D\phi_{\nu_{k}-1}(z)=
&\sum_{i=1}^{\nu_{k}-1}\partial_v f_*^{i-1}(f^{-i+1}z,
\phi_{\nu_{k}-i})\partial_{\xi}f_*(f^{-i}z,\phi_{\nu_{k}-1-i})Df^{-i}(z)\\&+
\partial_v f_*^{\nu_{k}-1}(f^{-\nu_{k}+1}z,
\phi_{0})D\phi_{0}(f^{-\nu_{k}+1}z)Df^{-\nu_{k}+1}(z),\end{align*}
where $\phi_{\nu_{k}-1-i}$ means $\phi_{\nu_{k}-1-i}(f^{-i}z)$.
By (\ref{pro1}), for every $1\leq i<\nu_{k},$
\begin{eqnarray*}\label{ve}
\Vert\partial_vf_*^{i-1}(f^{-i+1}z,\phi_{\nu_{k}-i})\Vert &\leq&
2^{i-1}\frac{|\det Df^{i-1}(f^{-i+1}z)|} {\Vert
Df^{i-1}(f^{-i+1}z)\phi_{\nu_{k}-i}\Vert^2}\\
&=& \frac{2^{i-1}} {\Vert
Df^{i-1}(f^{-i+1}z)\phi_{\nu_{k}-i}\Vert^2}\frac{\Vert
Df^{i-1}(f^{-i+1}z)\Vert}{\Vert Df^{-i+1}(z)\Vert}.\end{eqnarray*}
(\ref{pro2}) gives
\begin{align*}\Vert\partial_\xi f_*(f^{-i}z,\phi_{\nu_{k}-1-i})
Df^{-i}(z)\Vert&\leq \Vert\partial_\xi
f_*(f^{-i}z,\phi_{\nu_{k}-1-i})
Df^{-1}(f^{-i+1}z)\Vert\Vert Df^{-i+1}(z)\Vert\\
&\leq\frac{C\Vert Df^{-1}(f^{-i+1}z)\Vert\Vert
Df^{-i+1}(z)\Vert}{\Vert Df(f^{-i}z)\Vert}\leq CC_1^{-2}\Vert
Df^{-i+1}(z)\Vert.\end{align*}
Replacing all these in the above equality,
$$\Vert
D\phi_{\nu_{k}-1}(z)\Vert\leq\sum_{i=1}^{\nu_{k}-1}2^{i-1}CC_1^{-2}
\frac{\Vert Df^{i-1}(f^{-i+1}z)\Vert}{\Vert Df^{i-1}(f^{-i+1}z)
\phi_{\nu_{k+1}-i}\Vert^2}+\frac{\Vert
Df^{\nu_k-1}(f^{-\nu_k+1}z)\Vert} {\Vert
Df^{\nu_k-1}(f^{-\nu_{k}+1}z) \phi_{0}\Vert^2}\Vert
D\phi_0(f^{-\nu_{k}+1}z)\Vert.$$ To evaluate the denominators of
the fractions, we need
\begin{lemma}\label{expa}
For all $\xi$ in the unstable sides of
$f^{\nu_0+\cdots+\nu_{k-1}+1} Q(R_{i_0\cdots i_k})$ and
$0\leq j<\nu_k,$ $\Vert \phi_{\nu_{k}-1}(\xi)\Vert\geq
C\delta\Vert \phi_{j}(\xi)\Vert.$
\end{lemma}
\begin{proof}
Let $\zeta$ denote the critical point on the same unstable side of
$f^{\nu_{0}+\cdots+\nu_{k-1}}Q(R_{i_0\cdots i_{k-1}})$ as that of
$f^{-1}\xi$. Let $q$ denote the fold period for $f^{-1}\xi$. In
view of Proposition \ref{recovery0} and the bounded distortion,
for $q\leq j<\nu_k$ we have
$\frac{\Vert\phi_{\nu_{k}-1}(\xi)\Vert}{\Vert\phi_j(\xi)\Vert}
\geq C\frac{\Vert w_{\nu_{k}}(f^{-1}\xi)\Vert}{\Vert
w_j(f^{-1}\xi) \Vert}\geq C\delta,$ and for $0\leq j< q$,
$\Vert\phi_{q}\Vert>\Vert\phi_{j}(\xi)\Vert$.
\end{proof}
Lemma \ref{liplem6} and Lemma \ref{expa} give
$$\Vert D\phi_{\nu_{k}-1}(z)\Vert\leq \sum_{i=1}^{\nu_k-1}
C\delta^{-2}10^i+C\delta^{-2}5^{\nu_k} C_2C_3^{2\nu_k}\leq
C_2C_3^{3\nu_k}.$$ The last inequality holds for sufficiently
large $k_0$.
\end{proof}

\subsection{Area distortion bounds.}\label{aread}
Proposition \ref{lipschitz} and the next area distortion bounds
together allow us to estimate the Lebesgue measure of the set in
question.
\begin{prop}\label{metric2}
For every $k\geq1$ and all $\xi_1,\xi_2\in
f^{\nu_0+\cdots+\nu_{k-1}} Q(R_{i_0\cdots i_{k}})$,
$$\frac{|\det Df^{\nu_{k}}(\xi_1)|}{|\det Df^{\nu_{k}}(\xi_2)|}\leq
e^{C_1^{-1}}.$$
\end{prop}

\begin{proof}
Since $\Vert D\log|\det Df|\Vert\leq CbC_1^{-1}$, it suffices to
show \begin{equation}\label{bc}
\sum_{i=0}^{\nu_{k}-1}|f^i\xi_1-f^i\xi_2|\leq C.\end{equation}

Let $\gamma$ denote one of the unstable sides of
$f^{\nu_0+\cdots+\nu_{k-1}+1} Q(R_{i_0\cdots i_{k-1}})$. Let
$\eta_\sigma$ denote the point of intersection between
$\Gamma_{\nu_k-1}(\xi_\sigma)$ and $\gamma$ $(\sigma=1,2)$. Let
$0\leq i<\nu_k$.
 If
$\eta_1$ and $\eta_2$ are on the unstable sides of
$f^{\nu_0+\cdots+\nu_{k-1}+1}Q(R_{i_0\cdots i_{k}})$, Lemma
\ref{expa} implies $|f^i\eta_1-f^i\eta_2|\leq
C\delta^{-1}|f^{\nu_{k}-1}\eta_1- f^{\nu_{k}-1}\eta_2| \leq
C\delta^{\frac{\nu_{k}}{10}-1}.$ On the other hand, the
contraction along the long stable leaves gives
$|f^i(f\xi_\sigma)-f^i\eta_\sigma|\leq
(Cb)^{\frac{i}{2}}|f\xi_{\sigma}-\eta_{\sigma}|\leq
(Cb)^{\frac{i+1}{2}}$. It follows that $|f^i\xi_1-f^i\xi_2| \leq
C\delta^{\frac{\nu_{k}}{10}-1}$. Summing this over all $0\leq
i<\nu_k$ yields (\ref{bc}). Even if $\eta_1$ or $\eta_2$ is not on
the unstable side of $f^{\nu_0+\cdots+\nu_{k-1}+1}Q(R_{i_0\cdots
i_{k}})$, the constants in Lemma \ref{expa} are not significantly
affected because $f\xi_1,f\xi_2\in
f^{\nu_0+\cdots+\nu_{k-1}+1}Q(R_{i_0\cdots i_{k}})$ holds. Hence
we obtain the same conclusion.
\end{proof}

\subsection{Proof of Proposition \ref{zeromeasure}}\label{ch}
In what follows, we assume $k\geq k_0$ is large so that
$C_2C_3^{3\nu_k}\leq C_3^{4\nu_k}$. Denote by $\gamma_1$ and
$\gamma_2$ the two unstable sides of
$f^{\nu_0+\cdots+\nu_k}Q(R_{i_0\cdots i_{k}})$, and consider their
graph representations $\gamma_1=\{(x,\gamma_1(x))\}$,
$\gamma_2=\{(x,\gamma_2(x))\}$.
Let $L(x)=|\gamma_1(x)-\gamma_2(x)|$. 
Proposition \ref{lipschitz} and the Gronwall inequality give
 $L(x)/L(y)\leq e^{C_3^{4\nu_k}|x-y|}$ for all
$x$, $y$. As $|x-y|\leq C\delta^{\frac{\nu_k}{10}}$,
$L(x)/L(y)\leq 2$ holds.

Let $S_{\nu_k,1},S_{\nu_k,2},\cdots$ denote the components of
$R_{i_0\cdots i_{k-1}}(\nu_k)$, the total number of which is
clearly $\leq 2^{\nu_k}$. For each $S_{\nu_k,m}$, the above
estimate and Proposition \ref{rect2} give
$$\frac{|\mathcal B_0^{(\nu_k)}\cap
f^{\nu_0+\cdots+\nu_k}S_{\nu_k,m}|}
{|f^{\nu_0+\cdots+\nu_k}Q(R_{i_0\cdots i_{k}})|}\leq
2\delta^{\frac{\nu_k}{5}}.$$
Proposition \ref{metric2} gives $\frac{|\det
Df^{\nu_0+\cdots+\nu_k}(\xi_1)|} {|\det
Df^{\nu_0+\cdots+\nu_k}(\xi_2)|} \leq e^{C_1^{-1}k}$ for all
$\xi_1,\xi_2\in Q(R_{i_0\cdots i_{k}})$. Hence
$$\frac{|f^{-(\nu_0+\cdots+\nu_k)}(\mathcal B_0^{(\nu_k)})\cap
S_{\nu_k,m}|} {|R_{i_0\cdots
i_{k-1}}|}\leq\frac{|f^{-(\nu_0+\cdots+\nu_k)}(\mathcal
B_0^{(\nu_k)})\cap S_{\nu_k,m}|} {|Q(R_{i_0\cdots i_{k}})|}\leq
2e^{C_1^{-1}k}\delta^{\frac{\nu_k}{5}}.$$ The first inequality
follows from the obvious inclusion $Q(R_{i_0\cdots,i_{k}})\subset
R_{i_0\cdots i_{k-1}}.$ Summing this over all components, and then
for all feasible $\nu_k$,
$$\sum_{\nu_k}\sum_{m}
\frac{|S_{\nu_k,m}|}{|R_{i_0\cdots i_{k-1}}|} \leq
2\sum_{\nu_k\geq 4^kk_0}
2^{\nu_k}e^{C_1^{-1}k}\delta^{\frac{\nu_k}{5}}\leq
e^{C_1^{-1}k}\delta^{\frac{4^kk_0}{6}}.$$  Therefore
\begin{align*}|\Omega_k|=\sum_{(i_0,\cdots,i_k)}|R_{i_0\cdots i_k}|&
=\sum_{(i_0,\cdots,i_{k-1})} |R_{i_0\cdots i_{k-1}}|\sum_{\nu_k,m}
\frac{|S_{\nu_k,m}|} {|R_{i_0\cdots i_{k-1}}|}\leq
e^{C_1^{-1}k}\delta^{\frac{4^kk_0}{6}}|\Omega_{k-1}|.\end{align*}
 This completes the proof of Proposition \ref{zeromeasure}.

\subsection{Transitivity}\label{tranasitive}
We show $f$ is transitive on $K$. Let $H(Q)$ denote the closure of
transverse homoclinic points of $Q$. Then $H(Q)\subset K$ holds.
It suffices to show the reverse inclusion. Let $z\in K$, and let
$U$ be an open set containing $z$. Since the Lebesgue measure of
$U\cap K^+$ is zero, $U$ intersects $W^s(Q)$. It follows that
$W^s(Q)$ is dense in $K$. By Inclination Lemma, $z$ is accumulated
by transverse homoclinic points of $Q$. Hence $K\subset H(Q)$
holds.


\section*{Appendix}

\subsection*{A.1. Proof of Lemma \ref{c2}.}
From the next sublemma, it follows that $f^nG\cap I(\delta)$ is
made up of $C^2(b)$-curves. This yields the conclusion of Lemma
\ref{c2}. For a proof of it, the correct order for the reader is
to go over Sect.\ref{rec}, \ref{cripa} first. For $z\in W^u(Q)$,
let $t(z)$ denote any unit vector tangent to $W^u(Q)$ at $z.$
\begin{sublemma}\label{kuya}
Let $n\geq0$ and $z\in G$. If $f^iz\notin I(\delta)$ for
$0\leq i\leq n$, then there exists a sequence $0\leq
n_1<n_1+p_1\leq n_2<n_2+p_2\leq n_3<\cdots\leq n$ of integers such
that:
\smallskip

\noindent{\rm (a)} $f^{n_i}z\in I(\delta)$;

\noindent{\rm (b)} $f^jz\in \{(x,y)\in\mathbb R^2\colon |x|\geq
9/10\}$ for $n_i+1\leq j\leq n_{i}+p_i$;

\noindent{\rm (c)} $\Vert Df^{n_i}t(z)\Vert\geq(\delta/10)\Vert
Df^{j}t(z)\Vert$ for $0\leq j<n_i$.
\end{sublemma}

\noindent{\it Proof of Sublemma \ref{kuya}.} The argument is an
induction on $n$.
 For $n=0$, the assertions are direct consequences of
the definition of $G$. Suppose that they hold for $n=k$. From the
fact that the orbits of all critical points on $W^u(Q)$ are out of
$R_0$, all the estimates in Proposition \ref{recovery0} remain to
hold for them. This allows us to decompose the orbit of $z$ into
bound and free segments as follows: $n_i\leq k$ is a return time
to $I(\delta)$. By the assumption of the induction, there exists a
$C^2(b)$-curve in $W^u(Q)$ tangent to $Df^{n_i}t(z)$ stretching
across $I(\delta)$. Let $p_i$ denote the bound period, given by
the critical point on the $C^2(b)$-curve and an associated
critical partition in Sect.\ref{cripa}. Let $n_{i+1}$ denote the
next return time to $I(\delta)$. By (c) in Proposition
\ref{recovery0}, bound parts of $f^{k+1}G$ do not return to
$I(\delta)$. This recovers all the assertions for $n=k+1$. \qed

\subsection*{A.2. Proof of Lemma \ref{boun}}
First, for $M\leq  k<20n-1,$ we show
\begin{equation}\label{compte}
e^{-3\alpha k}D_k(\zeta)\leq D_{k+1}(\zeta)\leq
e^{-3\alpha}D_k(\zeta).
\end{equation}
To this end, let
$d_\ell(i)=\min_{j\in[i,\ell+1]}\|w_{j}(\zeta)\|^2
\|w_{i}(\zeta)\|^{-3}$. Then
\begin{align*}
\frac{D_{k+1}(\zeta)}{D_k(\zeta)}&
=e^{-3\alpha}\frac{\min_{i\in[1,k+1]}d_{k+1}(i)}{
\min_{i\in[1,k]}d_{k}(i)}\leq e^{-3\alpha}\frac
{\min_{i\in[1,k]}d_{k+1}(i)}{\min_{i\in[1,k]}d_{k}(i)}\leq e^{-3\alpha},
\end{align*}
and the second inequality holds.

(G2) gives
$\|w_{k+2}(\zeta)\|\geq e^{-2\alpha (k+1)}\| w_{k+1}(\zeta)\|$,
and thus for $1\leq i\leq k,$
\begin{align*}
d_{k+1}(i)&=\min\left\{d_{k}(i),
\|w_{k+2}(\zeta)\|^2\|w_{i}(\zeta)\|^{-3}\right\} \geq e^{-4\alpha
k}d_{k}(i)\geq e^{-4\alpha (k+1)}D_k(\zeta).
\end{align*}
 Using
$\|w_{k+1}(\zeta)\|\leq C_0\|w_{k}(\zeta)\|$ and
$\|w_{j}(\zeta)\|\geq e^{-2\alpha k}\| w_{j-1}(\zeta)\|$ from
(G2),
\begin{equation*}d_{k+1}(k+1)\geq C_0^{-3}e^{-4\alpha (k+1) }d_{k}(k).\end{equation*}

These two inequalities yield the first inequality in
(\ref{compte}).

We now show (a) (b). From (f) Proposition \ref{recovery0},
[\cite{T1} Lemma 2.3] and the fact that $\gamma$ is $C^2(b)$,
$f^{\chi(k)}\gamma_{k,s}$ is $C^2(b)$. Using (\ref{compte}),
\begin{align*}
{\rm length}(f^{\chi(k)}\gamma_{k,s})&\geq
Ce^{-3\alpha k}\|w_{\chi(k)}(\zeta)\|(D_k(\zeta)-D_{k+1}(\zeta))\geq
Ce^{-3\alpha k}\|w_{\chi(k)}(\zeta)\|D_k(\zeta)(1-e^{-3\alpha})\\&\geq
Ce^{-3\alpha k}\|w_{k}(\zeta)\|D_k(\zeta)C_0^{-\sqrt{\alpha}k}
(1-e^{-3\alpha}) \geq e^{-4\alpha k}.
 \end{align*}
The third inequality follows from $k-\chi(k)\leq\sqrt{\alpha} k$
in (G2). Using $D_{k+1}(\zeta)\geq C_0^{-3k}$ and $ {\rm
length}(\gamma_k)\leq Ce^{2\alpha k} D_{k+1}^{\frac{1}{2}}(\zeta)$
which follows from (\ref{compte}),
$${\rm length}(\gamma_{k,s})\leq e^{-3\alpha k}\cdot{\rm
length}(\gamma_k) \leq D_{k+1}^{\frac{1}{2}+\frac{\alpha}{3\log
C_0}}(\zeta)\leq d(\gamma_{k,s},\zeta)^{1+C\alpha}.$$ Here,
$d(\gamma_{k,s},\zeta)$ denotes the distance between
$\gamma_{k,s}$ and $\zeta$. Now (b) follows from
 [\cite{T1} Lemma 5.12]. \qed

\subsection*{A.3. Derivative estimates of projectivization}
We prove (\ref{pro1}) (\ref{pro2}).
Let $v^\bot$ denotes any unit
vector orthogonal to $v$.
Then
\begin{align*}\left|\partial_{v} f_*(\xi,v)\right| &=\lim_{\Delta\theta\to0}\left\Vert\frac{1}{\Delta\theta}
\left(\frac{Df(\xi)(v+\Delta\theta v^\bot)}{\Vert
Df(\xi)(v+\Delta\theta v^\bot)\Vert}- \frac{Df(\xi)v}{\Vert
Df(\xi)v\Vert}\right)\right\Vert\\
&\leq\frac{\Vert Df(\xi)v^\bot\Vert}{\Vert Df(\xi)v\Vert}+
\lim_{\Delta\theta\to0}\left\Vert\frac{1}{\Delta\theta}\frac{\Vert
Df(\xi)v\Vert-\Vert Df(\xi)(v+\Delta\theta v^\bot)\Vert}{\Vert
Df(\xi)v\Vert}\right\Vert\\
&\leq 2\frac{\Vert Df(\xi)v^\bot\Vert}{\Vert
Df(\xi)v\Vert}=2\frac{|\det Df(\xi)|}{\Vert
Df(\xi)v\Vert^2}.\end{align*}

Let $\xi=(x,y)$.
Writing $\xi_x=\xi+(\Delta
x,0)$ we have
\begin{align*}\left|\partial_x f_*(\xi,v)\right| &=\lim_{\Delta x \to0}\left\Vert\frac{1}{\Delta x}
\left(\frac{Df(\xi_x)v}{\Vert Df(\xi_x)v\Vert}-
\frac{Df(\xi)v}{\Vert
Df(\xi)v\Vert}\right)\right\Vert\\
&=\lim_{\Delta x \to0}\left\Vert\frac{1}{\Delta x}
\left(\frac{Df(\xi_x)v-Df(\xi)v}{\Vert Df(\xi)v\Vert}- \frac{\Vert
Df(\xi_x)v\Vert- \Vert Df(\xi)v\Vert}{\Vert Df(\xi)v\Vert\Vert
Df(\xi_x)v\Vert}Df(\xi_x)v\right)\right\Vert\\
&\leq 2\lim_{\Delta x \to0}\left\Vert\frac{1}{\Delta x}
\frac{(Df(\xi_x)-Df(\xi))v}{\Vert
Df(\xi)v\Vert}\right\Vert=\frac{2}{\Vert
Df(\xi)v\Vert}\left\Vert
\left(\frac{\partial}{\partial x }Df(\xi)\right)v\right\Vert.
\end{align*}
In the same way we get
$$\left|\partial_y F(\xi,v)\right|\leq
\frac{2}{\Vert
Df(\xi)v\Vert}\left\Vert
\left(\frac{\partial}{\partial y}Df(\xi)\right)v\right\Vert.$$




\bibliographystyle{amsplain}

\begin{thebibliography}{10}
\bibitem{AY83}
K. T. Alligood and J. A. Yorke, Cascades of period doubling
bifurcations: a prerequisite for horseshoes, {\it Bull. A.M.S.}
{\bf 9} (1983), 319-322.





\bibitem{BS1} E. Bedford and J. Smillie, Real polynomial
difeomorphisms with maximal entropy: tangencies {\it Ann. of
Math.} (2) 160 (2004), 1--25.

\bibitem{BS2} E. Bedford and J. Smillie, Real polynomial diffeomorphisms
with maximal entropy: II. small Jacobian, {\it Ergod. Th. $\&$ Dynam. Sys.}
{\bf 26} (2006) 1259--1283.

\bibitem{BC85} M. Benedicks and L. Carleson,
On iterations of $1-ax\sp 2$ on $(-1,1)$. {\it Ann. of Math. (2)}
{\bf 122} (1985), no. 1, 1--25.

\bibitem{BC91} M. Benedicks and L. Carleson, The dynamics of the
H\'enon map. {\it Ann of Math. (2)} 133 (1991) 73--169.






\bibitem{BV01} M. Benedicks and M. Viana, Solution of the basin problem
for H\'enon-like attractors. {\it Invent math.} 143 (2001)
375--434.






\bibitem{BY93} M. Benedicks and L-S. Young, Sinai-Bowen-Ruelle measures
for certain H\'enon maps. {\it Invent Math.} {\bf 112}, 541--576,
(1993)





\bibitem{CLR} Y. Cao, S. Luzzatto and I. Rios, The boundary of
hyperbolicity for H\'enon-like families. {\it Ergod. Th. $\&$
Dynam. Sys.} {\bf 28} (4) (2008), 1049--1080.





\bibitem{DN} R. Devaney and Z. Nitecki, Shift automorphisms in the
H\'enon mapping. {\it Comm. Math. Phys.} 67 (1979) 137--146.



\bibitem{GS} N. Gavrilov and L. Silnikov,
On the three dimensional dynamical systems close to a system with
a structurally unstable homoclinic curve. I. Math. USSR Sbornik
{\bf 17}, 467-485 (1972); II. Math. USSR Sbornik {\bf 19}, 139-156
(1973)


\bibitem{Jak81} M. Jakobson,
Absolutely continuous invariant measures for one-parameter
families of one-dimensional maps. {\it Comm. Math. Phys.} {\bf 81}
(1981), 39--88.








\bibitem{MV93} L. Mora and M. Viana, Abundance of strange attractors.
{\it Acta Math.} 171 (1993) 1--71.




\bibitem{NP}
S. Newhouse and J. Palis, Cycles and bifurcation theory, {\it
Ast\'erisque}, {\bf 31} (1976), 44--140.


\bibitem{PT0}
J. Palis and F. Takens, Cycles and measure of bifurcation sets for
two-dimensional diffeomorphisms, {\it Invent. Math.} {\bf 82}
397--422

\bibitem{PT1}
J. Palis and F. Takens, Hyperbolicity and the creation of
homoclinic orbits, {\it Ann. Math.} 125 (1987) 337-374.

\bibitem{PT2}
J. Palis and F. Takens, {\it Hyperbolicity \& sensitive chaotic
dynamics at homoclinic bifurcations.} Cambridge Studies in
Advanced Mathematics {\bf 35}. Cambridge University Press, 1993.




\bibitem{PY1} J. Palis and J-C. Yoccoz, Homoclinic tangencies for
hyperbolic sets of large Hausdorff dimension, {\it Acta Math.} 172
(1994) 91--136.

\bibitem{PY2} J. Palis and J-C. Yoccoz,
Fers \`a cheval non uniform\'ement hyperboliques engendr\'es par
une bifurcation homocline et densit\'e nulle des attracteurs, {\it
C. R. Acad. Sci. Paris S\'er. I Math.} 333 (9) (2001) 867-871.


\bibitem{PY3} J. Palis and J-C. Yoccoz,
Non-uniformly hyperbolic horseshoes arising from bifurcations of
Poincar\'e heteroclinic cycles, {\it Publ. Math. Inst. Hautes \'Etudes
    Sci,} No. 110 (2009) 1--217.




\bibitem{R01}
I. Rios, Unfolding homoclinic tangencies inside horseshoes:
hyperbolicity, fractal dimensions and persistent tangencies, {\it
Nonlinearity} 14 (2001) 431--462.




\bibitem{T1}
H. Takahasi, Abundance of nonuniform hyperbolicity in bifurcations
of surface endomorphisms. to appear in {\it  Tokyo J. Math.}







\bibitem{T93a} M. Tsujii, A proof of
Benedicks-Carleson-Jakobson Theorem. {\it  Tokyo J. Math.} 16
(1993), no. 2, 295--310.

\bibitem{T93b} M. Tsujii,
Positive Lyapunov exponents in families of one-dimensional
dynamical systems. {\it Invent. Math.} {\bf 111} (1993), no. 1,
113--137


\bibitem{WY01} Q. Wang and L-S. Young, Strange attractors with one
direction of instability. {\it Comm. Math. Phys.} {\bf 218}
(2001), 1--97.



\end{thebibliography}

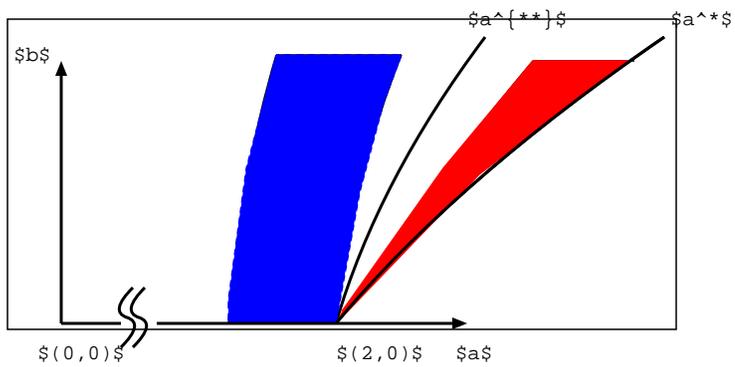
\end{document}